\documentclass[11pt]{article}

\usepackage{layout}
\usepackage{url}
\usepackage{color}
\usepackage{graphicx}
\usepackage{latexsym}
\usepackage{fullpage}
\usepackage{hyperref}
\usepackage{enumitem}
\usepackage{enumerate}
\usepackage{booktabs}
\usepackage{multirow}
\usepackage{parskip}
\usepackage{cancel,soul}
\usepackage[T1]{fontenc}
\usepackage{amsmath,bm}
\usepackage{amssymb,mathtools,esint}
\usepackage[english]{babel}
\usepackage[textsize=tiny, textwidth=2cm]{todonotes}
\usepackage{subcaption}
\usepackage{csquotes}
\usepackage[left=2cm,right=2cm,top=2cm,bottom=2cm]{geometry}

\usepackage{algorithm}
\usepackage{algorithmicx}
\usepackage{algpseudocode}
\usepackage{algpseudocode}
\alglanguage{pseudocode}
\usepackage{amsmath,epsfig}

\newenvironment{proof}[1][Proof]{\textbf{#1.} }{\ \rule{0.5em}{0.5em}}

\newcommand{\eq}[1]{\begin{equation}\label{#1}}
\newcommand{\en}{\end{equation}}

\newtheorem{theorem}{Theorem}[section]
\newtheorem{corollary}[theorem]{Corollary}
\newtheorem{lemma}[theorem]{Lemma}
\newtheorem{proposition}[theorem]{Proposition}
\newtheorem{remark}[theorem]{Remark}

\newcommand{\ds}{\displaystyle}

\newcommand{\norm}[1]{\| #1 \|}
\renewcommand{\d}{{\rm d}}

\def\inv{^{-1}}
\newcommand{\up}[1]{^{\text#1}}

\newcommand{\N}{\mathbb{N}}

\newcommand{\R}{\mathbb{R}}

\title{Eigenvector-based acceleration strategies for gradient-type methods}

\author{Jean-Paul Chehab\thanks{LAMFA ({\small UMR} CNRS 7352),
    Universit\'e de Picardie Jules Verne, France ({\tt jean-paul.chehab@u-picardie.fr}).}
  \and Gaspard Kemlin\thanks{LAMFA ({\small UMR} CNRS 7352),
    Universit\'e de Picardie Jules Verne, France ({\tt gaspard.kemlin@u-picardie.fr}).} \and Marcos
  Raydan\thanks{Center for Mathematics and Applications (NovaMath),  NOVA University of Lisbon,
  Portugal ({\tt m.raydan@fct.unl.pt}).} \and Yousef
  Saad\thanks{Department of Computer Science and Engineering, University of
    Minnesota, USA ({\tt saad@umn.edu}).}}

\begin{document}

\date{}

\maketitle

\begin{abstract}
  Several strategies are described and analyzed to speed-up gradient-type
  methods when applied to the minimization of strictly convex quadratics
  and strictly convex functions. The proposed techniques focus on relaxing
  the traditional optimal step length associated with gradient methods,
  including the steepest descent (SD) and the minimal residual (MR) methods.
  Such a relaxation avoids the well-known negative zigzag effect and
  allows the iterates to move in the entire space which in turn implies that
  every so often the search direction approaches some eigenvector of the
  underlying Hessian matrix. The proposed speedups then rely on taking advantage
  of the properties of the Lanczos method once a search direction that
  approaches an eigenvector has been identified in order to accelerate the
  convergence towards the global minimizer. After analyzing the proposed
  strategies, we illustrate them on the global minimization of strictly convex
  functions and compare their performances with those of two well-established
  gradient-type methods:  ABBmin and LMSD.
\end{abstract}

{\small
  {\bf Keywords:} {Gradient-type methods, Lanczos method, convex optimization}

  {\bf  AMS Subject Classification}: 65F10, 65K05, 15A06, 90C25, 90C06}

\section{Introduction}\label{sec:intro}

Gradient-type methods are still widely used for solving the following
large-scale optimization problem \begin{equation} \min_{x\in \mathbb{R}^{n}}
  f(x), \label{genprb} \end{equation} where $f:\R^{n}\to \mathbb{R}$ is
continuously differentiable.  From a given initial guess $x_0\in \R^n$, they are
described by the following iterative scheme \begin{equation} x_{k+1}  =  x_k -
  \alpha_k g_k, \label{genscheme} \end{equation} where $g_k = \nabla f(x_k)$ is
the gradient at $x_k$, and $\alpha_k >0$ is the steplength determined by the
method.
Some of the advantages that motivate the current use of gradient-type methods
include: they are simple to understand and implement; they require relatively
low memory and can handle a large number of variables making them suitable for
big data and large-scale machine learning applications, see, e.g., \cite{bbkm,
  bkms, Bottou}; and they are often supported by strong theoretical convergence
results. For instance, under standard conditions, they are guaranteed to
converge to a local minimum or even a global minimum in the case of convex
functions \cite{Andrei22, bertsekas99, Grippo23}. However, it is also well known
that the most popular gradient-type methods, in which the step length is chosen
to be optimal in some sense,  may take many iterations to locate a solution as
their asymptotic behavior depends on the spectral properties of the Hessian
matrix of $f$, so they may struggle in the presence of ill-conditioning; see,
e.g., \cite{ddrt13, drtz18, gonzaga, zoumag} and references therein. For the
sake of completeness, we briefly recall the two most emblematic gradient-type
methods.

Cauchy’s steepest descent (SD) algorithm \cite{cauchy} is the oldest
gradient-type method for solving (\ref{genprb}), and its optimal choice of
steplength   $\alpha_k^\text{SD}$ is given by \begin{equation}
  \alpha_k^\text{SD} \coloneqq \arg\min_{\alpha}  f(x_k - \alpha g_k).  \label{alfaSD}
\end{equation} The poor practical behavior of (\ref{genscheme})-(\ref{alfaSD})
has been known for many decades. In the late 1950s, Akaike \cite{Akaike}
characterized the asymptotic behavior of the method for strictly convex
quadratic functions. Akaike established that, with the choice of steplength
(\ref{alfaSD}),  the search directions asymptotically alternate within the
two-dimensional subspace spanned by the eigenvectors corresponding to the
largest   and smallest  eigenvalues of the constant Hessian matrix of $f$.
Moreover, in that case, $g_{k+1}^Tg_k=0$ for all $k$. This combination of facts
causes SD to zig-zag on a 2-dimensional subspace as it approaches the minimizer,
and hence it can be very slow on ill-conditioned problems. We recommend
\cite{gonzaga} for a deeper understanding of this type of negative behavior of
the SD method.

Closely related to SD is the Minimal Gradient method; see, e.g., \cite{Dai03,
  DeAsmun14, Huang21} (which in the case of solving linear systems is better
known as the minimal residual (MR) method) . In the MR  method, the optimal
steplength $\alpha_k^\text{MR}$  minimizes the 2-norm of the gradient at the
next iterate \begin{equation} \alpha_k^\text{MR} \coloneqq \arg\min_{\alpha} \|\nabla
  f(x_k - \alpha g_k)\|. \label{alfaMG} \end{equation} Many of the asymptotic
results for SD carry over analogously to the MR method. It is widely accepted
that the MR method, given by (\ref{genscheme})-(\ref{alfaMG}), can also perform
poorly and has similar asymptotic behavior as the SD method, i.e., it
asymptotically zigzags in a two-dimensional subspace; see \cite{Huang21} for a
formal proof of these facts.

Consequently, a tempting  and perhaps surprising  option to avoid the zigzag
curse, which in turn produces extremely slow schemes, is to disturb the exact
optimal choice of steplength associated with the negative gradient direction. In
this regard, Raydan and Svaiter \cite{Raydan02} investigate the random choice of
a parameter in $(0,2)$ that multiplies the  optimal step length, to avoid the
zigzag effect. They study that  random mechanism, precisely  to stress out that
the poor behavior of the SD method is due to the optimal  choice of steplength
and not to the choice of the search direction, and indeed a significant
improvement on the behavior of the SD method is observed. It is worth noticing
that they provide theoretical guarantees, confirming that SD with random
relaxation converges to the unique global minimizer when dealing with
strictly convex quadratic functions. In the last two decades, several variants
and improved relaxation schemes have been proposed and analyzed that also avoid
the optimal step length and thus the zigzag curse; see, e.g., \cite{zoumag} for
a comprehensive review.  Recently, in \cite{MacD25}, much insight has been added
to the topic of relaxed steplength selection, and the convergence of a large
family of options has been studied.  We also refer to
\cite{kalousekSteepestDescentMethod2017} for a study of the steepest descent
algorithm with fully random steplengths.

In this paper, in order to accelerate the gradient-type methods
(\ref{genscheme}), we add understanding to the  relaxed techniques for breaking
the zigzagging pattern, and also propose and analyze several strategies to
speed-up gradient-type methods when applied to the minimization of strictly
convex quadratics and strictly convex functions. Once the zigzag curse is
avoided,  our speed-up schemes take advantage of the relationship between
\eqref{genscheme} and the shifted power method, which in turn implies that every
so often the residual approaches some eigenvector of the underlying Hessian
matrix, therefore providing ideal search directions. The most effective proposed
speedup scheme in this work then relies on taking advantage of the properties of the
Lanczos method once a search direction that approaches an eigenvector has been
identified. The idea of exploiting the alignment of the search direction with
eigendirections of the Hessian matrix to construct appropriate steplengths have
also been explored in different ways in
\cite{ddrt13,frassoldatiNewAdaptiveStepsize2008}. The main novelty of the
present work is then the combination of the relaxed gradient-type methods with
the Lanczos algorithm to fully exploit ideal search directions once the gradient
approaches some eigenvector of the Hessian matrix. In particular, we show that
this allows to reach the sought-after minimizer with a reduced number of
matrix-vector products.

The rest of the paper is organized as follows. Section~\ref{sec:quad} is dedicated to some review of standard results on the MR and SD methods for minimizing quadratic functionals. We also make some numerical observations which are the motivation for the various spectral acceleration techniques from  Section~\ref{sec:spectral}. The latter is at the heart of the paper: we introduce an acceleration method based on Lanczos subspace iterations, with various adaptive variants, that we complete with an analogy to algebraic multigrid methods. 
Section~\ref{sec:convex} is then dedicated to the extension of these accelerations techniques to the global minimization of non-quadratic convex functions. Finally, in Section~\ref{sec:det_analysis}, we analyze mathematically and prove a number of results on the observations made in Section~\ref{sec:quad}.

\def\lbet{\bar \lambda_{\beta,k}}
\def\llbet{\bar \lambda_{\beta,k+1}}

\section{Numerical observations for quadratic functionals}
\label{sec:quad}

\textbf{General considerations on quadratic functionals.} An important nonlinear
instance of \eqref{genprb} is the minimization of the strictly convex quadratic:
\begin{equation} \label{quad} f(x) = \frac{1}{2} x^\top Ax - b^\top  x,
\end{equation} where $b\in \R^n$ and $A\in \R^{n\times n}$ is a symmetric and
positive definite (SPD) matrix. Since $A$ is SPD and the gradient $g(x)\equiv
\nabla f(x) = Ax-b$, then the global minimizer of \eqref{quad} is the unique
solution $x^*= A^{-1}b$ of the linear system $A x = b$. The solution of
\eqref{quad} is of fundamental importance in the field of applied mathematics
and in numerous engineering developments; see \cite{Saad00}. Moreover, solving
it efficiently is usually a pre-requisite for a method to be generalized to
solve more general optimization problems.  In addition, by Taylor’s expansion, a
general smooth function can be approximated by a quadratic function near a
minimizer. Thus, the local convergence behavior of gradient methods is often
reflected by solving \eqref{quad}. For this reason, we make in this section, a
number of observations that are relevant to motivate the proposed acceleration
schemes, introduced in Section~\ref{sec:spectral}.

First, for solving \eqref{quad}, the optimal steplengths for SD and MR are
respectively given by, with $g_k = \nabla f(x_k)$,
\begin{align}
  \alpha_k^\text{SD} &= \frac{g_k^\top  g_k}{g_k^\top  A g_k} \label{eq:alpSD} \\
  \alpha_k^\text{MR} &= \frac{g_k^\top  A g_k}{ g_k^\top  A^2 g_k} . \label{eq:alpMR}
\end{align}
Notice that  both steplengths defined above are inverses of  Rayleigh quotients
of $A$ for some special vector in $\R^n$. By simple algebraic manipulations,  if
we apply either SD or MR  to solve \eqref{quad} but multiplying by 2 their
optimal step lengths, it follows that, for all $k\in\N$,
\[ f(x_k - 2 \alpha_k^\text{SD} g_k) = f(x_k)\;\; \mbox{ and }\;\;
  \|\nabla f(x_k - 2 \alpha_k^\text{MR} g_k)\| = \|\nabla f(x_k)\|.   \]
Hence, by continuity and strict convexity of both merit functions, using any
fixed relaxation parameter $\sigma\in(0,2)$ to multiply the optimal steplength
of either SD or MR  produces iterative methods for which the functions $f(x)$ or
$\|\nabla f(x)\|$ are monotonically decreasing. Moreover, for such relaxed
gradient-type schemes convergence of the iterates  to the unique solution of
\eqref{quad} is guaranteed. Most important, as long as $\sigma\ne 1$ (except
perhaps for a few selected iterations) the zigzag curse in a 2-dimensional
subspace is avoided, and the iterates move in the entire $n$-dimensional space;
see \cite[pp. 178-179]{MacD25} and \cite{Raydan02} for details. However, in this
case, the behavior of the iterations is not easily analyzed as it exhibits a
chaotic behavior \cite{vandendoel12}. What is rather interesting is the speed-up
observed for the situation when  $\sigma \ne 1$. Specifically, as discussed in
\cite[Section 8]{vandendoel12}, to avoid a possible chaotic behavior it is
recommended to choose the factor $\sigma\in (0,1)$ which has proved to be more
effective. In fact, experiments with relaxation values of $\sigma\in (1,2)$ have
been reported in \cite{MacD25} to produce a significant slower behavior (in the
sense that more iterations are required to achieve convergence), sometimes even
worse than using $\sigma=1$; see also \cite[Section 5]{vandendoel12}.

Now we recall that, for a given relaxation parameter $\sigma\in(0,2)$, the
relaxed gradient-type method reads
\begin{equation} \label{Relaxed}
  x_{k+1}  =  x_k - \sigma \alpha_k g_k \;\; \mbox{ for }\;\; \sigma\ne 1,
\end{equation}
where $\alpha_k$ stands either for $\alpha_k^\text{SD}$,
$\alpha_k^\text{MR}$. Notice that multiplying by $A$ and subtracting  the vector
$b$ on both sides, we obtain the gradient vector at step $k+1$:
\begin{equation} \label{eq:PowerM}
  g_{k+1}  =  g_k - \sigma \alpha_k Ag_k= (I-\sigma \alpha_kA)g_k=
  \left(\prod_{i=0}^k (I-\sigma \alpha_iA)\right)g_0.
\end{equation}
This establishes a relationship between the sequence $(g_k)_{k\in\N}$ generated
by \eqref{Relaxed} and the shifted power method. It is worth noting that this
connection with the shifted power method holds in principle for any
gradient-type method applied to \eqref{quad}. However, this  relationship is
more complex than just a common shifted power method. For instance, when
$\sigma=1$ the iterates tend to follow the zigzag behavior established by Akaike
\cite{Akaike}. In this case, it is known that for SD we have $g_{k+1}^\top
g_k=0$ for all $k$ and the iterates asymptotically zigzag in a 2-dimensional
subspace, forcing an angle bounded away from zero between the gradient vectors
and the 2 extreme eigenvectors of the matrix $A$. Moreover, it is also clear
that in the presence of the zigzag effect the gradient vectors asymptotically
tend to maintain orthogonality with all the other $n-2$ eigenvectors of $A$.
Therefore, under the zigzag effect no gradient vector $g_k$ approaches any
eigenvector of $A$. On the positive side, once the zigzag curse is avoided, the
iteration \eqref{eq:PowerM} resembles a  shifted power method and this suggests
that every so often the search direction $g_k$ approaches some eigenvector of
the matrix $A$. In addition, it turns out that even when the sequence of
gradients does not converge to an eigenvector, it still tends to have strong
components in a small number of eigenvectors, suggesting that projection
processes onto subspaces of dimension larger than $1$ may be useful, which we
explore later in Section~\ref{sec:spectral}.

\textbf{Numerical setting.} In all this section and Section~\ref{sec:spectral},
we perform numerical tests on a quadratic functional given by a $900 \times
900$ matrix $A$, obtained from the finite differences discretization of a
Poisson problem on a $30 \times 30$ grid. Throughout the paper we denote by $0.02 \approx
\lambda_1 \le \lambda_2 \le \cdots \le \lambda_n \approx 7.98 $ the eigenvalues of $A$ and
$u_1, u_2, \cdots, u_n$ an associated orthonormal set of eigenvectors. We denote
by $U$ the associated unitary matrix of eigenvectors. Finally, for every method,
the source term $b$ is generated randomly and we start from a random point
$x_0\in\R^n$.

\textbf{Numerical observations.} We begin with a few illustrations of what is
observed on the matrix $A$ introduced above when using the MR and SD methods. In
Figure~\ref{fig:RQfig}, we display the Rayleigh quotients, i.e., the inverses of
the steplengths $\alpha_k\up{MR}$ observed for this problem in two different
cases, $\sigma=0.8$ and $\sigma=1.8$. For $1 < \sigma < 2$ large enough the
oscillatory behavior disappears and the Rayleigh quotients seem to converge to
the value $\sigma (\lambda_1 + \lambda_n)/2$.  An example is shown in
Figure~\ref{fig:RQfig} where $\sigma$ takes the value $1.8$. For the case
$\sigma = 0.8$ these Rayleigh quotients tend to oscillate back and forth between
a high and a low value.  It is remarkable that there is a `bias' in the left
curve: whereas the median of the eigenvalues is 4.0, here the oscillation takes
place around a value significantly below this median. Visually it seems that the
oscillations hover around the value 3.6 which is the product of $\sigma$ by the
median $(\lambda_1 + \lambda_n)/2 = 4$. This is also the value to which the
Rayleigh quotients seem to converge when $\sigma = 1.8$. Similar observations are
made for the SD method and we refer to Theorem~\ref{thm:55} for a justification in a simplified setting.  This suggests that, when $\sigma \in (0,1)$, the
inverses of the steplengths span the entire spectrum, thus hitting from time to
time values that are very close to an eigenvalue of $A$.

Next we consider the sequence of normalized residuals and examine the components
of these vectors in the eigenbasis of $A$: we define the vectors
$(\beta_{i,k})_{i=1,\dots,n}$ given by $U^\top g_k / \|g_k\|$ (SD) or $U^\top g_k /
\|g_k\|_A$ (MR) -- a vector of length 900 in both cases. We plot the components
of these vectors in Figure~\ref{fig:betFig}, for a few values of $k$ and observe
that the components of the residuals are biased toward the extremal eigenvalues.
After 300 iterations, the normalized residual has a value very close to 1 in the
direction $u_1$ and a few (less than 10) components of order $10^{-4}$ at the
other end of the spectrum. After around $k=43$ there is not much qualitative
difference between the different figures obtained. Similarly, we plot in
Figure~\ref{fig:test_sigma} the convergence of (a few of) the normalized
components $|\beta_{i,k}|$ for different values of $i$ along the iterations $k$
for both the MR and SD methods. It clearly appears that, asymptotically, the
intermediate modes are vanishing and the residuals are mostly supported on a few
of the extremal modes, see Lemma~\ref{lem:random} for an analysis in a simplified setting.
Finally, in Figure~\ref{fig:test_sigma_res}, both MR and SD
with $\sigma=0.8$ exhibit a faster convergence of the residuals than $\sigma=1.8$,
in accordance with the discussion in \cite[Section 8]{vandendoel12} where
$\sigma\in(0,1)$ is shown to be more efficient. This is also reminiscent of
the conclusions from \cite{drtz18, fletcherLimitedMemorySteepest2012}, where steplengths
whose inverses span the whole spectrum of $A$ yield faster convergence.

\begin{figure}[h!]
  \centering
  \includegraphics[height=0.48\textwidth]{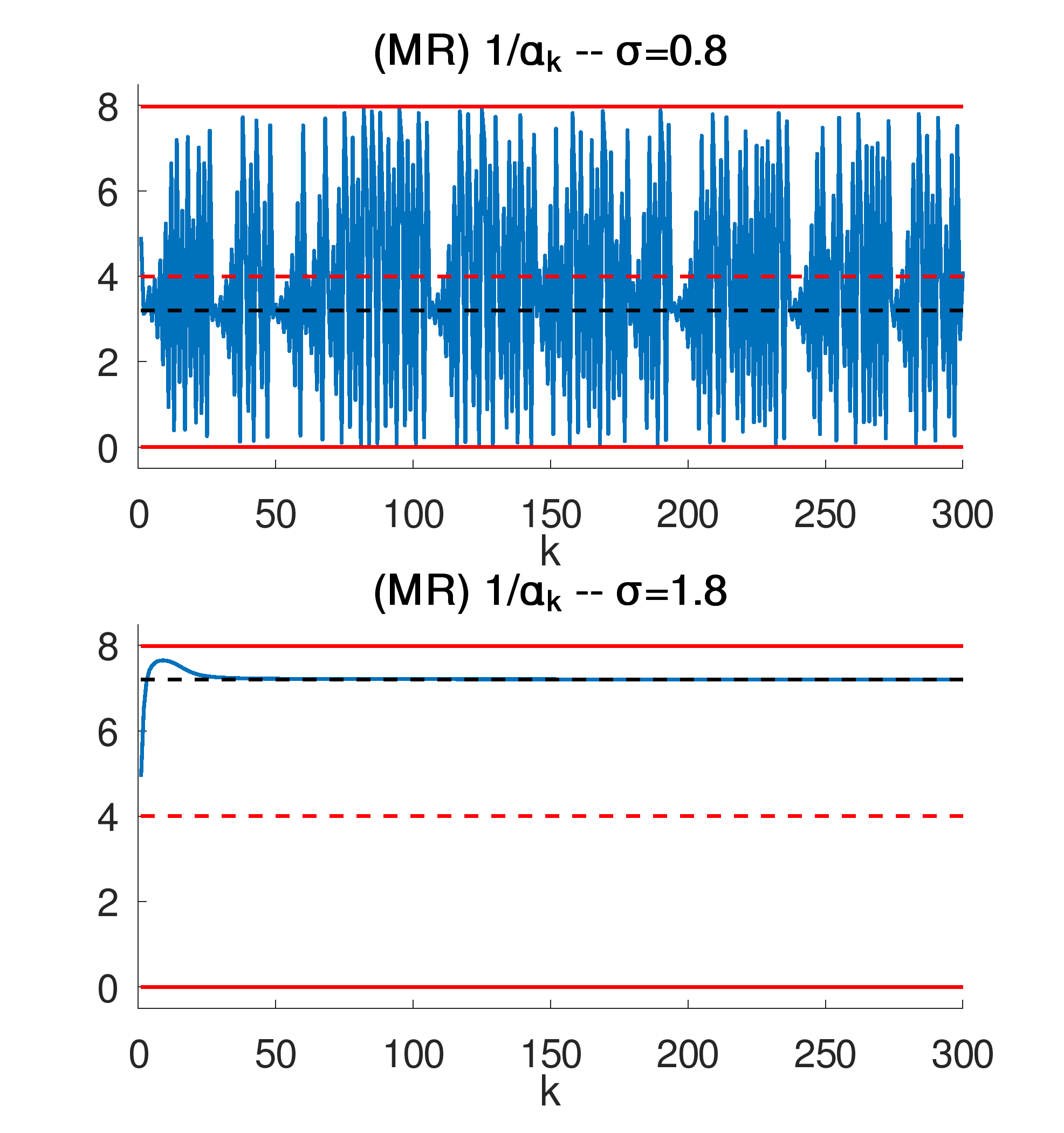}\hfill
  \includegraphics[height=0.48\textwidth]{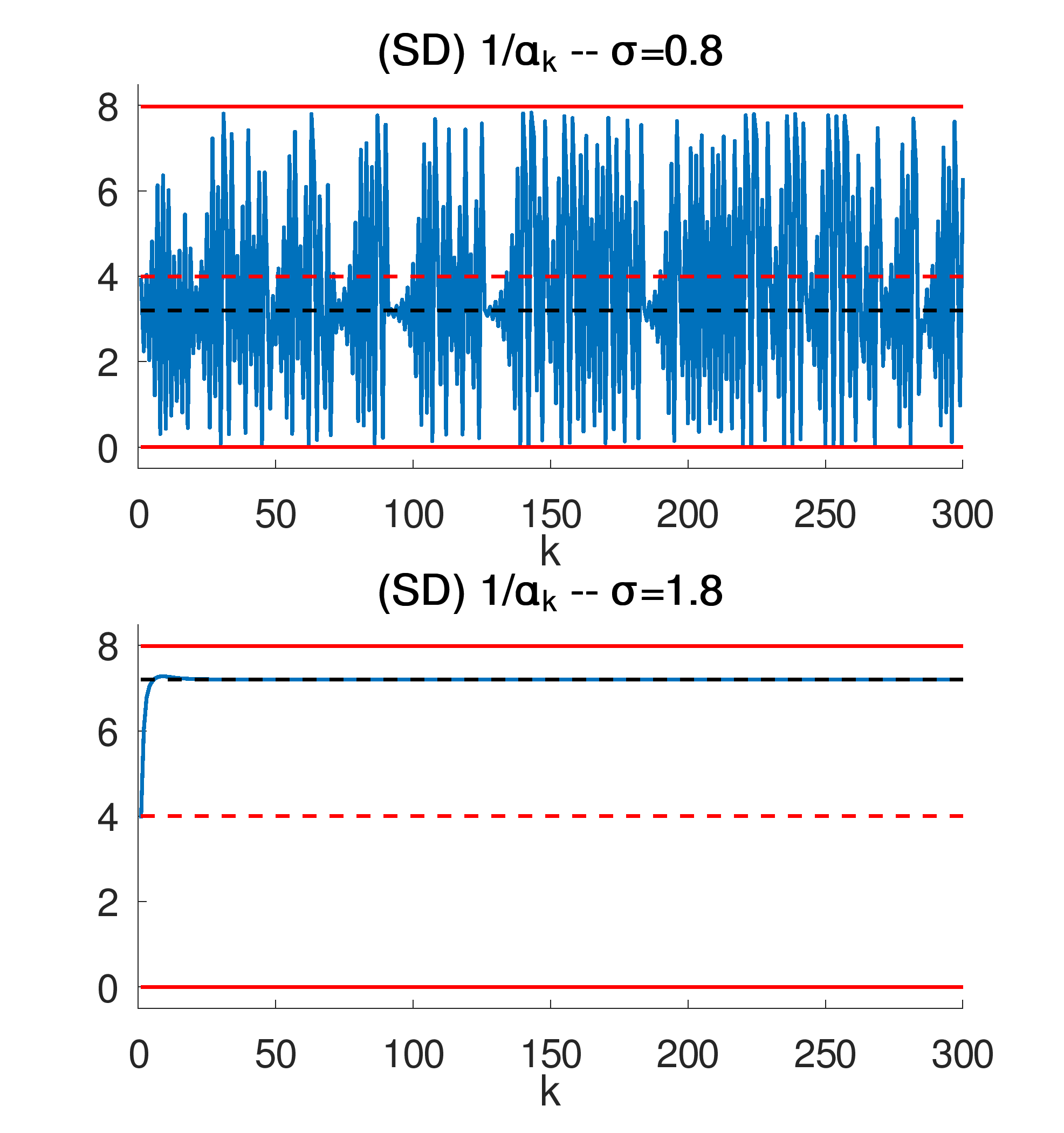}\\
  \caption{Sequence of Rayleigh quotients (${1/\alpha_k }$). The top and
    bottom lines are the largest and smallest eigenvalues. The dashed black line
    is $\sigma(\lambda_1 + \lambda_n)/2$ while the dashed red line is
    $(\lambda_1 + \lambda_n)/2$. }
  \label{fig:RQfig}
\end{figure}

\begin{figure}[h!]
  \centering
  \includegraphics[height=0.45\textwidth]{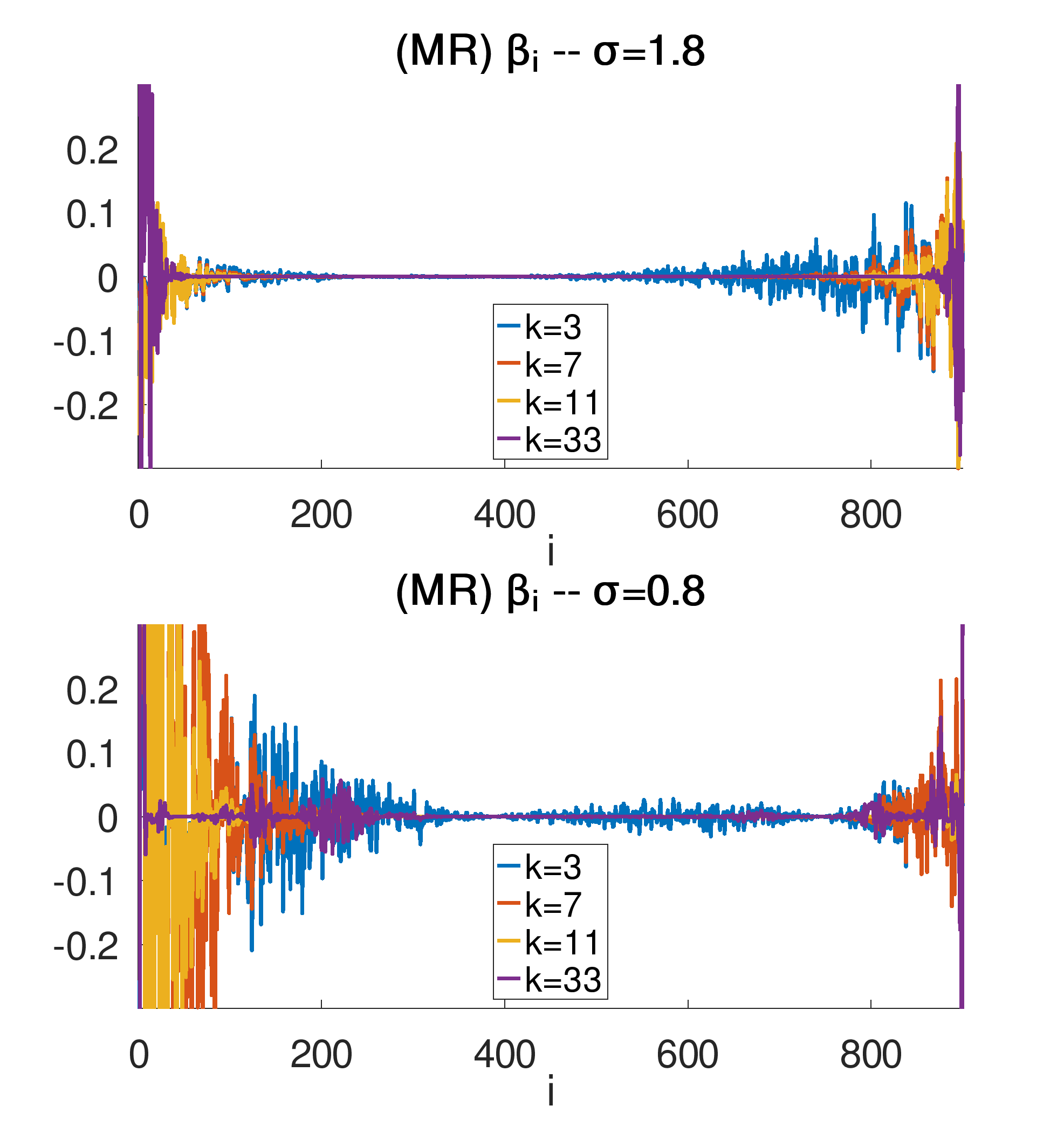}\hfill
  \includegraphics[height=0.45\textwidth]{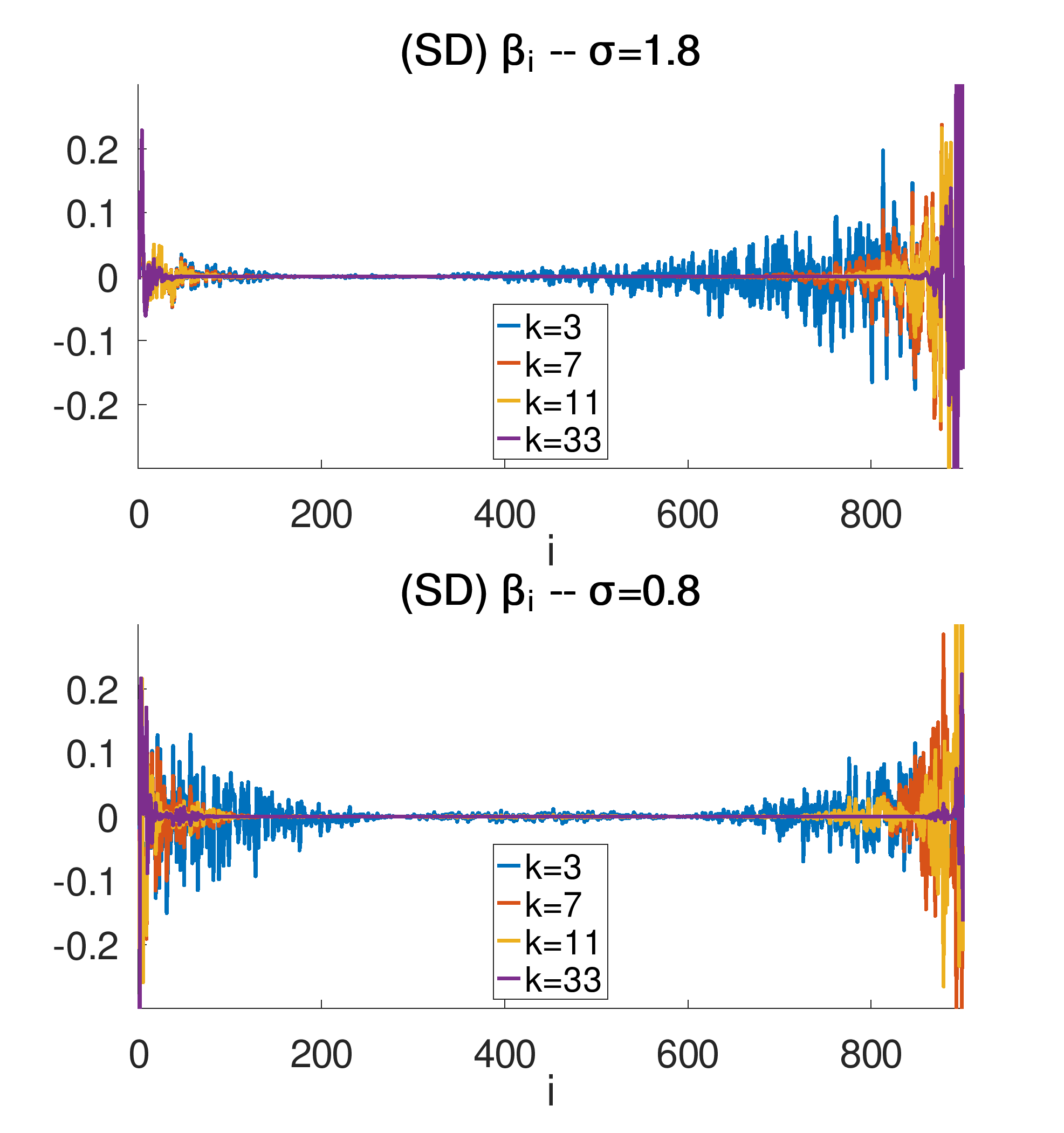}\\
  \caption{Components of the normalized residuals in the eigenbasis for
    iterations $k=3,7,11,33$.} \label{fig:betFig}
\end{figure}

\begin{figure}[p!]
  \includegraphics[width=0.45\linewidth]{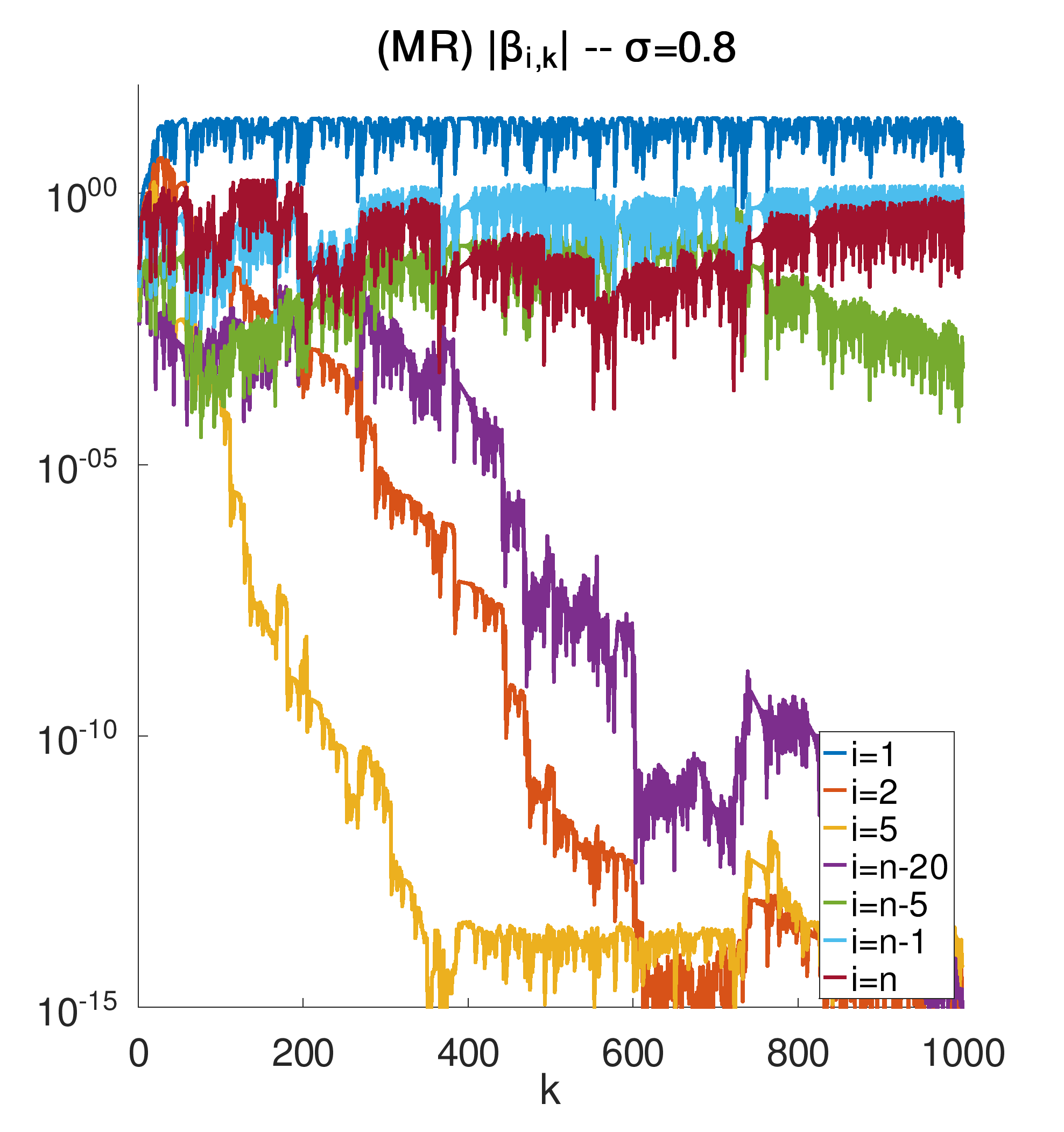}\hfill
  \includegraphics[width=0.45\linewidth]{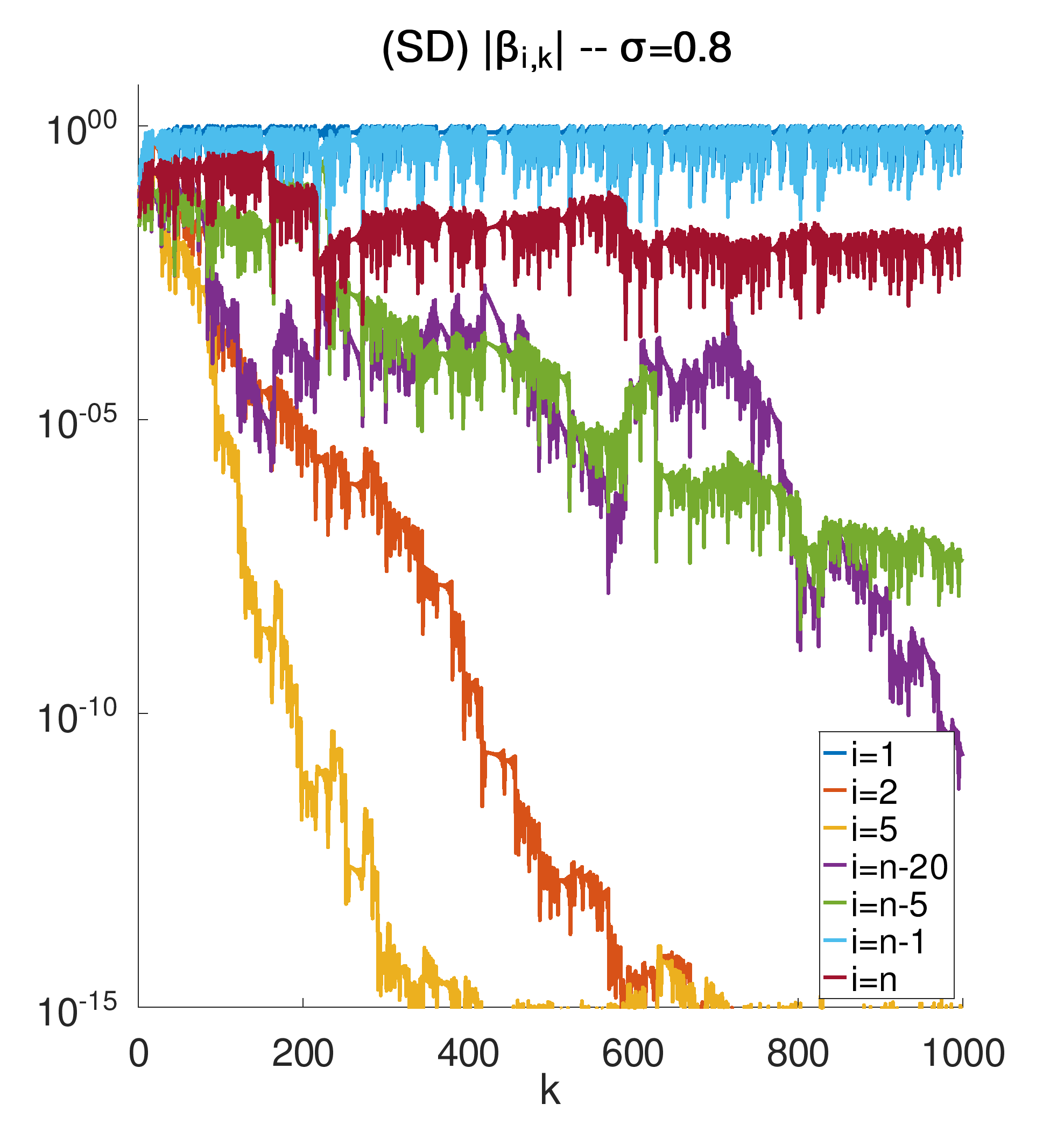}\\
  \includegraphics[width=0.45\linewidth]{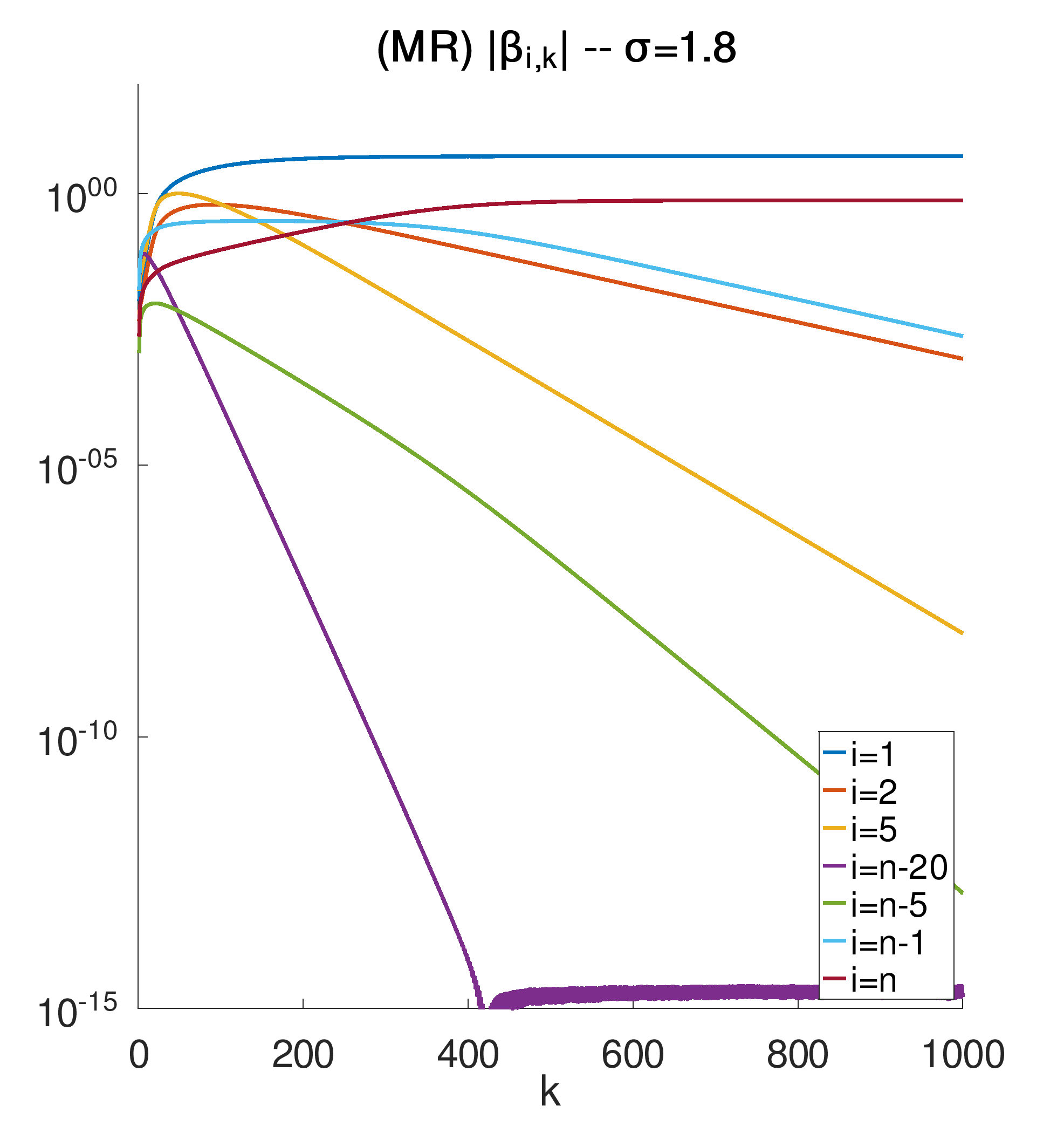}\hfill
  \includegraphics[width=0.45\linewidth]{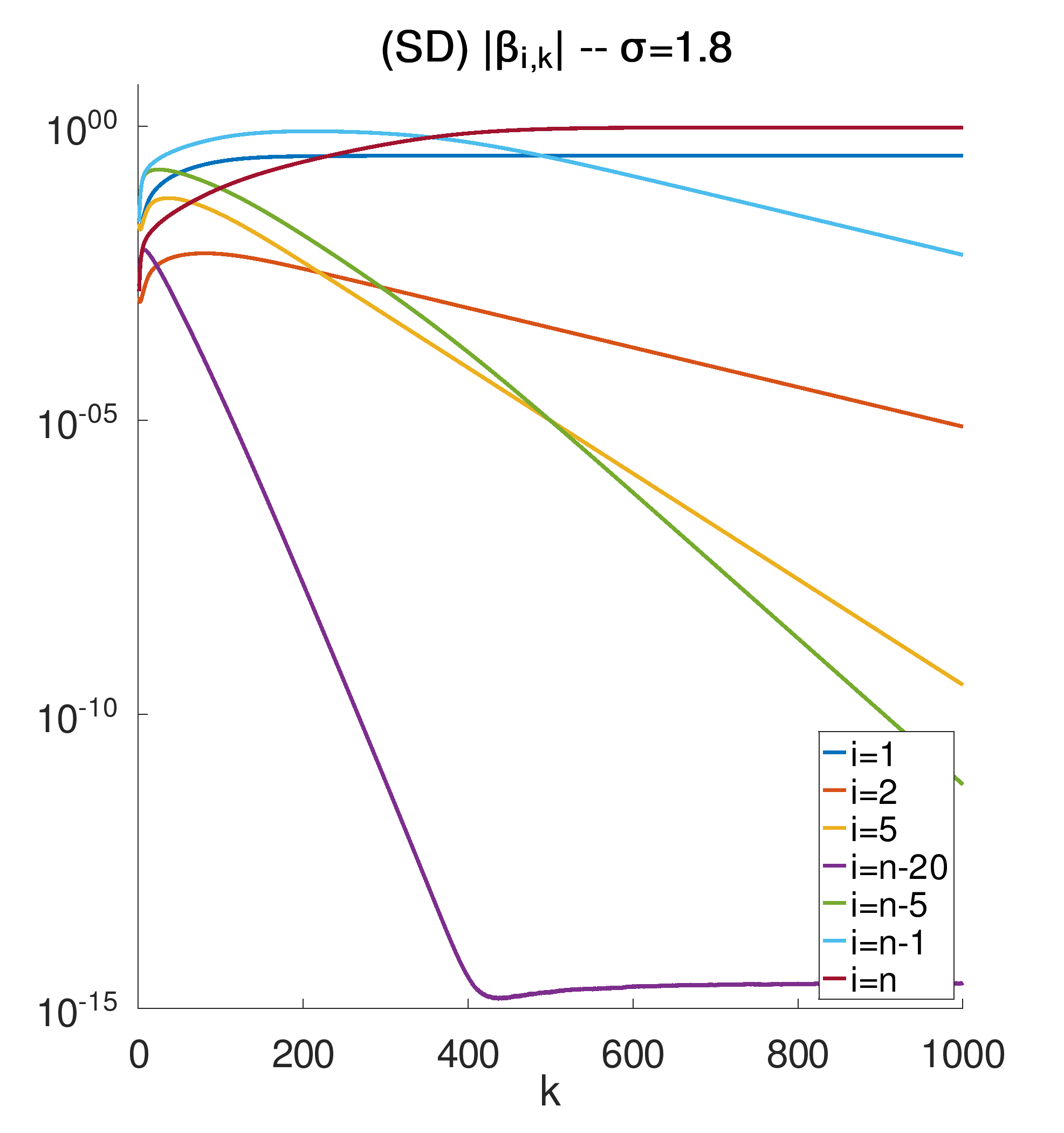}
  \caption{Convergence of $|\beta_{i,k}|$ for different values of $i$ and
    $\sigma=0.8$ (top) or $\sigma=1.8$ (bottom),
    in the case of the MR (left) or SD (right) algorithms. We used
    different random initial guesses. The normalized residuals are supported
    by a few extremal modes (here, mostly a few of the lowest modes and a few of the
    highest modes), and all the intermediate modes
    vanish asymptotically. Only a few modes are plotted for the sake of
    clarity.}
  \label{fig:test_sigma}
\end{figure}

\begin{figure}[h!]
  \includegraphics[width=0.45\linewidth]{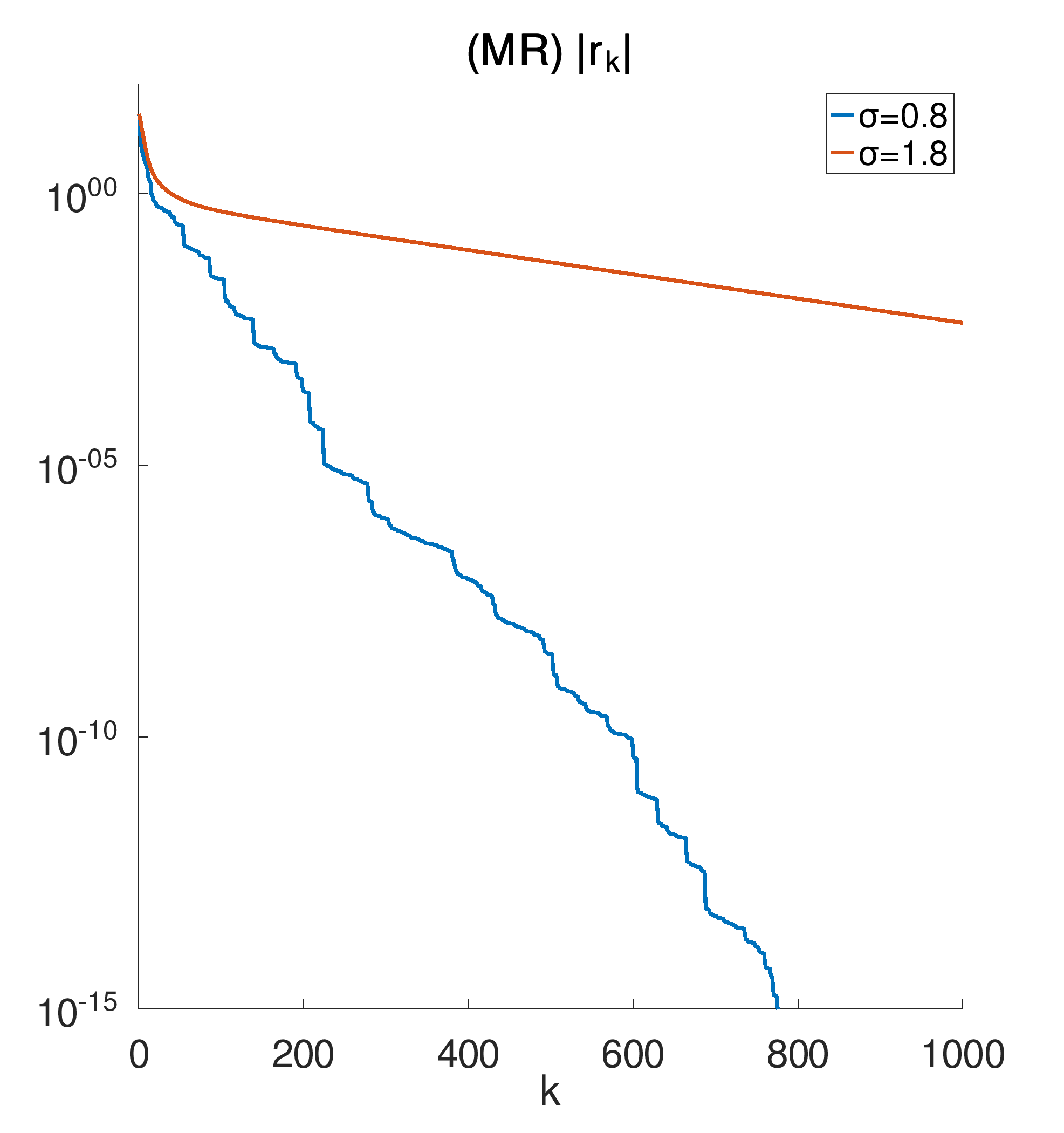}\hfill
  \includegraphics[width=0.45\linewidth]{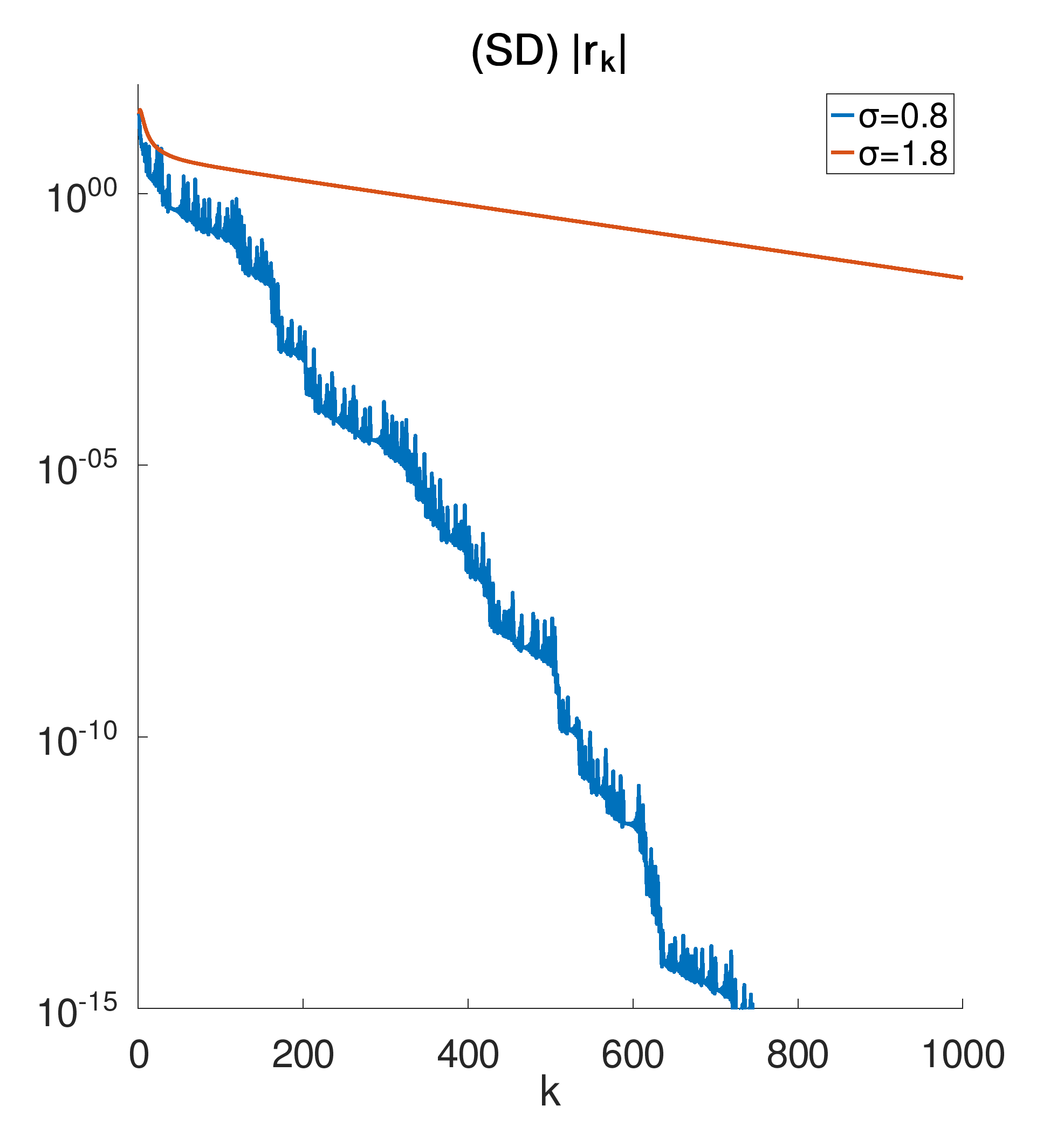}\\
  \caption{Convergence behavior of the residuals for $\sigma=0.8$ or $\sigma=1.8$, for the MR (left) and SD (right) algorithms.}
  \label{fig:test_sigma_res}
\end{figure}

\textbf{Conclusions.} We observed numerically that, in the absence of the zigzag
effect, the residuals of the MR and SD methods are asymptotically supported by a
few extremal eigenmodes of the matrix $A$. This motivates to use this extra
information in order to accelerate the convergence once the residuals are well
approximating eigenvectors, which is the goal of the next section.
Moreover, when resorting to the relaxed scheme,
choosing $\sigma\in(0,1)$ also gives faster convergence, justifying the practice of selecting
such values for $\sigma$. Note that some of the observations we made
in this section can be theoretically justified, see Section~\ref{sec:det_analysis}.

\section{Spectral acceleration techniques}\label{sec:spectral}

\subsection{Eigenvector based acceleration}\label{sec:EigAcc}

\subsubsection{Preliminaries}

Once a search direction $g_k$ becomes a good approximation to an eigenvector of
$A$ at some iteration $k$, then the relaxation factor $\sigma$ should be set to
1 in order to use the exact optimal steplength, since this optimal choice drives
$x_{k+1}$ very close to the exact solution of the minimization problem, as the
following result indicates.
\begin{lemma} \label{eigvgrad}
  If the iterative scheme $x_{k+1}  =  x_k - \sigma \alpha_k g_k$ is used to
  solve  \eqref{quad} with $\sigma\in(0,2)$, $\alpha_k$ obtained via
  \eqref{eq:alpSD} or \eqref{eq:alpMR}, and at some iteration $k\geq 1$ the
  gradient vector $g_{k}=v$, where $Av=\lambda v$, then setting $\sigma=1$ at
  iteration $k$ results in $x_{k+1}= x^*= A^{-1}b$.
\end{lemma}
\begin{proof}
  From \eqref{eq:alpSD} -- \eqref{eq:alpMR}  and the fact that $Av=\lambda v$,
  we obtain that  $\alpha_{k}=1/\lambda$. Since at that iteration  $\sigma=1$,
  then $x_{k+1}  =  x_{k} - (1/\lambda)g_{k}$. Multiplying by $A$ and
  subtracting $b$ from both sides, we get: $g_{k+1}= g_{k} -v =0$, and hence
  $x_{k+1}= x^*= A^{-1}b$.
\end{proof}

Concerning the tendency to produce gradient directions that approximates an
eigenvector of $A$, from now on we favor the use of the relaxed MR method,
since it has been frequently observed that it tends to produce a next gradient
direction  more inclined to an eigenvector direction than the SD method; see
\cite{Zhou06} for additional details. We now present the linear convergence
result of the relaxed MR method with a fixed parameter $\sigma\in(0,1)$ which is
adapted from Theorem 8 and Lemma 11 in \cite{MacD25} for the specific choice
$\ell=1/2$ and that will be required to justify the global convergence of the
acceleration techniques introduced in this paper.
\begin{theorem} \label{convMR}
  Let $f$ be the strictly convex function  given in  \eqref{quad},   $\sigma \in
  (0, 2)$, and $c(\sigma)$ defined as $$ c(\sigma) = \Big(1- \sigma(2-\sigma)
  \frac{4 \lambda_{\min} \lambda_{\max}}{(\lambda_{\min} + \lambda_{\max})^2}
  \Big)< 1,  $$ where  $\lambda_{\min}$  and   $\lambda_{\max}$ are the smallest
  and the largest eigenvalues of the Hessian matrix $A$, respectively. Given an
  initial point $x_0 \in \R^n$, for $k\geq 0$, let the iterates be generated by
  the MR method $$ x_{k+1}  =  x_k - \sigma \alpha_k^\text{MR} g_k. $$ Then
  \begin{equation} \label{LinRate} \|g_k\|^2  \leq  c(\sigma) \|g_{k-1}\|^2
    \leq\dots \leq  c(\sigma)^k \|g_0\|^2. \end{equation}
\end{theorem}
Notice that the linear convergence result in Theorem \ref{convMR} is based on a
worst-case analysis, and in practice for $\sigma \in (0, 1)$ the relaxed MR
method requires less iterations than the slow linear rate predicted by
\eqref{LinRate}. Notice also that this result, for any fixed $\sigma \in (0,
2)$,  implies that  $\|g_k\|$ goes  to zero when $k\rightarrow \infty$. Since
$f$ is strictly convex, by continuity it  follows that the sequence
$\{x_k\}_{k\in\N}$ converges to $x^*= A^{-1}b$.

\subsubsection{Analysis for approximate eigenvectors}

We now extend the result of Lemma~\ref{eigvgrad} to the more practical situation
when the vector $g_{k} = v$ is only an approximate eigenvector. Consider a step
$k$  at which we apply the iterate \eqref{genscheme} with either
\eqref{eq:alpSD} or  \eqref{eq:alpMR}. Either of these two steplengths is
denoted  by $\alpha_k $ in what follows. If we view  the current gradient
direction $g_k$ as an approximate eigenvector and we recall that $\alpha_k\inv$
is the associated Rayleigh quotient, we can define the associated residual:
\begin{equation}
  \rho_k = \frac{(A - \alpha_k \inv  \ I )g_k}{\| g_k \|} .  \label{eq:EigRes}
\end{equation}

Let us now compute the negative residual for the next iterate of the SD or MR
scheme. Starting with
\[  g_{k+1} = A x_{k+1} - b = A (x_k - \alpha_k g_k) - b= g_k - \alpha_k A g_k, \]
we see that
\begin{equation*}
  g_{k+1} = \alpha_k [ \alpha_k\inv g_k - A g_k ] = - \alpha_k \| g_k \| \  \rho_k
\end{equation*}
We have therefore proved the following result, which is actually valid for any gradient-type method.
\begin{proposition}\label{prop:eig}
  Assume that the iteration \eqref{genscheme} is applied at a certain step $k$
  with either steplength \eqref{eq:alpSD} or \eqref{eq:alpMR} and let $ \rho_k$
  be the residual of the vector $g_k$ considered as an approximate eigenvector
  of $A$, as given in \eqref{eq:EigRes}. Then the gradient of the next iterate
  satisfies the \emph{equality}: \begin{equation} \| g_{k+1} \| = |\alpha_k |
    \times \|\rho_k \| \times \| g_k \| \label{eq:Nrmglp1} , \end{equation} in
  which $\alpha_k$ stands for either $\alpha_k\up{SD} $ or  $\alpha_k\up{MR} $.
\end{proposition}

Note that we recover the case of Lemma~\ref{eigvgrad} when $g_k$ is an exact
eigenvector since in this situation $\rho_k=0$ and therefore we clearly
have $g_{k+1} = 0$.  In the general situation and when $\alpha_k$ is not
large the next gradient norm is of the order of the \emph{product of the
  previous gradient norm and the norm of the residual norm of the same gradient
  considered as an approximate eigenvector.} Note that the eigenvector residual
norm uses a different Rayleigh quotient for each of the methods SD or MR.

Next we examine in more details the gain that can be made in one acceleration step
under the above conditions. Our specific goal is to show that if the current
residual $g_k$ is a good approximate eigenvector of $A$, then $\sigma = 1$
should result in a good reduction in the cost function, i.e., the standard
Euclidean norm for MR and the $A$-norm of the error (or $A\inv $-norm of the
residual) for SD. Since $\sigma = 1$, we have
\eq{eq:Step} g_{k+1} = g_k - \alpha_k A g_k. \en As the next Lemma states, the
sine of the angle\footnote{Recall that for an SPD matrix $B$ $\| x \|_{B}^2 = (B
  x,x)$ and that $\cos_{B} \angle (x,y) = (Bx,y)/(\|x\|_B \|y\|_B)$.} between
$g_k$ and $Ag_k$ -- as measured by the appropriate inner product -- provides
a measure of the progress made in one projection step.

\begin{lemma}\label{lem:gain}
  From one step to the next the residual norms in MR are related as follows when
  $\sigma=1$: \eq{eq:MRgain}
  \frac{\| g_{k+1}\|^2}{\|g_k\|^2 } = \sin^2 \angle (g_k,Ag_k)
  \en
  From one step to the next the $A\inv$ residual norms in SD are related as
  follows when $\sigma=1$:
  \eq{eq:SDgain}
  \frac{\| g_{k+1}\|_{A\inv}^2}{\|g_k\|_{A\inv}^2 }
  = \sin^2 \angle_{A\inv}  (g_k,Ag_k)
  \en
\end{lemma}
\begin{proof}
  Inequality \eqref{eq:MRgain} is well-known when studying the convergence of
  the MR method -- see for example, \cite[Section 5.3.2]{Saad-book2}. The same
  analysis goes through to the SD case if we replace the standard inner product
  by the $A\inv$-inner product and norm.
\end{proof}

The well-known formula \eqref{eq:MRgain} thus expresses the progress made in one
step in terms of the angle between $g_k$ and $Ag_k$.  Notice now that if $g_k$
is close to an eigenvector then the angle between $g_k$ and $Ag_k$ is small
and this results in good progress from the current step to the next. We now
express the fact that $g_k$ is an approximate eigenvector by stipulating that $
Ag_k = \mu_k g_k + w_k $ where $w_k$ is a vector of small length (relative to
$g_k$), that is orthogonal to $g_k$, and $\mu_k$ is the associated approximate
eigenvalue.

\begin{proposition}\label{prop:gain}
  Let    $g_k$ be an approximate eigenvector such that
  \eq{eq:revec}
  Ag_k = \mu_k g_k + w_k
  \en
  where $w_k$ is orthogonal to $g_k$ for MR, and $A\inv$-orthogonal to $g_k$ for
  SD. If we assume that
  \eq{eq:epsDef}
  \text{MR:} \quad
  \frac{\| w_k \|}{\| g_k \|} \le \epsilon , \qquad
  \text{SD:} \quad
  \frac{\| w_k \|_{A\inv}}{\| g_k \|_{A\inv} } \le \epsilon
  \en
  then the pair of vectors  $g_k$ and $Ag_k$  satisfy:
  \eq{eq:sinTheta}
  \text{MR:} \quad \sin \angle (g_k,Ag_k) \le \frac{ \epsilon  }  { \sqrt{ \mu_k^2 + \epsilon^2 } } ,\qquad
  \text{SD:} \quad \sin \angle_{A\inv} (g_k,Ag_k) \le \frac{ \epsilon  }  { \sqrt{ \mu_k^2 + \epsilon^2 } } \  .
  \en
\end{proposition}
\begin{proof}
  This proof is for the MR algorithm. An  identical proof for the SD method
  uses $A\inv$ inner products and norms\footnote{Note that for SD we also assume
    that the relation \eqref{eq:revec} holds but that $w_k$ is
    $A\inv$-orthogonal to $g_k$ so that $(Ag_k,g_k)_{A\inv} = \mu_k
    (g_k,g_k)_{A\inv}$. This means that in this case, $\mu_k =
    (Ag_k,g_k)_{A\inv} / (g_k,g_k)_{A\inv} = (g_k,g_k)  /(A\inv g_k,g_k)$.}. From
  the assumptions  $(Ag_k,g_k) = \mu_k (g_k,g_k)$ and so $\mu_k$ is equal to the
  standard Rayleigh quotient of $g_k$ with respect to $A$. Also $ \|Ag_k \|^2 =
  \| \mu_k g_k + w_k \|^2 =\mu_k^2 \|g_k\|^2 + \|w_k\|^2$. In the end, with the
  assumption \eqref{eq:epsDef}  the angle $\theta \equiv \angle (g_k,Ag_k)$ is
  such that
  \begin{align*}
    \cos \theta &=
    \frac{(Ag_k,g_k) } {\| Ag_k \| \| g_k\|}
    =  \frac{ \mu_k  \|g_k\|^2 } {\|g_k \| \sqrt{ \mu_k^2 \| g_k \|^2 + \| w_k \|^2} }
    =  \frac{ \mu_k  }  { \sqrt{ \mu_k^2 + (\|w_k\|/\|g_k\|)^2 } }  \ge
    \frac{\mu_k}{\sqrt{ \mu_k^2 +\epsilon^2 } } \\
  \end{align*}
  from which we deduce
  \begin{align*}
    \sin^2 \theta & \le 1-
    \frac{ \mu_k^2  }  {\mu_k^2 + \epsilon^2 } =  \frac{ \epsilon^2 }  {\mu_k^2 +
      \epsilon^2 }.
  \end{align*}
  Inequality \eqref{eq:sinTheta} then follows.
\end{proof}

For example  a reduction by a factor of at least 2 in residual norm is
obtained if $\epsilon \le \mu_k / \sqrt{3}  $. All these results motivates a
specific treatment of the iterations for which we can detect that the residual
is close to an eigenvector, in order to accelerate the convergence of the method
towards the minimizer. This is the goal of what follows.

\subsubsection{Practical application}

We now discuss how to detect, without increasing the number of matrix-vector
products, that $g_k$ approaches an eigenvector of $A$. According to
\eqref{eq:EigRes}, the closeness can be measured by monitoring\footnote{Note
  that the matrix-vector product $A g_k$ that appears in \eqref{eq:eigRes} is
  also required in the MR method and therefore the test does not incur
  significant additional computation.} at every iteration the normalized
eigenvector-residual
\begin{equation}
  \frac{\alpha_k}{\|g_k\|} \Big\|A g_k - \alpha_k^{-1}g_k \Big\|
  = |\alpha_k|\norm{\rho_k}.
  \label{eq:eigRes}
\end{equation}
If this eigenvector-residual is close to zero, then clearly $g_k$ is close to an
eigenvector. Moreover, if $|\alpha_k|\norm{\rho_k} < \epsilon_{\rm eig}$ then,
by Proposition~\ref{prop:eig}, this ensures that $\norm{g_{k+1}} < \epsilon_{\rm
  eig}\norm{g_k}$.  Taking this remark into account, we present in
Algorithm~\ref{alg:EigAcc} our first acceleration scheme, named the `eigenvector
acceleration scheme'.

\begin{algorithm}[ht]
  \centering
  \caption{Eigenvector Acceleration Scheme}\label{alg:EigAcc}
  \begin{algorithmic}[1]
    \State Start: a given initial guess $x_0$; set $r_0=b-Ax_0$ and $p_0=Ar_0$;
    \State Choose: $\sigma$ such that  $0 < \sigma < 1$, and $0<\epsilon_{\rm
      eig} < 1$
    \While{(not converged)}
    \State $\ds\alpha_k = \frac{p_k^\top  r_k}{p_k^\top  p_k}$
    \If {$\frac{\alpha_k}{\|r_k\|}\Big\|p_k - \alpha_k^{-1}r_k\Big\| <
      \epsilon_{\rm eig}$}
    \Comment{Test whether $r_k = -g_k$ is aligned with an eigendirection\\{\hspace{7.5cm} (recall that $p_k = Ar_k$).}}
    \State $\tau_k = 1$ \Comment{If yes, the steplength $\alpha_k$ is optimal.}
    \Else
    \State $\tau_k = \sigma$ \Comment{If not, use steplength $\sigma\alpha_k$ to break the zigzag
      effect.}
    \EndIf
    \State $x_{k+1} = x_k + \tau_k\alpha_k r_k$
    \State $r_{k+1} = r_k - \tau_k\alpha_k p_k$
    \State $p_{k+1} = Ar_{k+1}$
    \EndWhile
  \end{algorithmic}
\end{algorithm}

Our next result,  obtained as a direct consequence of Theorem \ref{convMR},
establishes the convergence of the eigenvector acceleration algorithm.
\begin{corollary} \label{convEigAcc}
  Let the conditions of Theorem \ref{convMR} hold. Then the sequence
  $\{x_k\}_{k\in\N}$ generated by the eigenvector acceleration algorithm
  converges to $x^*= A^{-1}b$.
\end{corollary}
\begin{proof}
  Using a fixed  $\sigma \in (0, 1)$ the eigenvector acceleration algorithm
  produces a subsequence $K_1\subset \N$ for which $\tau_k=\sigma$ is used to
  generate the iterates. There exists a complementary  subsequence  $K_2\subset
  \N$ ($K_1 \cup K_2 = \N$) for which $\tau_k=1$ is used to generate the
  iterates.  At any iteration $k$, regardless of whether $\tau_k\in (0, 1)$ or
  $\tau_k=1$, $\|g_{k+1}\| < \|g_{k}\|$. Therefore, the sequence
  $\{\|g_{k}\|\}_{k\in \N}$  is monotonically decreasing, and  since it is
  clearly bounded below it converges. Now, from Theorem  \ref{convMR} and the
  monotonicity of the entire sequence,  we have that $\{\|g_{k}\|\}_{k \in K_1}$
  converges to zero. Consequently,   the whole sequence $\{\|g_{k}\|\}_{k\in
    \N}$   converges to zero, which in turn implies  by continuity that the
  sequence $\{x_k\}_{k\in \N}$  converges to   $x^*= A^{-1}b$.
\end{proof}

As before, the convergence result in Corollary \ref{convEigAcc}  is based on a
worst-case analysis, and in practice, thanks to what is indicated in
Lemma~\ref{eigvgrad}, Lemma~\ref{lem:gain} and Proposition~\ref{prop:gain}, the
eigenvector acceleration algorithm requires far fewer iterations than the
expected slow linear rate described in \eqref{LinRate}. We now illustrate in
Figure \ref{fig1} the behavior of the eigenvector-based acceleration for the
matrix $A$ introduced in Section~\ref{sec:quad}.  The tracked quantity in the
plots is  the 2-norm of the residuals
$r_k = b - Ax_k = - g_k$. The left plot shows that (i) a significant drop is
observed when the criterion from Line~5 in Algorithm~\ref{alg:EigAcc} is
activated and (ii) the convergence towards the solution is faster than with
standard MR, due to break of the zigzag effect. However, on the right plot, we
observe that breaking the zigzag effect with $\tau_k=\sigma$ at every iteration
(without monitoring when the residual is close to an eigenvector) produces a
similar convergence history. This observation motivates the search for an
alternative acceleration scheme that exploits further the search direction when
it aligns with an eigenvector, as the method we introduce in the next section.

\begin{figure}[h!]
  \centerline{
    \includegraphics[width=0.45\linewidth]{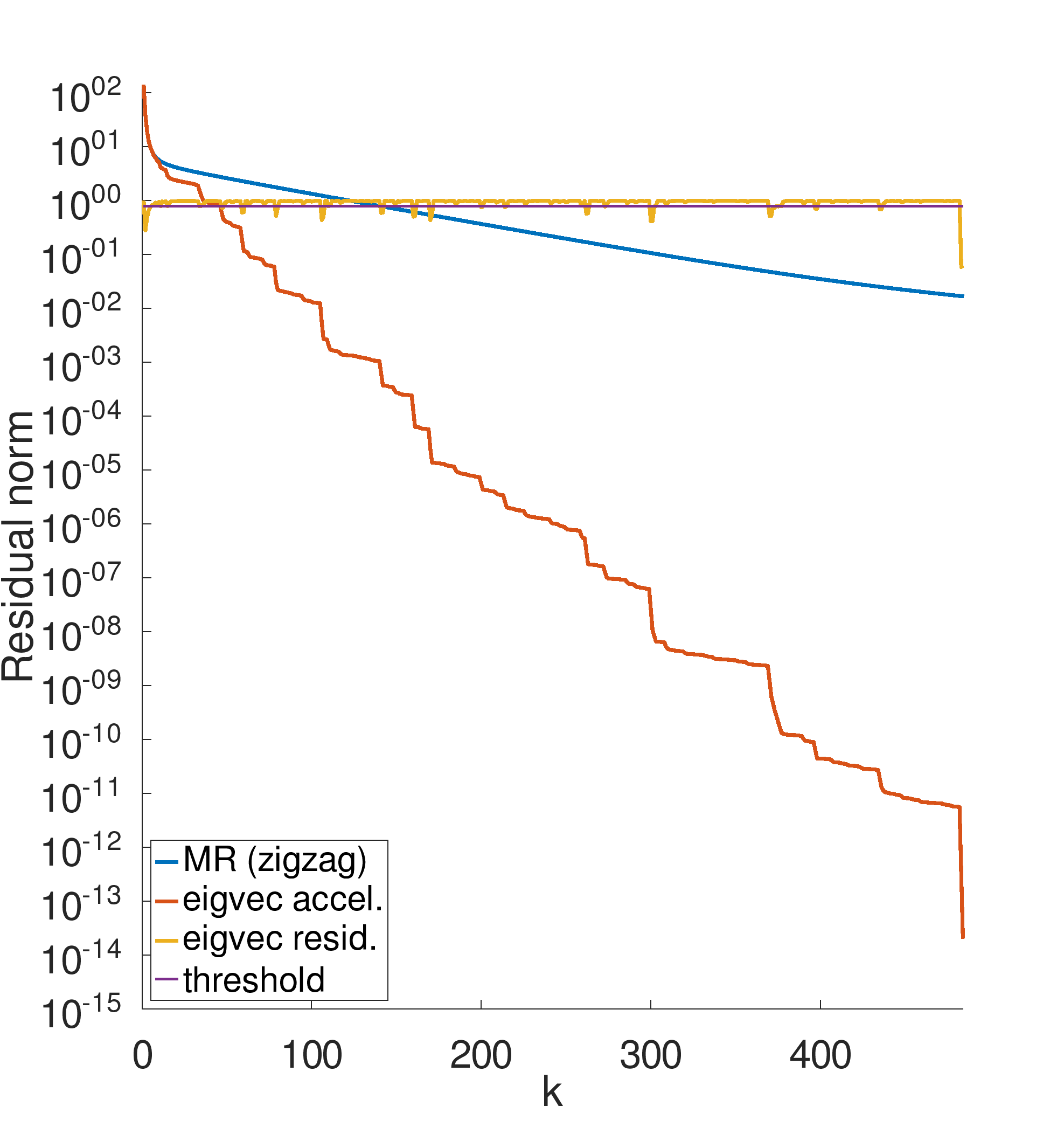}
    \includegraphics[width=0.45\linewidth]{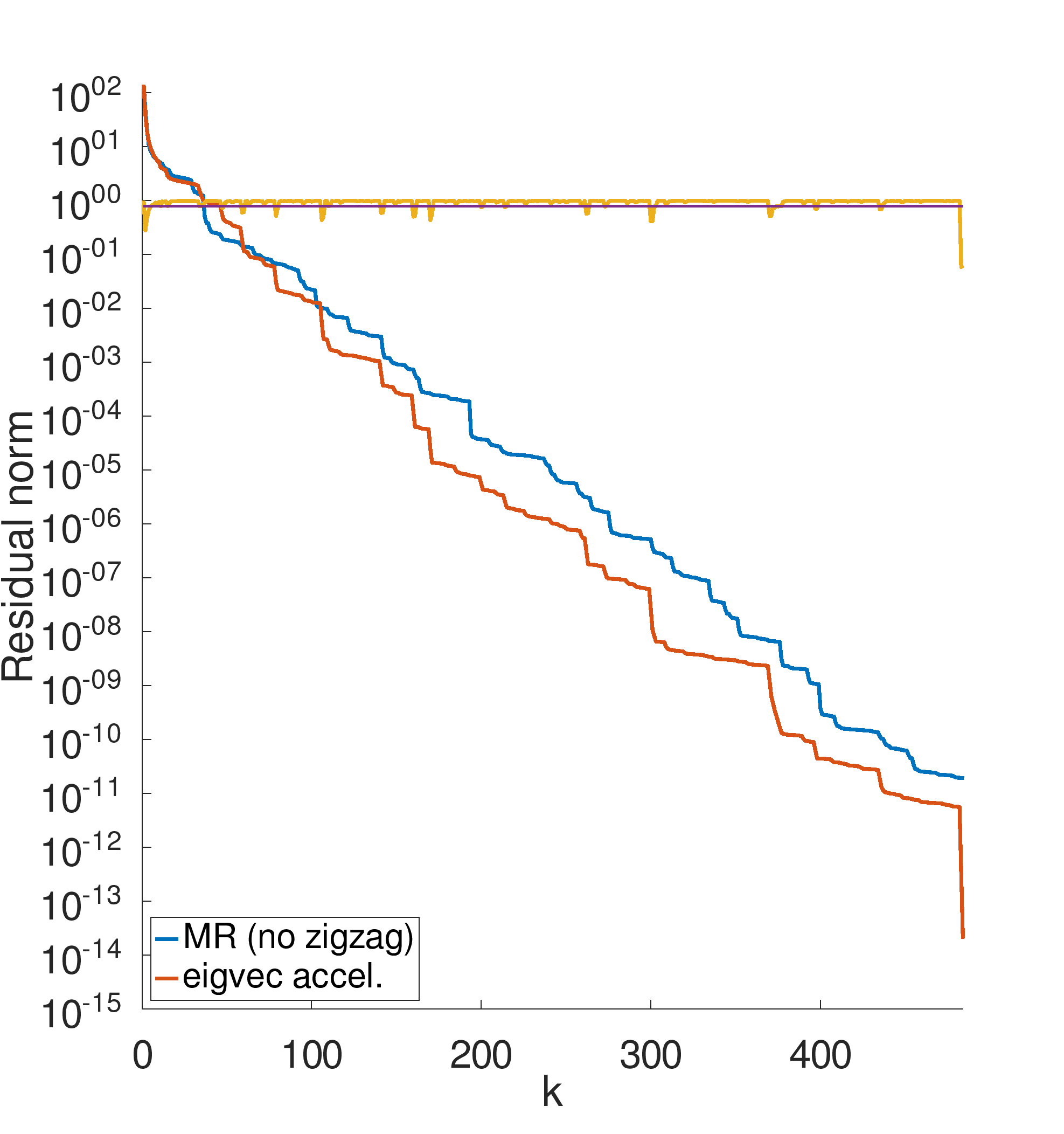}}
  \caption{Eigenvector based acceleration ($\sigma=0.8$, $\epsilon_{\rm eig} =
    0.8$): when the residual is close to an eigenvector (that is, the yellow
    line goes below the threshold), the acceleration is activated (a drop is
    observed on the red line). On the left, the blue line is for comparison with
    MR ($\tau_k=1$ at every iteration): it converges much slower due to the
    zigzag effect. On the right, the blue line is MR with
    $\tau_k=\sigma\in(0,1)$ at every iteration: it behaves similarly to the
    eigenvector acceleration.}\label{fig1}
\end{figure}

\subsection{Lanczos-based acceleration}\label{sec:LBA}

The eigenvector acceleration idea from the previous section can be pushed
further by using a few step of the Lanczos algorithm when the closeness
criterion is activated. To motivate the Lanczos-based projection, we first go
back to the illustration of the behavior of the residual vectors seen in
Section~\ref{sec:quad}, see in particular the discussion of the plots in
Figure~\ref{fig:betFig}. There we saw that the residual vector tends to have not
one large components in the eigenvectors of $A$ but several. This means
that in the transient stage it may be beneficial to project out several
eigenvectors instead of just one as was done in the eigenvector acceleration
case.

We use the same activation criterion as Algorithm~\ref{alg:EigAcc} (Line
5), after which we apply a few (say $m\geq 1$) Lanczos iterations, starting from
the residual vector $r_k=-g_k$ that approximates an eigenvector,  to get an
orthonormal column matrix $V_m\in\R^{n\times m}$ and a tridiagonal matrix $T_m
\in\R^{m \times m}$ such that $T_m=V_m^TAV_m$. Solving the least-squares problem
\eq{eq:LSpb}
  y_m = {\rm argmin}_{z\in\R^{m}}  \|r_k - AV_mz\|^2,
\en
we can accelerate the convergence by setting $x_{k+1} = x_k + V_my_m$ and
$r_{k+1} = r_k - AV_my_m$. Clearly, a suitable choice to start the Lanczos
procedure is to take the vector $r_k$ as its initial vector since $r_k$ is known
from the activation criterion to be a good approximation of at least one
eigenvector of $A$. The Lanczos-based acceleration (LBA) scheme is now presented
in Algorithm~\ref{alg:LBA}. Note that there is a structural difference between
the eigenvector acceleration (Section~\ref{sec:EigAcc}) and the Lanczos-based
acceleration that is important to highlight. The eigenvector acceleration
algorithm uses two different relaxation parameters ($\sigma\in(0,1)$ or 1), but
always uses the negative gradient direction (or residual) to generate the next
iterate. Therefore, a proper use of Theorem \ref{convMR} formally justifies its
global convergence. In the Lanczos-based algorithm, there is a subsequence of
iterations where the direction is no longer the negative gradient, but a
modification of the gradient vector (or residual) resulting from the use of a
few Lanczos iterations.

\begin{algorithm}[ht]
  \centering
  \caption{Lanczos Based Acceleration (LBA)}\label{alg:LBA}
  \begin{algorithmic}[1]
    \State Start: a given initial guess $x_0$; set $r_0=b-Ax_0$ and
    $p_0=Ar_0$; a number of Lanczos steps $m>0$;
  \State Choose: $\sigma$ such that $0 < \sigma < 1$,  and $\epsilon_{\rm eig} < 1$
  \While{(not converged)}
  \State $\ds\alpha_k = \frac{{p_k^\top r_k}}{{p_k^\top p_k}}$
  \If {$\frac{\alpha_k}{\|r_k\|}\Big\|p_k - \alpha_k^{-1}r_k\Big\| <
      \epsilon_{\rm eig}$}
 \Comment{Test whether $r_k = -g_k$ is aligned with an eigendirection\\{\hspace{7.5cm} (recall that $p_k = Ar_k$).}}
  \State $V_m,T_m = {\rm Lanczos}(A, r_k, m)$
  \Comment{If yes, perform $m$ steps of Lanczos and update.}
  \State Solve $y_m = {\rm argmin}_{z\in\R^m} \norm{r_k - AV_m z}$
  \State $x_{k+1} = x_k + V_m y_m$
  \State $r_{k+1} = r_k - AV_m y_m$
  \Else \Comment{If not, use steplength $\sigma\alpha_k$ to break the zigzag effect.}
  \State $x_{k+1} = x_k + \sigma\alpha_k r_k$
  \State $r_{k+1} = r_k - \sigma\alpha_k p_k$
\EndIf
 \State $p_k = Ar_{k+1}$
\EndWhile
\end{algorithmic}
\end{algorithm}

\begin{remark} \label{reductres}
  It is clear from the way in which  we define the iterate $x_{k+1} = x_k+
  V_my_m$ generated by the Lanczos process, that the residual $r_{k+1} = r_k -
  AV_my_m$ satisfies $\|r_{k+1}\| \le \|r_k\|$. In fact we have $\|r_{k+1}\|^2 =
  \|r_k\|^2 - \|A V_my_m \|^2$ because the optimality of $y_m$ is equivalent to
  the orthogonality of  $r_{k+1}$ to the span of $AV_m$. This projection step is
  actually  equivalent to a GMRES projection and the solution $y_m$ to the
  minimization problem in Line~8 of Algorithm~\ref{alg:LBA} can be easily
  computed by solving a small least-squares problem, see for instance
  ~\cite[Section 6.5]{Saad-book2}.
  Moreover, by optimality of $y_m$ and since the first column of $V_m$ is
    proportional to the residual $r_k=-g_k$, it holds, after an application of the
    Lanczos acceleration,
    \[
      \norm{g_{k+1}}^2 \leq \norm{\tilde{g}_{k+1}}^2 \leq c(\sigma)
      \norm{g_{k}}^2
    \]
    where $\tilde{g}_{k+1}$ is the gradient after a standard MR step (no
    acceleration) and $c(\sigma)$ is the constant from Theorem~\ref{convMR}.
\end{remark}

\begin{remark}\label{rem:Impl}
As was observed in the previous remark, the update in Line~9, which is of the form
$x_{k+1} = x_k + \delta_k$, can be obtained from $m$ steps of the GMRES algorithm for solving
the system $A \delta =r_k$ starting with a zero initial vector.  However, since $A$ is Hermitian a short-term recurrence implementation called the Conjugate Residual (CR) method, can instead be invoked, see e.g.,
\cite[Sec. 6.8]{Saad-book2}. The CR algorithm is mathematically equivalent to GMRES when $A$ is Hermitian but in contrast to the implementation of the above algorithm,   it requires storing  five vectors in all.
\end{remark}

\begin{remark}[Case $m=1$] In the case of a single Lanczos step $m=1$, the computation of $y_m$ falls down to the minimization of $\|r_k - zAr_k\|$ over $z\in\R$, that is performing a standard MR step $r_{k+1} = r_k - \alpha_kp_k$. Since this is used only if the criterion on Line 5 of Algorithm~\ref{alg:LBA} is triggered, this implies that LBA with $m=1$ is actually equivalent to the eigenvector acceleration scheme from Algorithm~\ref{alg:EigAcc}.
\end{remark}

\begin{corollary} \label{convLanczAcc}
  Let the conditions of Theorem \ref{convMR} hold.  Then the sequence
  $\{x_k\}_{k\in \N}$ generated by the Lanczos-based acceleration algorithm
  converges to $x^*= A^{-1}b$.
\end{corollary}
\begin{proof}
  Using a fixed  $\sigma \in (0, 1)$ the  Lanczos-based  acceleration algorithm
  produces a subsequence $K_1\subset \N$ for which only the relaxation parameter
  $\sigma$ is used to generate the iterates. There exists a  complementary
  subsequence $K_2\subset \N$ ($K_1 \cup K_2 = \N$) for which the Lanczos method
  is activated to generate the next iterate. However, from Theorem  \ref{convMR}
  and Remark \ref{reductres} we have that $\|g_{k+1}\| \leq \sqrt{c(\sigma)}
  \|g_{k}\|$ for all $k\in K_1$ or $k\in K_2$.
  Consequently,   the whole sequence $\{\|g_{k}\|\}_{k\in \N}$   converges to
  zero, which  implies  by continuity that $\{x_k\}_{k\in \N}$  converges to
  $x^*= A^{-1}b$.
\end{proof}

We now illustrate in Figure \ref{fig2} the behavior of the Lanczos-based
acceleration, on the same problem described in Figure~\ref{fig1}, when compared
with the eigenvector-based acceleration and also with the MR method cleared from
the zigzag effect. This time, even when breaking the zigzag effect, the impact
of the few Lanczos steps when the residual aligns with an eigenvector gives a
significant faster convergence. Obviously, one should keep in mind that every
$m$ Lanczos steps require $m+1$ additional matrix-vector products when triggered
but the overall performance is still better (in Algorithm~\ref{alg:LBA},
counting only matrix-vector products when the matrix is $n\times n$, there are
$m$ products in Line 7 and 1 in Line 10).

\begin{figure}[h!]
  \centering
    \begin{minipage}{0.49\textwidth}
      \centering
      \includegraphics[width=0.8\linewidth]{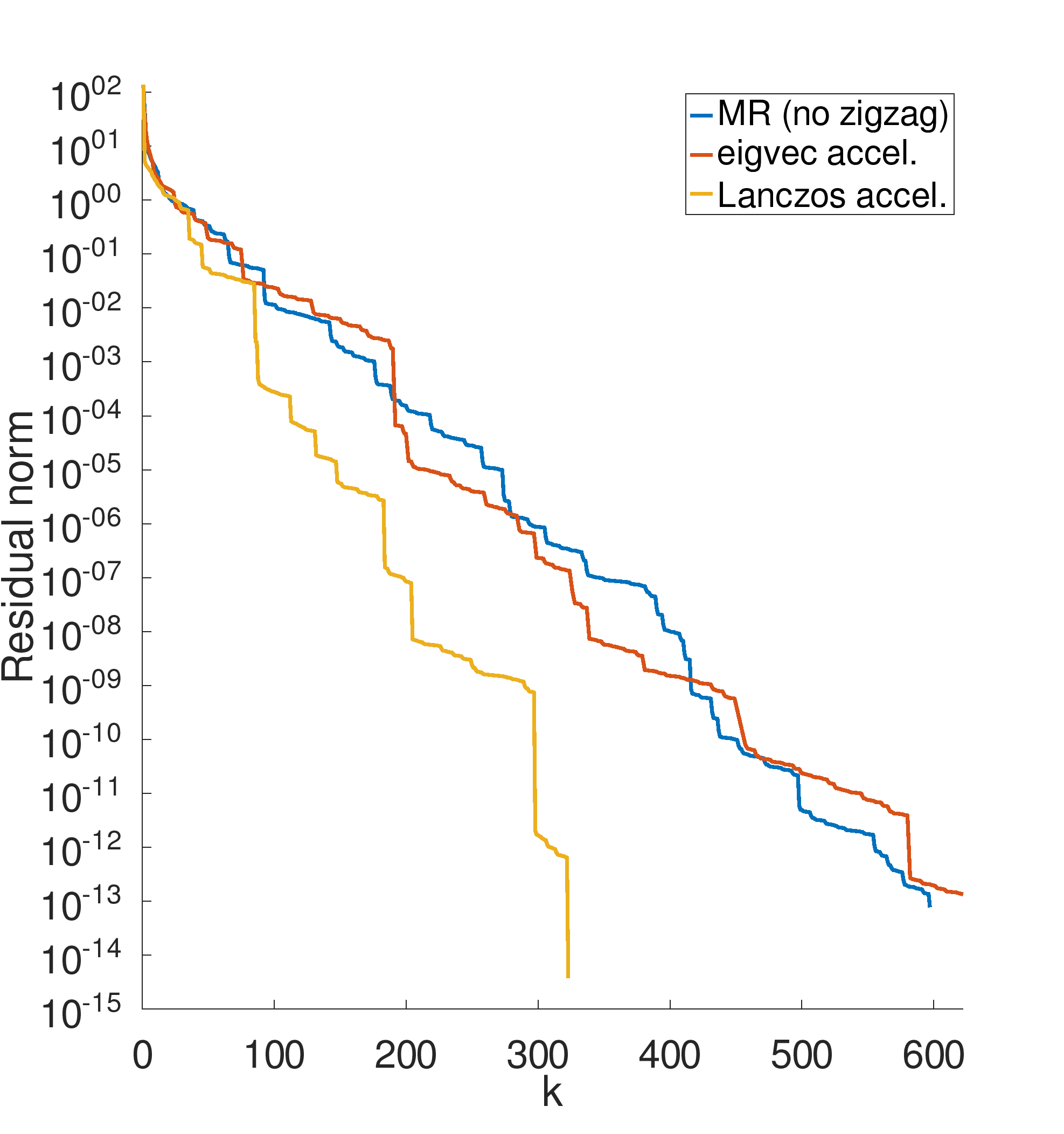}
    \end{minipage}\hfill
    \begin{minipage}{0.49\textwidth}
    \centering
    \begin{tabular}{@{}lcc@{}}
        \toprule
                             & matvec  & time (s)      \\ \midrule
        MR (no zigzag)       & 597     & $3.90\cdot10^{-2}$ \\
        Eigenvector accel.   & 622     & $3.84\cdot10^{-2}$ \\
        Lanczos-based accel. & 396     & $2.42\cdot10^{-2}$ \\ \bottomrule
    \end{tabular}
    \end{minipage}

  \caption{Lanczos based acceleration ($\sigma=0.8$, $\epsilon_{\rm eig} =
    0.8$): using a few steps ($m=5$) of the Lanczos algorithm  makes the
    acceleration process more efficient, both in terms of iterations (left),
    number of matrix-vector products and CPU time (right).
    Note that, for the Lanczos-based
    acceleration, the number of matrix-vector products is higher than the number
    of iterations because every time a Lanczos projection is performed,
    additional matrix-vector products are used.} \label{fig2}
\end{figure}


\subsection{Adaptive Lanczos projection}  \label{adplncz}

In the previous section, a fixed number $m$ of steps were performed in the
Lanczos projection. This might not be optimal since (i) the user does not have
control on the residual norm reduction performed by the projection and (ii) in
the early stages of the convergence when the current iterate is far from the
solution, it might not be worth performing all the Lanczos projection steps. One
possibility is to use instead an adaptive strategy where, at every steps during
the Lanczos algorithm, we check if the current residual is small enough: if
$\norm{r_k - AV_my_m} \leq \texttt{reltol}\norm{r_k}$, then we stop the Lanczos
projection at this point and can use $y_m$ for the acceleration step as in
\eqref{eq:LSpb}, otherwise we perform an additional Lanczos step. The parameter
\texttt{reltol} can then be freely chosen by the user, either as a constant
smaller than 1 or a value that depends on the residual norm\footnote{Similarly
  to inexact Newton methods for instance.}. This adaptive Lanczos projection is
summarized below in Algorithm~\ref{alg:LBA_adapt}.

\begin{algorithm}[ht]
  \centering
  \caption{Adaptive Lanczos}\label{alg:LBA_adapt}
  \begin{algorithmic}[1]
    \State Start: a given initial guess $r$; the matrix $A$; a relative tolerance $\texttt{reltol}$ and a maximum number of Lanczos steps $m$
    \State Initialize matrices $V_0$ and $H_0$ as in Lanczos
  \For{$i = 1\dots m$}
  \State Compute $V_i$ and $H_i$ from $V_{i-1}$ and $H_{i-1}$ with one Lanczos step
  \State Solve $y_i = {\rm argmin}_{z\in\R^i} \norm{r - AV_iz}$
  \Comment{This is done as explained in Remark~\ref{rmk:givens} below.}
  \If {$\norm{r - AV_iy_i} \leq {\texttt{reltol}}\times\norm{r}$}
  \State break
  \EndIf
\EndFor
\end{algorithmic}
\end{algorithm}

\begin{remark}[Givens rotations]\label{rmk:givens}
For practical implementations of Algorithm~\ref{alg:LBA_adapt}, one actually
only needs to update the residual until sufficient decrease is reached and
compute the solution to the least-square problem. Updating the residual
iteratively can be done using Givens rotations, similarly to what is done in
practical implementations of GMRES, see \cite[Section 6.5.3]{Saad-book2}.
\end{remark}

We compare this idea to having a fixed number $m$ of Lanczos steps and we
present the results in Figure~\ref{fig:adapt_Lanczos}. In
all cases, we stop the process when $\|g_k\| \leq \texttt{tol} \|g_0\|$ where
$\texttt{tol} = 10^{-15}$. The conclusions are quite
striking, both with a fixed number of iterations and with the adaptive version
of the Lanczos algorithm where \texttt{reltol} is fixed.
As expected, as we increase $m$ or tighten the
relative tolerance, the number of iterations decreases. Still, one has to pay
attention to the total number of matrix-vector products. More precisely,
regarding the LBA algorithm with a fixed number of Lanczos steps during the
speed-up phase, we see that the minimum number of matrix-vector products is
achieved for $m=10$. We decided not to further increase $m$ to keep the
resolution and storage cost of the least-squares problems negligible. The
adaptive Lanczos with a (fixed) tight relative tolerance of $10^{-2}$
gives an algorithm that converges similarly.
Note however that, in this case, we fixed the maximum
number of Lanczos steps to $m=10$ and that it is reached almost everytime the
projection is triggered. Finally, we ran the same experiment with an adaptive
Lanczos projection where \texttt{reltol} is some power of the residual. This time, the power of the residual does not have much influence on the number of iterations and all the choices lead to sharp acceleration.
As a conclusion, with proper choice of stopping criterion
for the Lanczos projection, all these experiments yield better convergence,
both in terms of iterations and matrix-vector products, than both the standard
relaxed MR algorithm and the eigenvector-based acceleration.

Finally, as for computation times, the LBA algorithm with fixed $m$ clearly
performs better than the relaxed MR algorithm. The adaptive strategies are not
competitive timing-wise, but this is actually due to the iterative updates of
the residuals which are not negligible for matrices of size $900\times900$. We
ran the same experiments on larger matrices and we could observe a clear
advantage for adaptive LBA algorithms against relaxed MR iterations. For
instance, when  $A$ is larger, of  size $2,25\cdot10^4\times2,25\cdot10^4$, the
relaxed MR converges in about $2.9$ s \emph{vs} $1.4$ s for the adaptive LBA
with \texttt{reltol} $= 0.05$.

\begin{figure}[p!]
    \begin{minipage}{0.4\textwidth}
      \centering
      \includegraphics[width=0.9\linewidth]{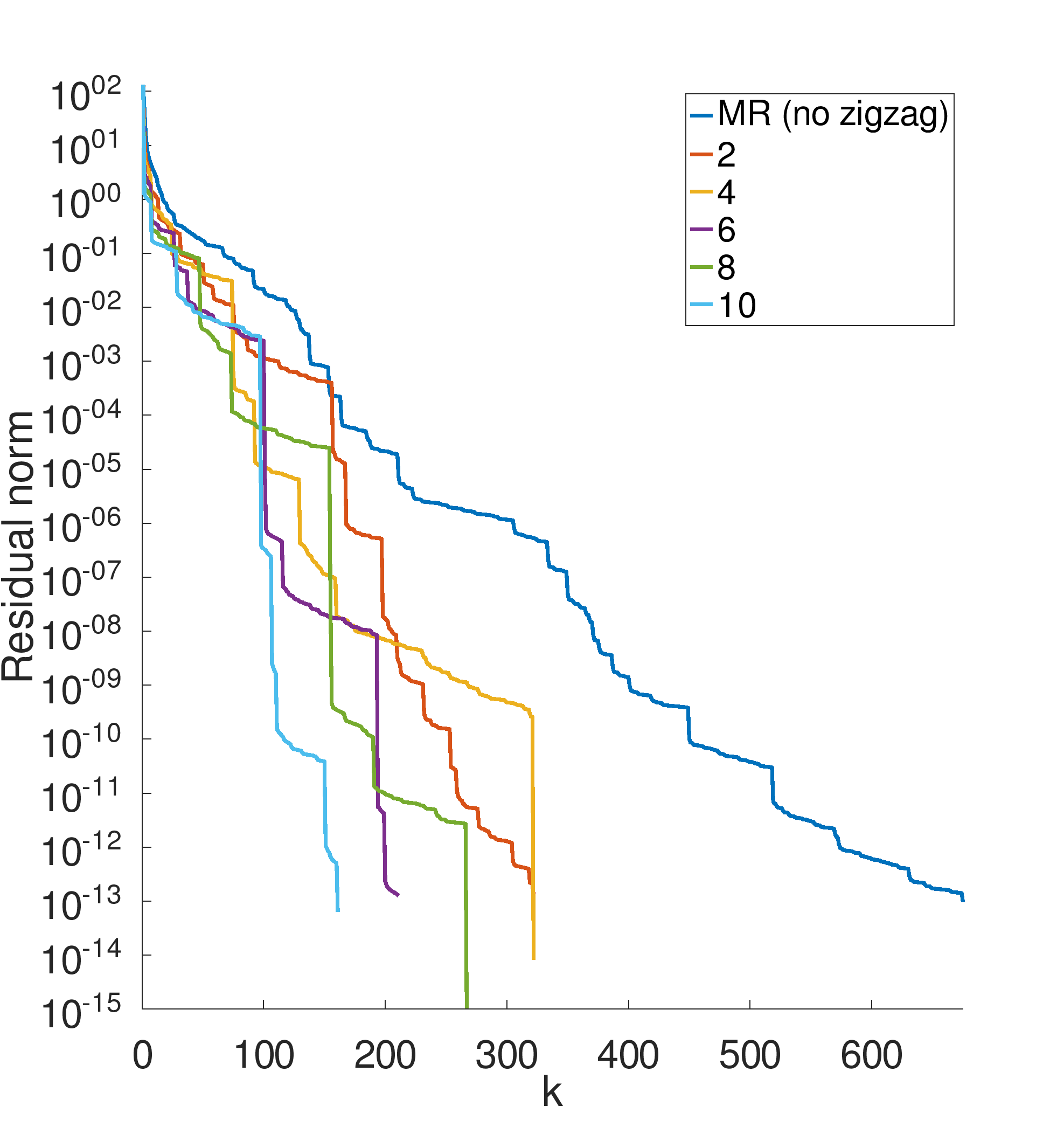}
    \end{minipage}\hfill
    \begin{minipage}{0.59\textwidth}
    \centering
    \begin{tabular}{ccccc}
        \toprule
        \multicolumn{4}{c}{fixed}                    \\
        $m$ & iterations & matvec & time (s) & Lanczos calls \\ \midrule
        2   & 321       & 388     & $2.94\cdot10^{-2}$ & 22              \\
        4   & 321       & 367     & $2.75\cdot10^{-2}$ & 9               \\
        6   & 210       & 274     & $1.84\cdot10^{-2}$ & 9               \\
        8   & 266       & 339     & $1.98\cdot10^{-2}$ & 8               \\
        10  & 160       & 249     & $1.68\cdot10^{-2}$ & 8               \\ \bottomrule
    \end{tabular}
    \end{minipage}

    \begin{minipage}{0.4\textwidth}
      \centering
      \includegraphics[width=0.9\linewidth]{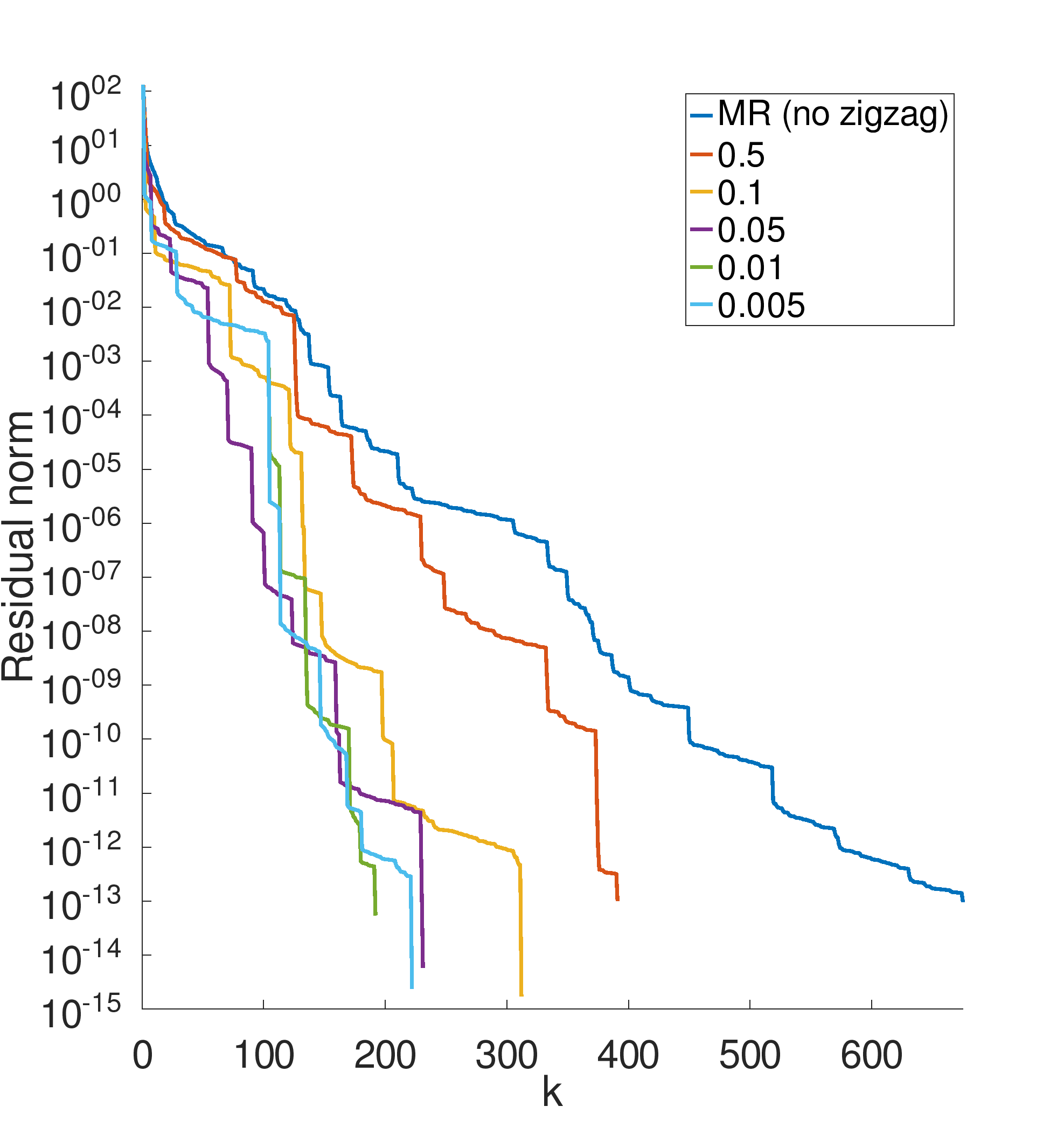}
    \end{minipage}\hfill
    \begin{minipage}{0.59\textwidth}
    \centering
    \begin{tabular}{ccccc}
        \toprule
        \multicolumn{4}{c}{adaptive}                             \\
        \texttt{reltol} & iterations & matvec & time (s) & Lanczos calls \\ \midrule
        $5\cdot10^{-1}$ & 390        & 446    & $4.16\cdot10^{-2}$ & 19               \\
        $10^{-1}$       & 311        & 385    & $3.89\cdot10^{-2}$ & 11              \\
        $5\cdot10^{-2}$ & 230        & 345    & $4.51\cdot10^{-2}$ & 12                \\
        $10^{-2}$       & 191        & 297    & $4.18\cdot10^{-2}$ & 10               \\
        $5\cdot10^{-3}$ & 221        & 318    & $4.68\cdot10^{-2}$ & 9                \\ \bottomrule
    \end{tabular}
    \end{minipage}

    \begin{minipage}{0.4\textwidth}
      \centering
      \includegraphics[width=0.9\linewidth]{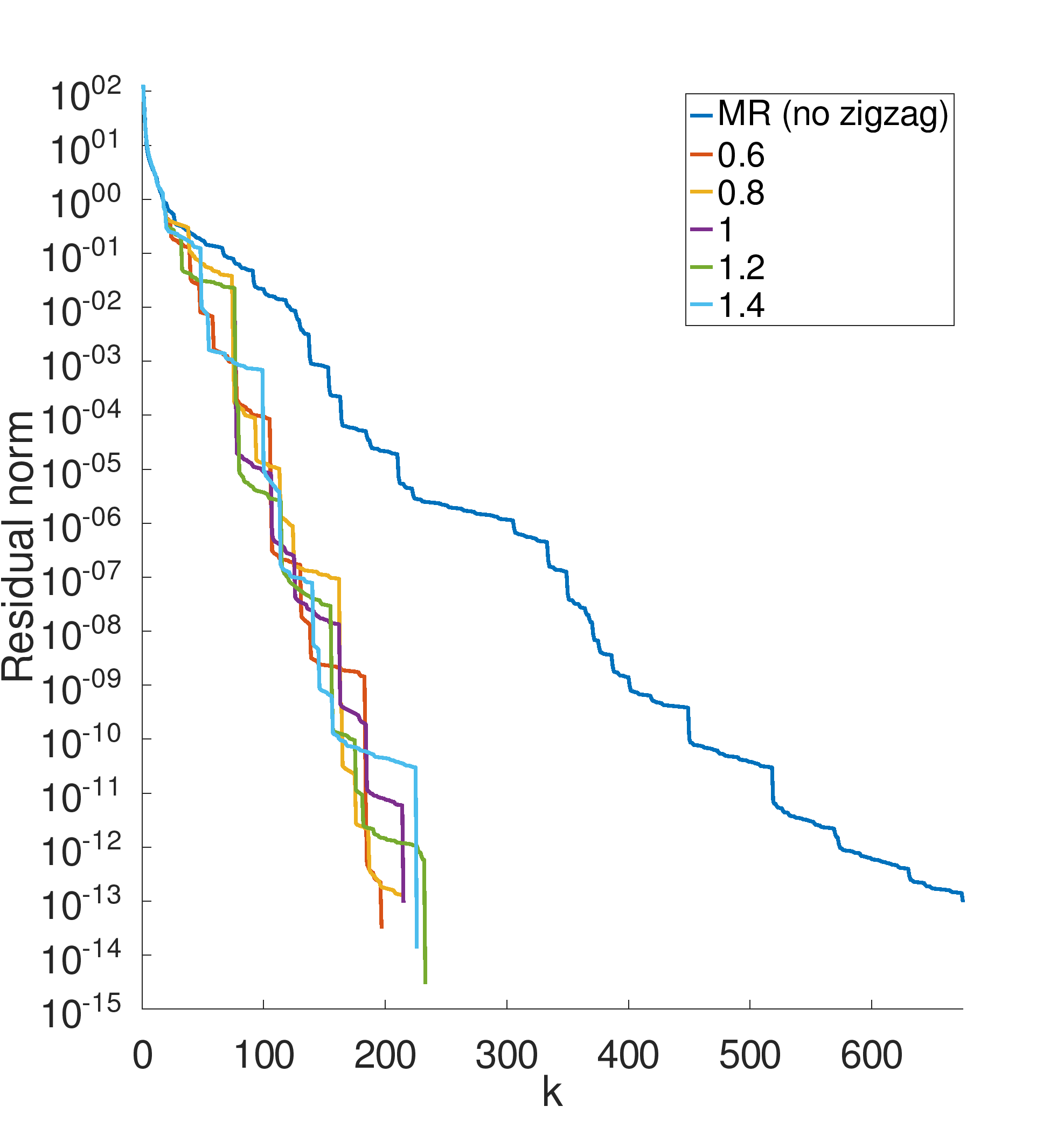}
    \end{minipage}\hfill
    \begin{minipage}{0.59\textwidth}
    \centering
    \begin{tabular}{ccccc}
        \toprule
        \multicolumn{4}{c}{adaptive}                                    \\
        \texttt{reltol}       & iterations & matvec & time (s) & Lanczos calls \\ \midrule
        $\norm{r_k}^{0.6}$ & 196        & 338    & $5.43\cdot10^{-2}$ & 20               \\
        $\norm{r_k}^{0.8}$ & 213        & 333    & $5.01\cdot10^{-2}$ & 18               \\
        $\norm{r_k}^{1}$   & 214        & 319    & $6.63\cdot10^{-2}$ & 16               \\
        $\norm{r_k}^{1.2}$ & 232        & 349    & $5.03\cdot10^{-2}$ & 17               \\
        $\norm{r_k}^{1.4}$ & 225        & 333    & $4.63\cdot10^{-2}$ & 16              \\ \bottomrule
    \end{tabular}
    \end{minipage}

    \caption{Lanczos-based acceleration vs MR  (no zigzag, in dark blue and with
      reference 675 matrix-vector products and a CPU time of $\approx 3.26\cdot10^{-2}$ s)
      for different cases. We used $\sigma=0.8$ and $\epsilon_{\rm eig}=0.8$.
      (Top) Fixed number $m$ of Lanczos iterations when triggered.
      (Middle) Adaptive Lanczos projection with fixed parameter \texttt{reltol}.
      (Bottom) Adaptive Lanczos projection with \texttt{reltol} as some power of the residual.
      In the two adaptive versions, the Lanczos projections stop if $m=10$ is reached.}
    \label{fig:adapt_Lanczos}
\end{figure}

\subsection{An analogy with algebraic multigrid methods}

\textbf{Theoretical comparison}
At this point we can draw an analogy between the Lanczos-based acceleration algorithm (LBA) and the basic two-grid V-cycle algebraic multigrid method (AMG) in the quadratic case:

\begin{itemize}
	\item  AMG combines approximations of the solution in spaces of different dimensions. When considering two subspaces which support the approximations of the solution, the fine space $V_f$ is associated to a large dimensional approximation space, while the coarse space $V_c$ is associated to a small dimensional approximation space. The basic principle of AMG consists in computing the residual $r_f$ after some fixed iterations of a pre-smoother (typically the damped Jacobi method -- DJM) in $V_f$, and then to restrict it to $V_c$ as $r_c=I_f^cr_h$. The error $e_c$ in $V_c$ solves the linear system restricted to $V_c$: $A_ce_c=r_c$ which is solved exactly. The next approximation of the solution is then updated by adding the extension $e_f$ of $e_c$ to $V_f$ as $e_f=I_c^fe_c$, and an additional fixed number of iterations a post-smoother (DJM) is done. We take $I_c^f=(I_f^c)^\top$ and a new cycle can begin. We refer to  \cite{xu17} for further details.
    \item LBA combines the approximation of the solution of the linear system on the ambient space and a correction which consists in an extension of an approximation of the current residual in a small dimensional space (obtained in the Lanczos process  by few iterations of Arnoldi method). The two steps are applied alternatively (as a V-Cycle), the Lanczos acceleration is activated not after a fixed number of MR iterations but after a criterion is satisfied.
\end{itemize}

We summarize in a nutshell the parallel between the two methods in Table~\ref{tab:AMG}. Moreover, we presented above different versions of LBA, from the fixed number $m$ of iterations at each Lanczos acceleration call to the various adaptive LBA with a limited $m$ (generally $\le 10$). As an illustration of the corresponding multigrid cycle, Figure~\ref{MGcycle} (left) hereafter represents schematically an history of the MG cycle when considering the original LBA algorithm. In Figure~\ref{MGcycle} (right), we present similarly the MG cycle associated to an adaptive version of LBA where this time the dimension of the Lanczos spaces can vary. Notice that the way the coarse spaces are defined differs from a method to the other: in AMG $V_c$ is deduced from $V_f$ by a coarsening procedure based on matrix graph technique, associating neighbor vertices, and coupled vertices, and its dimension is fixed from the begining ; in LBA, the coarse space changes at each outer iteration, say at each application of the Lanczos process, and its dimension can vary, according to the chosen strategy (number of Lanczos iterations, adaptive or non adaptive stopping criteria).

\vfill
\begin{table}[h!]
\begin{center}
\begin{tabular}{lll}
	\toprule
	Step & AMG&LBA\\
	\midrule
	Pre-smoothing &$u\coloneqq \nu\mbox{ iterations (DJM)}$ & $u\coloneqq \mu_1 \mbox{ iterations (MR)}$\\
	Compute residual & $r\coloneqq b-Au$&$r\coloneqq b-Au$\\
	Approximate error $e_H$
	& $e_c\coloneqq A_c^{-1}I_f^cr$ & $V,H={\rm Lanczos}(r,A,m)$\\
	Extend error to large dim space&
	$e_f\coloneqq I_c^fe_c$ & $e\coloneqq Vy$\\
	Correction & $u\coloneqq u+e_f$ & $u\coloneqq u+Vy$\\
	Post-smoothing& $u\coloneqq \nu\mbox{ iterations (DJM)}$ &$ u\coloneqq\mu_2\mbox{ iterations (MR)}$\\
	\bottomrule
\end{tabular}
\end{center}
\caption{Analogy between algebraic multigrid and Lanczos-based acceleration.}\label{tab:AMG}
\end{table}
\vfill

\begin{figure}[h!]
    \centering
\includegraphics[height=0.24\linewidth]{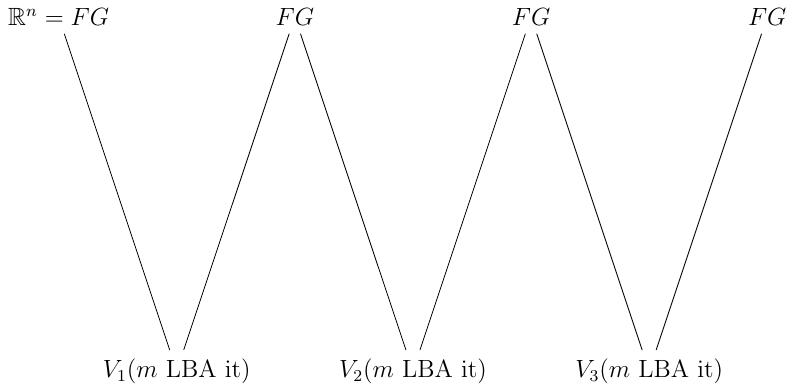}\hfill
\includegraphics[height=0.24\linewidth]{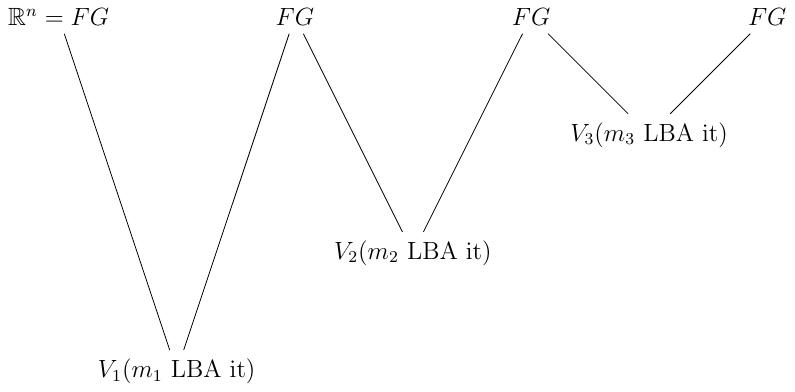}
\caption{(Left) Multigrid-like cycle for LBA: despite the dimension of the Lanczos spaces being the same at every activation of the acceleration, the underlying spaces are built from different residuals and are thus different. (Right) Multigrid-like cycle for an adaptive LBA: this time, the dimension of the Lanczos spaces may vary at each iteration.}
\label{MGcycle}
\end{figure}

\textbf{Numerical comparison}
In AMG, pre-smoothing and post-smoothing procedures are done by a fixed number $\nu$ of DJM iterations while, in LBA, it is realized with a variable number $\mu_1$ or $\mu_2$ of MR iterations cleared from the zigzag effect. A question that can be addressed is to explore the pre-smoothing and post-smoothing procedures of AMG with the coarsening and the low dimensional solution of the system provided by the Lanczos subspace acceleration used in LBA. We give an illustration of such a combination in Figure~\ref{figmglba} and Table~\ref{tab:mglba} on the same matrix $A$ as above (with random $b$): we compare MR without the zigzag effect, LBA ($m=10$) and a two-levels method\footnote{The appropriate term should actually be a bi-space method rather than a multigrid method} based on Lanczos acceleration on the low dimensional space, together with a pre-smoother and post-smoother on the whole space provided by DJM iterations with a random damping coefficient ($\beta>0$ is a free parameter):
\[
u \coloneqq u +\alpha {\rm diag}(A)^{-1}r \quad\text{where}\quad \alpha \text{ follows a uniform law on } [0,\beta].
\]
We observe that the acceleration criterion is still frequently activated when using a classical smoother such as DJM instead of relaxed MR and that the acceleration itself is attributable exclusively to the Lanczos subspace iterations. The values of the damping parameter $\alpha$ also matter: an optimal (in terms of matrix-vector products) value seems to be $\beta\approx7.0$. Even with such an optimal value, the LBA-DJM bi-space scheme performs only slightly better than the relaxed MR iterations (and can even fail to be monotone for $\beta$ too large). It never outperforms LBA-MR: based on additional numerical tests, we can say that this is due to (i) the DJM producing residuals supported on much more eigenmodes than relaxed MR and (ii) the fact that relaxed MR only requires one matrix-vector product per iteration (the activation criterion can be directly computed) while DJM requires two matrix-vector products by iteration (one to evaluate the activation criterion and one to update the residuals). It is therefore crucial to couple the Lanczos subspace acceleration with a proper smoother in order to obtain significant accelerations.

\vfill
\begin{table}[h!]
\centering
LBA-DJM ($m=10$)
\begin{tabular}{@{}ccccccc@{}}
\toprule
Value of $\beta$    & $0.5$       & $1.0$      & $5.0$     & $7.0$     & $10.0$     & $15.0$     \\ \midrule
iterations (matvec) & 492 (1175) & 251 (663) & 97 (505) & 89 (489) & 82 (495)  & 92 (585) \\
\midrule
time (s) & $7.48\cdot 10^{-2}$ & $4.31\cdot 10^{-2}$ & $3.62\cdot 10^{-2}$ & $3.57\cdot 10^{-2}$ & $3.54\cdot 10^{-2}$ & $4.21\cdot 10^{-2}$ \\
\bottomrule
\end{tabular}
\[ \text{Reference:} \quad \text{MR (no zigzag):}\; 674\ (675) \text{ in } 3.26\cdot10^{-2} \text{s}
\qquad \text{LBA-MR $(m=10)$:} \; 160\ (249) \text{ in } 1.68\cdot10^{-2} \text{s}\]
\caption{Comparison of LBA with DJM as smoother (LBA-DJM) and MR (no zigzag) or
  LBA with MR as smoother (LBA-MR). We used $\sigma=0.8$ and $\epsilon_{\rm
    eig}=0.8$.  In parenthesis is the number of matrix-vector products, which
may be larger than the actual number of iterations when Lanczos subspace
iterations are used.} \label{tab:mglba}
\end{table}
\vfill

\begin{figure}[p!]
    \centerline{{~}\hfill
    \includegraphics[width=0.4\linewidth]{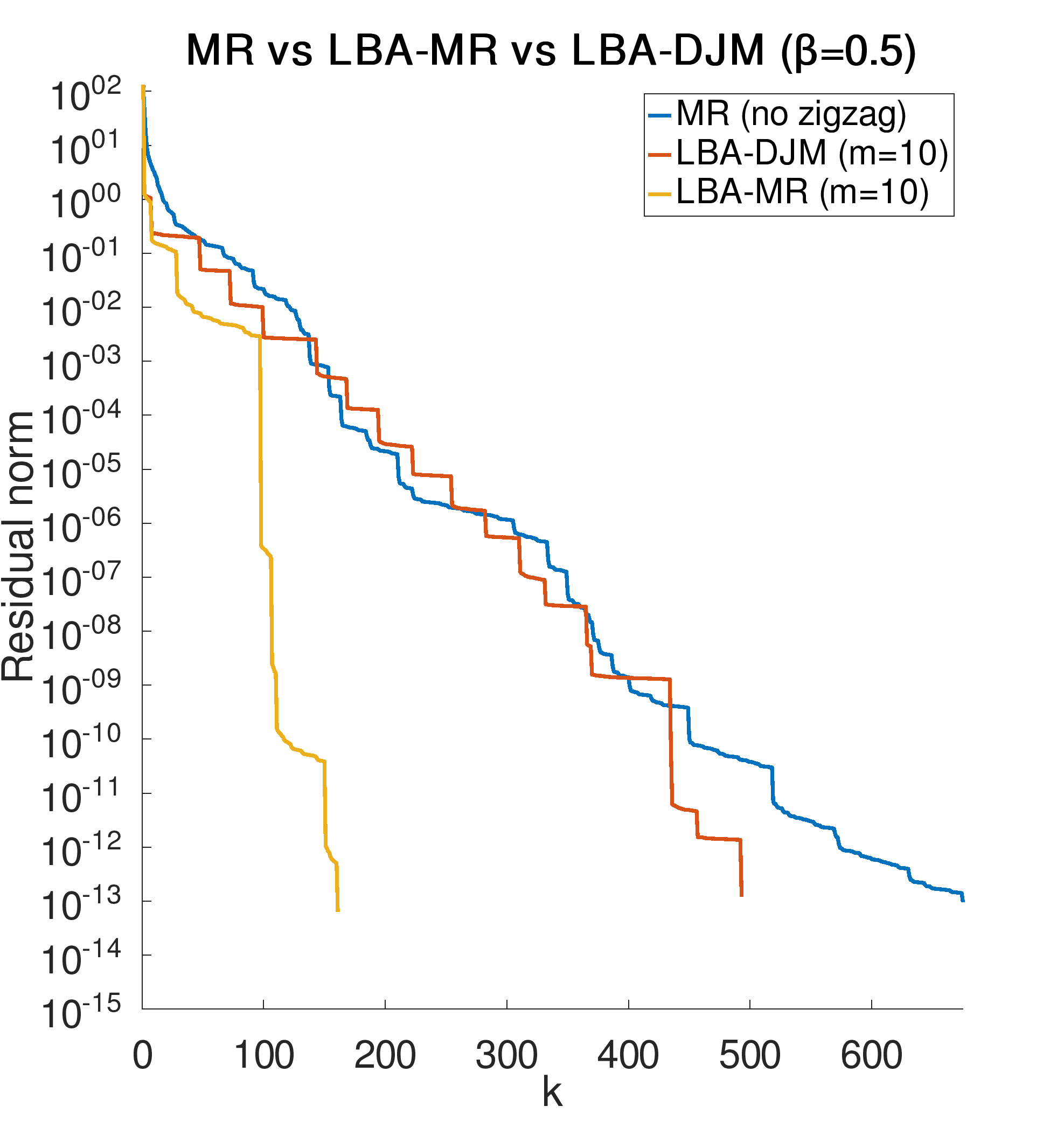}\hfill
    \includegraphics[width=0.4\linewidth]{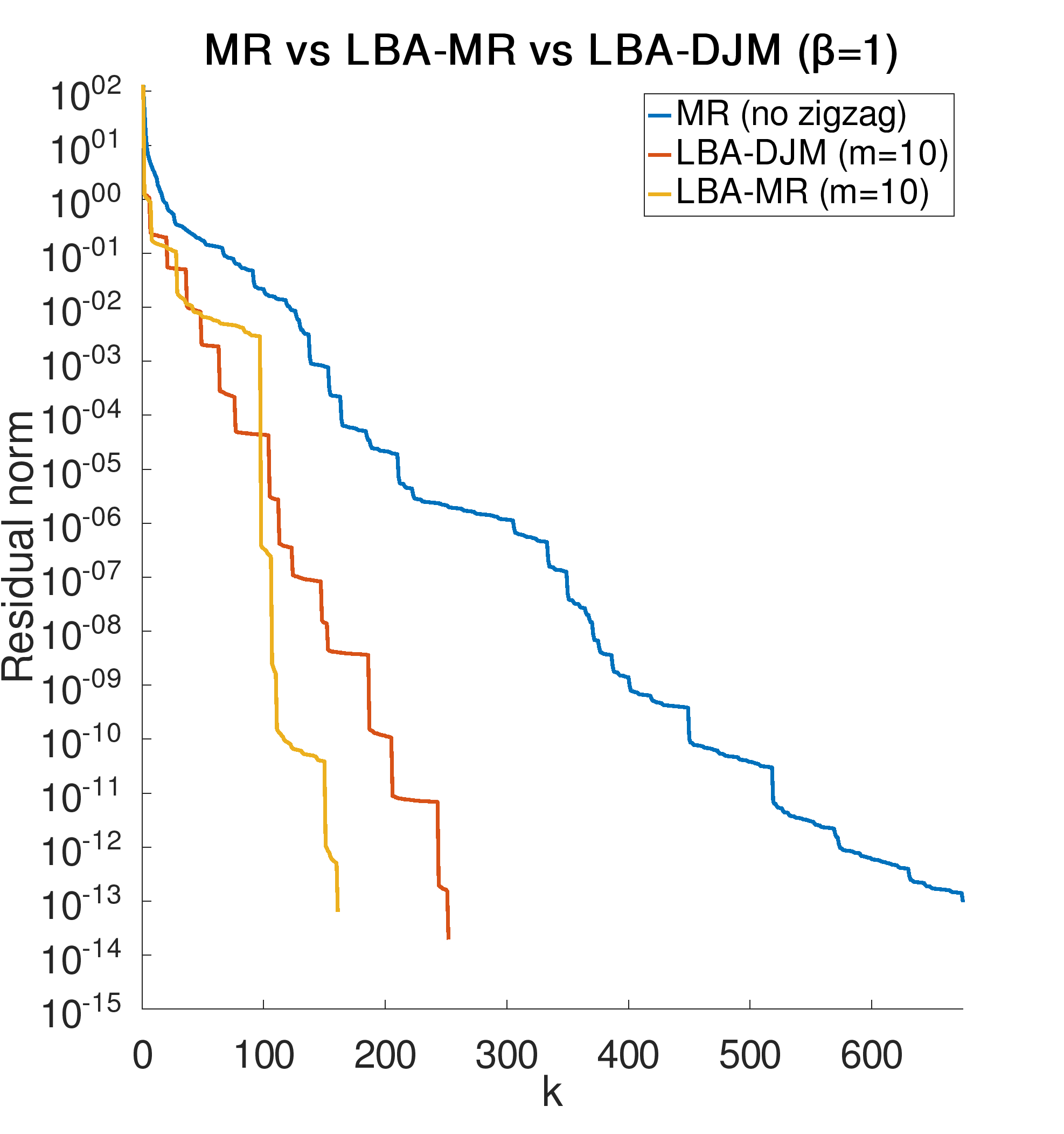}\hfill{~}}
    \centerline{\hfill
    \includegraphics[width=0.4\linewidth]{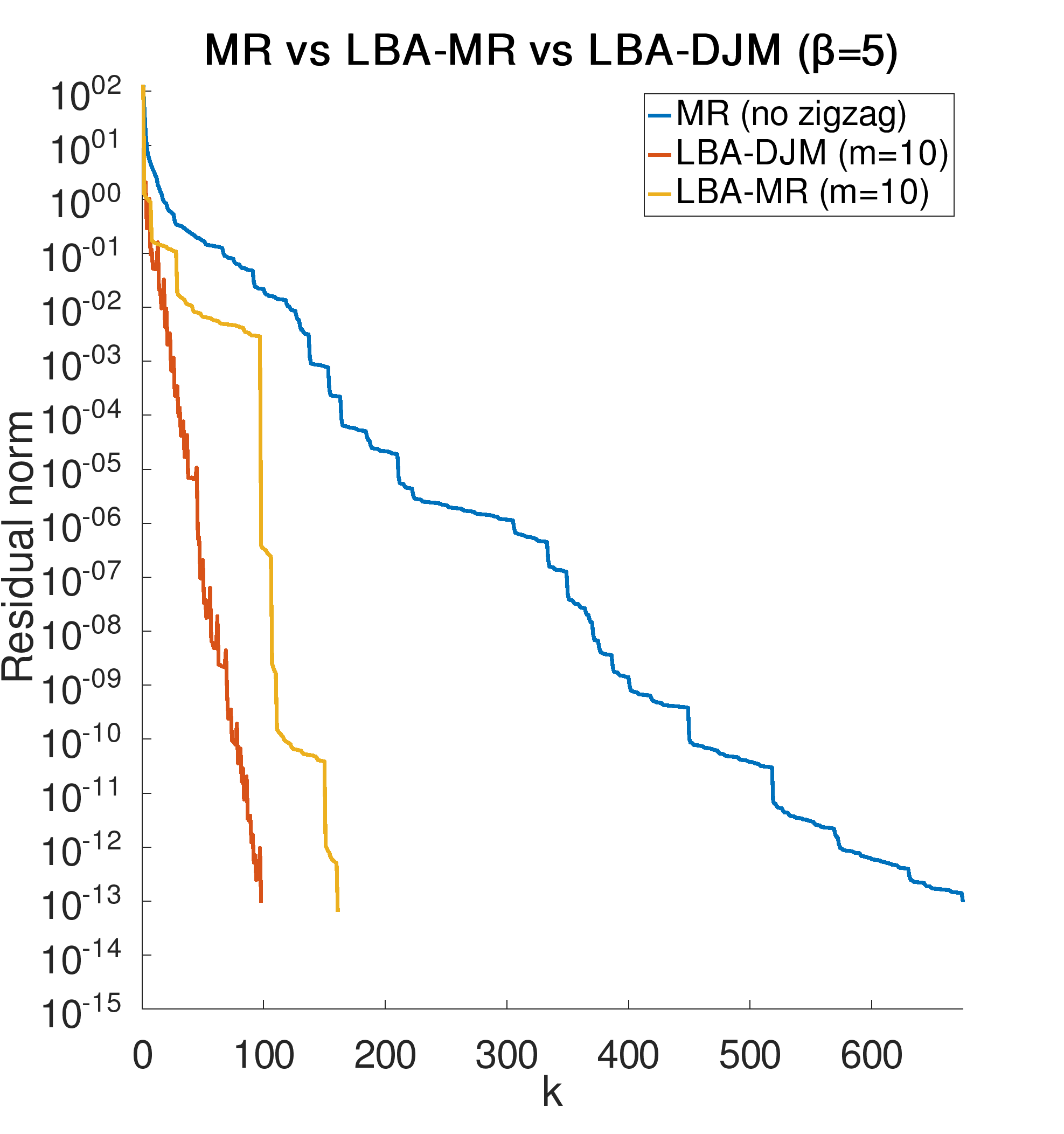}\hfill
    \includegraphics[width=0.4\linewidth]{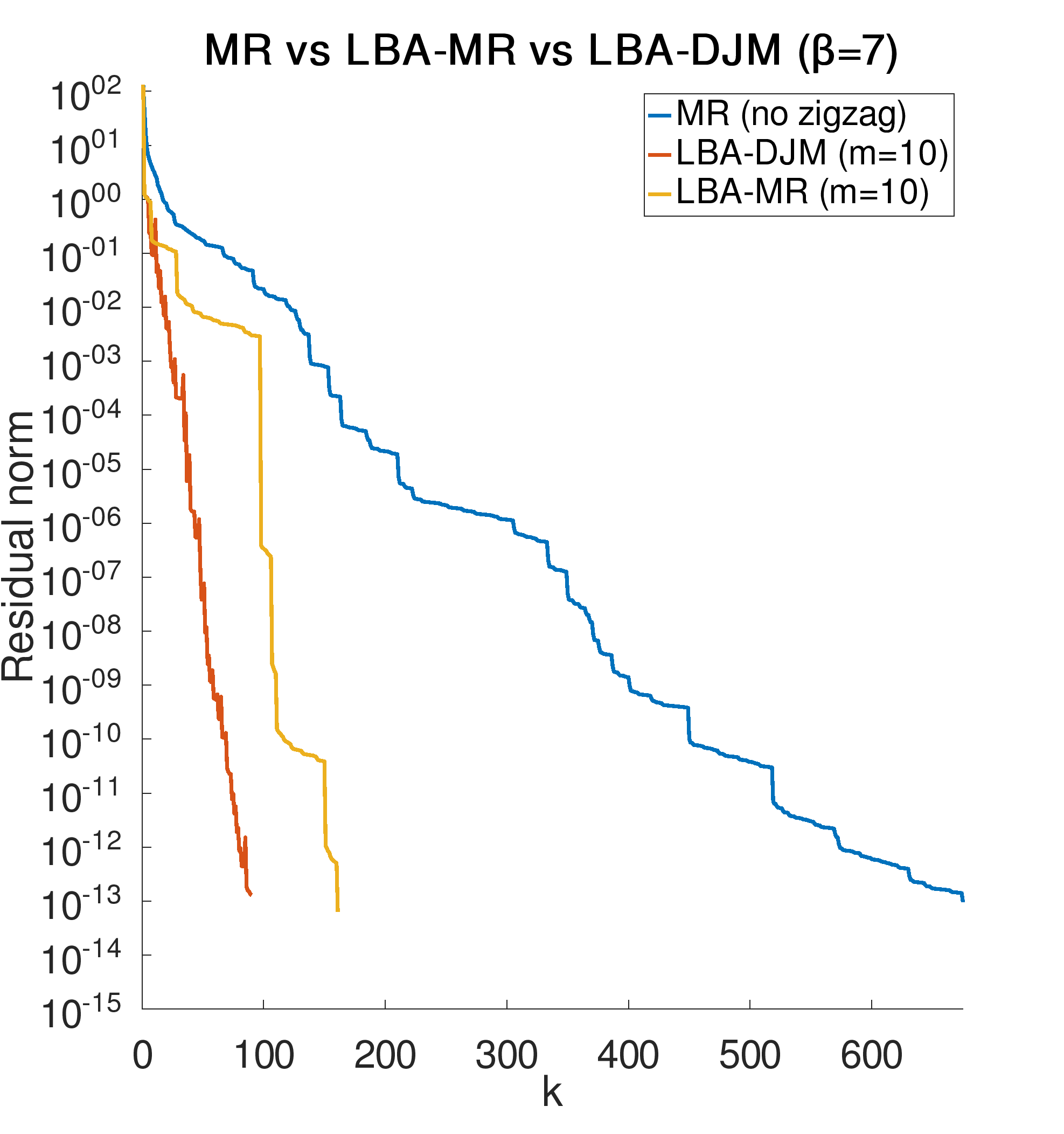}\hfill{~}}
    \centerline{\hfill
    \includegraphics[width=0.4\linewidth]{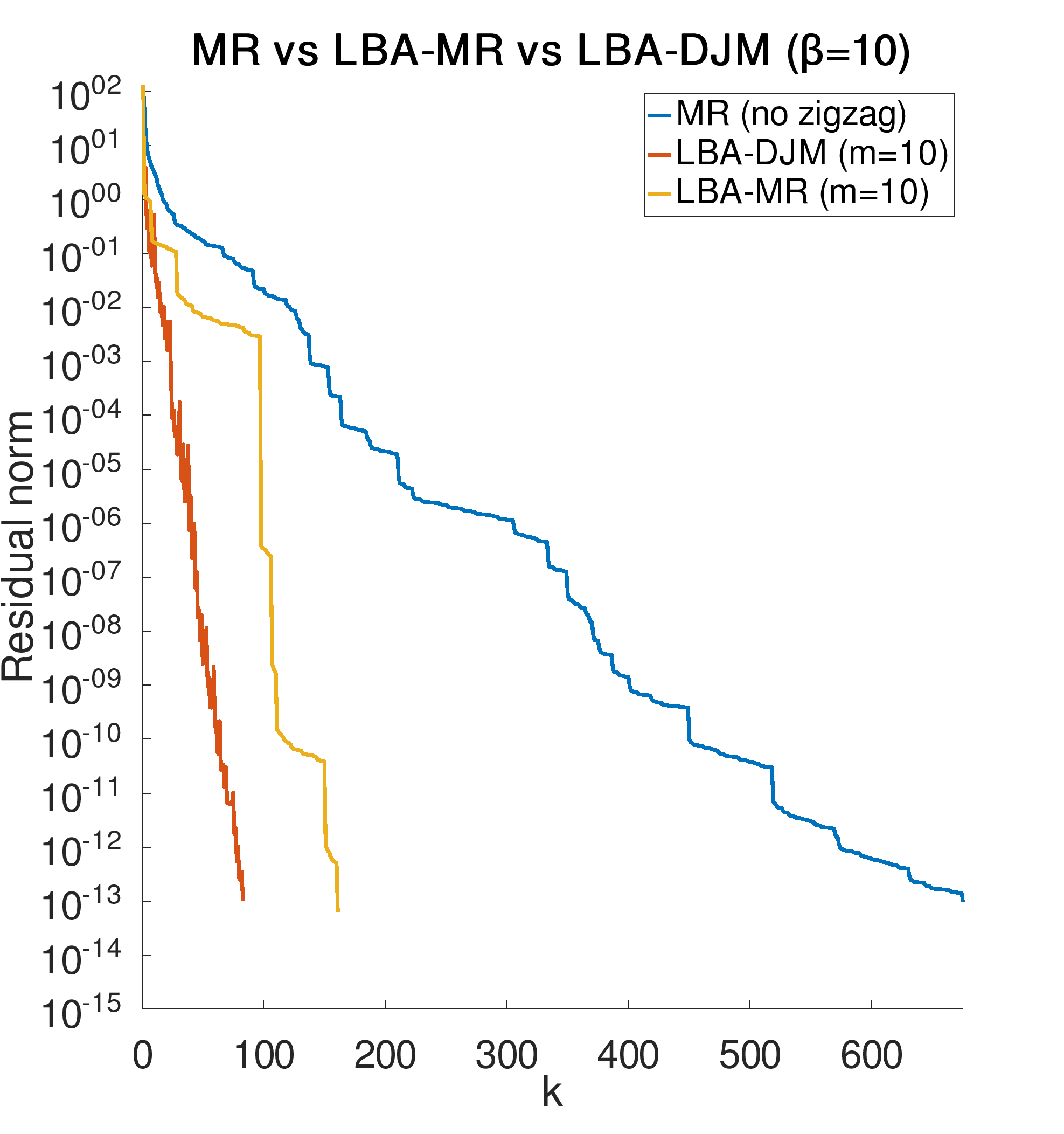}\hfill
    \includegraphics[width=0.4\linewidth]{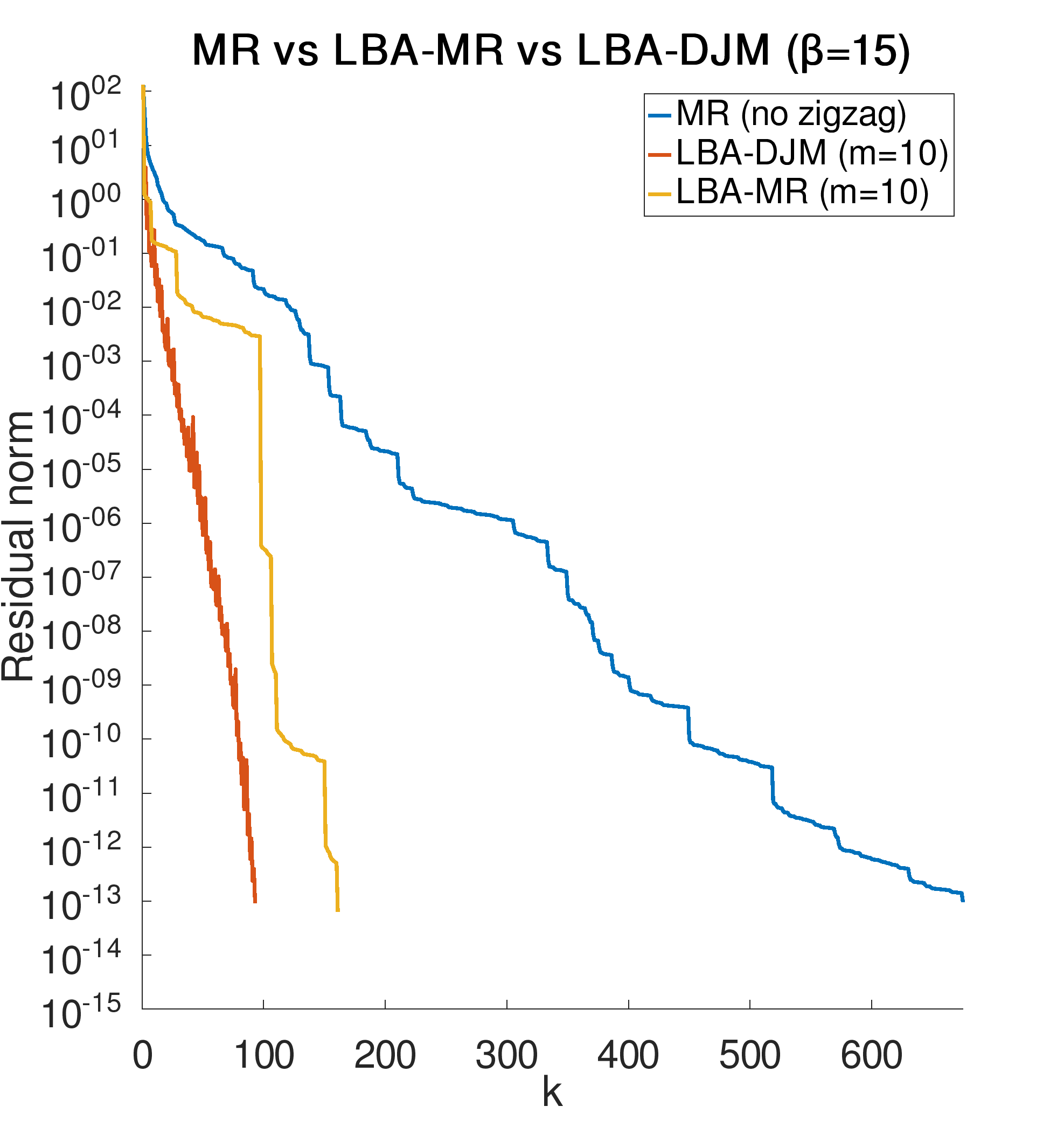}\hfill{~}}
    \caption{Comparison of the relaxed MR (no zigzag) with LBA ($\sigma=0.8$, $\epsilon_{\rm eig}=0.8$ and $m=10$) when combined with MR as smoother (LBA algorithm) and with DJM smoother, for different values of $\beta$.}\label{figmglba}
\end{figure}

\clearpage
\section{Extension to non-quadratic convex functions}\label{sec:convex}

\subsection{Incorporating line searches}

In the previous section, we introduced the Lanczos based acceleration method and
illustrated it on the minimization of quadratic functionals. However, for such
problems, it is well know that the golden standard is the Conjugate Gradient
(CG) algorithm. The CG method is optimal in the case of minimizing  strictly
convex quadratic functions \cite{Saad00},  but beyond that case and without a
constant Hessian, it loses its main powerful features. In the non-quadratic
case, specially for convex functions, gradient-type schemes retain their
simplicity and main convergence properties, which make them appealing for
large-scale minimization problems. Indeed,  if the cost function $f$ is twice
continuously differentiable, the eigenvector-based accelerations discussed and
analyzed in Section  3 can  easily be extended to non-quadratic convex
functions. The residual $r_k$ can be computed as $-\nabla f(x_k)$ and the
matrix-vector product $Ar_k$ that appears at each iteration $k$ of all the
considered iterative schemes can be obtained, in principle, as $-\nabla^2 f(x_k)
\nabla f(x_k)$, where $\nabla^2 f(x_k)$ is the  Hessian of $f$ at $x_k$. In
practice, the use of the exact Hessian of $f$ is in most cases not feasible and
not convenient in the large-scale case. Fortunately,  the Hessian matrix at
$x_k$ is only required to build the matrix-vector product $\nabla^2 f(x_k) r_k =
-\nabla^2 f(x_k)g_k= -H_k g_k$, which arises when calculating the
step length $\alpha_k$ and in the $m$ internal Lanczos iterations to obtain $d_k$ during
the few iterations where Lanczos acceleration is activated.  This matrix-vector product
can be obtained  with high numerical accuracy using a finite difference
approximation that only requires an additional gradient evaluation:
\begin{equation}\label{approxAv}
  H_k g_k \approx (\nabla f(x_k + h g_k) - g_k)/h,
\end{equation}
where $h>0$ is a small number. In practice, using the exact Hessian, $\nabla^2
f(x_k)$, or the finite difference expression in (\ref{approxAv}) produces quite
similar convergence behavior for smooth convex functions.

We focus our attention on the extension of the Minimal  Residual (MR) method. To
the best of our knowledge, the MR method has not been extended or adapted to
minimize non-quadratic convex functions. It is worth noting that the
gradient-type method that has been widely extended and adapted to solve general
large-scale minimization problems is the SD method \cite{bertsekas99}. These
extensions of SD typically require some form of step size control or a line
search.

We are thus interested in integrating the considered extensions of the MR method
into a globalization technique which is as tolerant as possible, allowing the
pure methods to act as freely as possible to maintain the effective accelerated
behavior observed in the convex quadratic case. At the same time we expect that
it affects  the pure methods as little as possible, to  ensure their convergence
to stationary points (i.e., global minimizers in the convex case). The most
suitable option for large-scale problems that satisfy these two opposing
requirements is the use of undemanding line search (LS) globalization
strategies.

Let us recall that, in the convex quadratic case, our schemes in Section 3
compute the steplength $\alpha_k$  carrying out an exact line search based on
closed formulas. However, the selection of an exact step-size in the general
non-quadratic case is not  feasible, due to the nonlinearity of  $\nabla f(x)$.
Therefore, we propose to equip the considered methods with an inexact line
search. Specifically,  we propose to adapt a tolerant non-monotone LS, discussed
and  analyzed in \cite{LaCruz}, that avoids unnecessary additional backtrackings
mainly at the initial iterations, while maintaining the convergence properties,
and tending to a monotonic behavior of the scheme at the final stage of the
iterative process. Roughly speaking, we require at each iteration of any of the
considered schemes that
\begin{equation} \label{roughLS}
  f(x_k + \sigma \alpha_k d_k)  \leq  f(x_{k}) - \gamma (\sigma\alpha_k)^2
  \|g_k\|^2 + \eta_k,
\end{equation}
where $\gamma$ is a small positive number, $\eta_k >0$ is chosen such that
$\sum_{k\in\N}\eta_k \leq \eta <\infty$, and $\sigma\alpha_k>0$ is the
steplength   used to obtain the next iterate $x_{k+1}$.

In most cases, the search direction $d_k$ is  chosen as $d_k= -g_k$. In a few cases,
$d_k = V_my_m$ (see Algorithm~\ref{alg:LBA}, step 9) when the Lanczos acceleration is activated.
 Note that in these cases, the steplength $\sigma \alpha_k$ is equal to 1 and the Hessian-times-vector
 approximation (\ref{approxAv}) is used internally in the Lanczos algorithm. Since this approximation
 poses the potential risk of the obtained vector $d_k$ being  either unbounded with respect to the gradient or
 not leading to sufficient decrease, it is advisable to add a safeguard control to the direction vector obtained by Lanczos.
 Specifically, we ensure that  for a fixed and very small $\tilde{\tau}>0$,
\begin{equation} \label{safedk}
  \|d_k\|_2 \le (1 / \tilde{\tau}) \|g_k\|_2\;  \mbox{ and }\; d_k^T g_k \leq -\tilde{\tau}\|g_k\|_2^2.
\end{equation}
 If (\ref{safedk})  is not satisfied, it is best to avoid setting $d_k = V_my_m$. Instead,
for safety, we set $d_k = -g_k$ (no zigzag MR) in this iteration as well.


To extend the methods presented in Section 3 to the non-quadratic convex case, the relaxed
parameter $\sigma$ is chosen in $(0,1)$, and the first trial for $\alpha_k>0$ is
obtained as
\begin{equation} \label{safeMG}
 \alpha_k = \min(\max(\tilde{\tau}, \tilde{\alpha_k}), 1/\tilde{\tau}),
\end{equation}
where  $\tilde{\alpha_k}$  is calculated using the MR  formula
\begin{equation} \label{trialMG}
  \tilde{\alpha_k} = \frac{g_k^\top H_k g_k}{ g_k^\top H_k^2 g_k}=\frac{g_k^\top (H_k
    g_k)}{ (H_k g_k)^\top (H_k g_k)}.
\end{equation}

Notice that  the matrix-vector product $H_k g_k$ is computed using
(\ref{approxAv}) only once at every $k$ and used three times in (\ref{trialMG}).
The safeguard strategy (\ref{safeMG}) guarantees that the step length
 $\alpha_k$  remains in a large positive interval away from zero and finite.
We also note that the  term $\eta_k > 0$ is responsible  for ensuring that
inequality (\ref{roughLS}) will hold regardless of the quality obtained using approximation
(\ref{approxAv}), and the  forcing term $- \gamma (\sigma\alpha_k)^2  \|g_k\|^2$ provides
the arguments for proving global convergence. In practice, to determine such a
step-size $\sigma\alpha_k$, we  use a well-known backtracking strategy, starting
with $\alpha_k$ given by (\ref{safeMG}-\ref{trialMG}), since this quotient exploits the local
information of the objective function. Therefore, the idea is to backtrack from
that step value in case it is necessary to satisfy (\ref{roughLS}).

The backtracking process is based on a safeguarded quadratic interpolation.
Given $0<\sigma_1 < \sigma_2 <1$ and setting $\widehat{d}_k = \sigma \alpha_k
d_k$,  the  safeguarding procedure acts when the minimum of the one-dimensional
quadratic $q(\cdot)$, such that $q(0)=f(x_k)$, $q(\beta)=f(x_k + \beta
\widehat{d}_k)$, and $q'(0)= \nabla f(x_k)^\top \widehat{d}_k$, lies outside
$[\sigma_1 \beta, \sigma_2 \beta]$, in which case  the more conservative
bisection strategy is preferred. Given a sequence $\{\eta_k\}_{k\in\N}$ such
that $\eta_k >0$ for all $k$ and $\sum_{k\in\N}\eta_k \leq \eta <\infty$, the
complete line search procedure is now described as:

\begin{algorithm}[H]
\centering
\caption{ Line search (backtracking)} \label{alg:LS}
\begin{algorithmic}[1]
    \State Input:  $\sigma \in (0,1)$, $0< \gamma <1$,  $x_k$, $\alpha_k$,  $g_k$, $d_k$, and  $\eta_k >0$
    \State Set $\widehat{d}_k = \sigma \alpha_k d_k$,   $x_+ = x_k + \widehat{d}_k$,    $\delta = g_k^\top\widehat{d}_k$, and  $ \beta = 1$
    \While{$( f(x_+) > f(x_k) - \gamma (\beta \sigma \alpha_k)^2 \|g_k\|^2 + \eta_k)$}
    \State $\beta_{\rm temp} = - \frac{1}{2} \beta^2 \delta / (f(x_+) - f(x_k) - \beta \delta)$
    \If {$\beta_{\rm temp} \ge \sigma_1 \beta \mbox{ and } \beta_{\rm temp} \le \sigma_2 \beta$}
    \State $\beta = \beta_{\rm temp}$
    \Else
    \State $\beta = \beta / 2$
    \EndIf
    \State  $ x_+ = x_k + \beta \widehat{d}_k$
    \EndWhile
    \State  Output: $\sigma\alpha_k \gets \beta \sigma \alpha_k\;$ \mbox{ and } $\;x_{k+1} \gets x_+$
\end{algorithmic}
\end{algorithm}

Regarding the algorithms in this section, we incorporate the LS (backtracking)
described above into each of the iterative schemes considered in Section 3 to
solve (\ref{genprb}) in the convex case. To be precise, at each algorithmic step
where a new iterate is computed, the candidate point passes through the LS just
before the assignment to obtain a step length $\sigma\alpha_k$  that guarantees
that (\ref{roughLS}) holds. When the assignment is of type $x_{k+1}= x_k - \sigma
\alpha_k g_k$, then $d_k = -g_k$ and the trial step length is $\sigma \alpha_k$.
When it is of type $x_{k+1}= x_k+V_m y_m$ then $d_k= V_m y_m$ and the trial step
length is set to 1.

Concerning the line search above, a few remarks are in order.
Once the step lengths $\alpha_k$   and directions $d_k$ are well-defined for all $k$,
 regardless of their quality,  the presence of the fixed and positive term $\eta_k$
   on the right side of  (\ref{roughLS}), and the continuity of the
  objective function $f$, ensures that the inequality will be satisfied after a finite number
  of backtrackings at every iteration.  The systematic reduction of $\beta$ (between steps 3 and
  11 of Algorithm 4)  in each backtracking, with all  other terms fixed and well-defined on both sides
  of inequality (\ref{roughLS}),  ensures that the  term  $\gamma (\beta \sigma \alpha_k)^2 \|g_k\|_2^2$
   approaches zero. Therefore, the right-hand side of the inequality  approaches  $f(x_k) + \eta_k$.
   Meanwhile, the left side of the inequality at Step 3   approaches $f(x_k)$ due to $f$'s continuity.
   Hence, since $\eta_k$ on     the right hand side remains fixed and positive, the inequality  will
   be satisfied for a sufficiently small value of $\beta$, i.e., after a finite number of  reductions of $\beta$.
We also note that in the case of rejection of the first trial point, the next ones
are computed along the same search direction. As a consequence, the gradient
function is evaluated only once during the backtracking process. Finally, since
the sequence $\{x_k\}_{k\in\N}$ generated by any of the considered extended
algorithms to solve (\ref{genprb}) satisfies (\ref{roughLS}), it follows that
for all $k\geq 0$
\[
  f(x_{k+1})  \leq  f(x_{k}) - \gamma (\sigma\alpha_k)^2 \|g_k\|^2 + \eta_k.
\]
Therefore, we extract the following results directly from the theoretical
results in \cite{LaCruz}:
\begin{itemize}
  \item  (Proposition 2.1 in \cite{LaCruz}) The sequence $\{x_{k}\}_{k\in\N}$
    generated by the algorithm is contained in $\Omega_0$, given by $ \Omega_0 =
    \{x\in \R^n\::\: f(x)\leq f(x_0) +\eta\}$. Notice that $\Omega_0$ is closed
    and bounded above. If $\Omega_0$ is also bounded below then it is a compact
    set, which happens, for example, if $f$ is strictly convex.
  \item (Theorem 2.1 in \cite{LaCruz})  If  $\Omega_0$ is a compact set, then
  $\lim_{k\rightarrow\infty} \|g_k\|=0. $ Notice that if $f$ is strictly
    convex, then by the continuity of $\nabla f(x)$ the previous result implies
    that $x_k \rightarrow x^*$, the unique global minimizer of $f$.
\end{itemize}

\subsection{Numerical experiments on strictly convex functions}

Using the extended algorithms discussed above, we now illustrate the different methods and acceleration strategies on different convex functions, all defined on $\R^n$ with $n=1000$ and with initial guesses and parameters kept constants for all tests:
\begin{description}
  \item[Strictly convex 2:] $\ds \R^n \ni x \mapsto f(x) = \sum_{i=1}^n
    \frac{i}{10}\big(\exp(x_i) - x_i)$
    with $x_0$ given by random coordinates in the interval $[0,3]$.
  \item[Logistic loss:]  $\ds \R^n \ni x \mapsto f(x) =
    \frac{0.1}{2}\norm{x}^2 + \sum_{i=1}^p \log\big( 1 + \exp(- (x^\top z_i)
    y_i)\big)$ with $p = 100 < n$ and $x_0$ the vector of all ones. $z_i\in\R^n$ are random vectors and $y_i$ are randomly chosen as $\pm1$.
    \item[Convex log-sum-exp:] $\ds \R^n \ni x \mapsto f(x) = \log\left(1+\sum_{i=1}^n \frac{i}{10} \exp(x_i^2) \right)$
    with $x_0$ given by normal random coordinates, with mean $0$ and standard deviation $3$.
    \item[Negative entropy:] $\ds \R^n_{>0} \ni x \mapsto f(x) = \sum_{i=1}^n
    \frac{i}{10}x_i\log(x_i)$
    with $x_0$'s coordinates given by small random perturbations of the minimizer $(\frac1{e},\frac1{e},\dots,\frac1{e})$ to avoid leaving $\R^n_{>0}$.
\end{description}

Concerning  the approximation in (\ref{approxAv}), we set
\begin{equation*}
  h = \frac{10^{-5}}{\min \{1, \max\{10^{-3},10^5\|g_k\|\}\}} \in [10^{-5},10^{-2}].
\end{equation*}
Hence, $h=10^{-5}$ when $\|g_k\|\geq 10^{-5}$ and $h>10^{-5}$ otherwise. The
idea is to take $h$ larger if $\|g_k\|$ is too small, avoiding numerical
instabilities associated with very small steps $hg_k$. In practice, this
strategy turns out to be more effective than using a fixed $h$. For the line
search we set $\gamma = 10^{-4}$, we start with $\eta_0 =\|g_0\|$ and for $k\geq
1$ we set $\eta_k=\eta_0/k^{1.1}$ such that  $\sum_{k\in\N}\eta_k <\infty$.
For the safeguard strategies (\ref{safedk}) and (\ref{safeMG}),
we set $\tilde{\tau} =10^{-10}$. In
all cases, we stop the process when $\|g_k\| \leq \texttt{tol}$ where
$\texttt{tol} = 10^{-10}$.   Finally, the relaxation parameter to avoid the zigzag
effect is chosen as $\sigma =0.8$ and to detect an eigenvector approximation we
set  $\epsilon_{\rm eig} = 0.8$.

To further evaluate and illustrate the efficiency of LBA, we next compare it not only to MR and the eigenvector acceleration schemes but also to state-of-the-art ABBmin and LMSD non-monotone gradient-type methods,  for the global minimization of non-quadratic
convex functionals. Note that since our focus is on  gradient-type methods, we have chosen not to include (quasi-)Newton-type methods in this
 comparison.
\begin{itemize}
    \item We have chosen the low-cost, non-monotone ABBmin method (see \cite{drtz18, frassoldatiNewAdaptiveStepsize2008, Zhou06}) because
    it has proven
    to be the most effective and competitive accelerated version of the Barzilai-Borwein (BB) method for smooth unconstrained
    optimization; see, e.g.,  \cite{CheRay20, ddrt13, drtz18}. For comparison, we used the globalized version with a non-monotone line search
    described in \cite{drtz18}, where global convergence  is established  to critical points of general smooth functions.

    \item LMSD, initially introduced in \cite{fletcherLimitedMemorySteepest2012}, has been chosen since it also makes use of the Lanczos process, in a different manner than LBA, to produce steplengths as the inverse of Ritz values. For comparison, we use the extended LMSD from \cite{ferrandiLimitedMemoryGradient2025}, which has a freely available Python implementation\footnote{Downloaded from \url{https://github.com/gferrandi/lmsdpy} on May 28th, 2026.}. LBA variants have been compared to the default LMSD (Fletcher's) method with BB1 stepsizes.
\end{itemize}
Let us emphasize that
ABBmin and LMSD require the  additional storage of  one and $m$ previous gradient vectors, respectively.
Although the implementation of LBA shown in Algorithm \ref{alg:LBA} appears to also require $m$ additional vectors of storage, we point out that the alternative implementation based on CR discussed in Remark~\ref{rem:Impl} requires only a few additional  vectors of storage.
Below, we consider history of sizes $5$ or $10$ when comparing performances.

First, recall that in Section \ref{adplncz} we discussed  two variants of  adaptive LBA. We tested these variants on the selected convex functions for several dimensions and found that fixing the number of Lanczos steps $m$ consistently produced better results than the adaptive variants. Hence, Table \ref{tab:comparison_nonquad} presents the performance of the various algorithms and LBA with different values of $m$ ($4 \leq m \leq 10$) for each of the four convex functions above. Based on these results, it seems that the choice of $m$ has less influence than for the quadratic case, even though fixing $m=8$ or $m=10$ seem to be the best choices overall. Additionally, Figure \ref{fig:plots_nonquad} shows the behavior of the various algorithms along the iterations. The superiority of LBA over MR and the eigenvector acceleration schemes in terms of gradient evaluations is clearly observed. However, it is striking that, if LBA competes with LMSD and ABBmin in therms of iterations, it requires about twice the number of gradient evaluations to reach the same accuracy. This can be explained by (i) LBA requiring two gradient evaluations per iterations in order to use the finite differences approximation of the Hessian (instead of one for LMSD and ABBmin) and (ii) LBA using additional gradient evaluations whenever the Lanczos acceleration is triggered.
Nevertheless, LBA has one important advantage over non-monotonic algorithms like LMSD and ABBmin: it does not need information or
storage of past iteration data. This makes LBA particularly efficient for large-scale optimization, especially for stochastic problems where the objective
function relies on random samples that change with each iteration; see \cite{Bottou}. A gradient-type method that uses only current
information is ideal for these types of problems because it avoids inconsistencies from outdated historical data.
Finally, it is worth noting that, even when backtracking is not activated, the line search strategy is essential to ensure convergence.

\begin{table}[p!]
    \centering
    \begin{tabular}{@{}cccccccccccccc@{}}
    \toprule
    \multicolumn{2}{c}{}                        &  & \multicolumn{5}{c}{\textbf{Strictly convex 2}}            &          & \multicolumn{5}{c}{\textbf{Logistic loss}}    \\
    \multicolumn{2}{c}{}                        &  & iter.       & grad.     & Lan.      & $t$       & backtr. &          & iter.        & grad.    & Lan.   & $t$     & backtr.  \\ \midrule
    \multicolumn{2}{c}{\textbf{MR (no zigzag)}} &  & 786         & 1573      & n.a.      & 17.1      & 0       &          & 94           & 189      & n.a.   & 0.50    & 0  \\ \midrule
    \multicolumn{2}{c}{\textbf{eigvec. accel.}} &  & 782         & 1565      & n.a.      & 16.9      & 0       &          & 97           & 195      & n.a.   & 0.48    & 0  \\ \midrule
    \multirow{5}{*}{\textbf{LBA}}     & $m=4$   &  & 373         & 799       & 13        & 8.5       & 0       &          & 31           & 155      & 23     & 0.32    & 0  \\
                                      & $m=5$   &  & 430         & 936       & 15        & 9.8       & 0       &          & 31           & 178      & 23     & 0.35    & 0  \\
                                      & $m=6$   &  & 400         & 891       & 15        & 9.2       & 0       &          & 28           & 171      & 19     & 0.36    & 0  \\
                                      & $m=8$   &  & 405         & 923       & 14        & 9.5       & 0       &          & 20           & 153      & 14     & 0.28    & 0  \\
                                      & $m=10$  &  & 373         & 877       & 13        & 8.9       & 0       &          & 18           & 157      & 12     & 0.28    & 0  \\ \midrule
    \multirow{2}{*}{\textbf{ABBmin}}  & $m=5$   &  & 364         & 365       & n.a.      & 4.0       & 3       &          & 84           & 85       & n.a.   & 0.22    & 0  \\
                                      & $m=10$  &  & 364         & 365       & n.a.      & 4.0       & 3       &          & 84           & 85       & n.a.   & 0.24    & 1  \\ \midrule
    \multirow{2}{*}{\textbf{LMSD}}    & $m=5$   &  & 446         & 448       & n.a.      & n.a.      & n.a.    &          & 72           & 74       & n.a.   & n.a.    & n.a.  \\
                                      & $m=10$  &  & 456         & 458       & n.a.      & n.a.      & n.a.    &          & n.c.         & n.c.     & n.a.   & n.a.    & n.a.  \\ \bottomrule
                                      &         &  &             &           &           &           &         &          &              &          &        &         & \\ \toprule
                                      &         &  & \multicolumn{5}{c}{\textbf{Convex log-sum-exp}}          &           & \multicolumn{5}{c}{\textbf{Negative entropy}} \\
                                      &         &  & iter.       & grad.     & Lan.      & $t$      & backtr. &           & iter.        & grad.    & Lan.   & $t$    & backtr.    \\ \midrule
    \multicolumn{2}{c}{\textbf{MR (no zigzag)}} &  & 540         & 1081      & n.a.      & 16.9     & 0       &           & 526          & 1053     & n.a.   & 11.3   & 0   \\ \midrule
    \multicolumn{2}{c}{\textbf{eigvec. accel.}} &  & 425         & 851       & n.a.      & 13.0     & 0       &           & 519          & 1039     & n.a.   & 11.2   & 0   \\ \midrule
    \multirow{5}{*}{\textbf{LBA}}     & $m=4$   &  & 268         & 609       & 18        & 9.3      & 0       &           & 293          & 619      & 8      & 6.4    & 0   \\
                                      & $m=5$   &  & 280         & 626       & 13        & 9.5      & 0       &           & 343          & 722      & 7      & 7.6    & 0   \\
                                      & $m=6$   &  & 300         & 685       & 14        & 10.3     & 0       &           & 250          & 549      & 8      & 5.7    & 0   \\
                                      & $m=8$   &  & 276         & 665       & 14        & 9.9      & 0       &           & 214          & 477      & 6      & 4.9    & 0   \\
                                      & $m=10$  &  & 257         & 645       & 13        & 9.5      & 0       &           & 254          & 569      & 6      & 5.8    & 0   \\ \midrule
    \multirow{2}{*}{\textbf{ABBmin}}  & $m=5$   &  & 298         & 299       & n.a.      & 4.7      & 9       &           & 259          & 260      & n.a.   & 3.0    & 14   \\
                                      & $m=10$  &  & 342         & 343       & n.a.      & 5.4      & 8       &           & 235          & 236      & n.a.   & 2.7    & 12   \\ \midrule
    \multirow{2}{*}{\textbf{LMSD}}    & $m=5$   &  & 489         & 491       & n.a.      & n.a.     & n.a.    &           & 440          & 442      & n.a.   & n.a.   & n.a.   \\
                                      & $m=10$  &  & 366         & 368       & n.a.      & n.a.     & n.a.    &           & 275          & 277      & n.a.   & n.a.   & n.a.   \\ \bottomrule
    \end{tabular}
    \caption{Performance of the LBA algorithm ($\sigma=0.8$ and $\epsilon_{\rm eig}=0.8$) for the global optimization of various convex, non-quadratic, functionals and for different values of $m$, the fixed number of Lanczos steps when activated. For ABBmin and LMSD, $m$ is the size of the memory. \enquote{iter.} (resp. \enquote{grad.} / \enquote{Lan.} / \enquote{$t$} / backtr.) stands for the number of iterations to reach convergence (resp. the number of gradient evaluations / the number of activations of the Lanczos algorithm / the CPU time in seconds / the number of bactracking activations). \enquote{n.a.} meens \enquote{not applicable} and is used either for the number of Lanczos activation for algorithms other than LBA or for timings / backtr. for the LSMD algorithm which was run using Python while all other algorithms are run with Octave (and backtr. history not available). \enquote{n.c.} means \enquote{not converged} and is used when $\|g_k\| \leq 10^{-10}$ is not reach before a maximum number of iterations. Note that none of the LBA variants required backtrackings.}
    \label{tab:comparison_nonquad}
\end{table}

\begin{figure}[p!]
    \centering
    \includegraphics[width=0.45\linewidth]{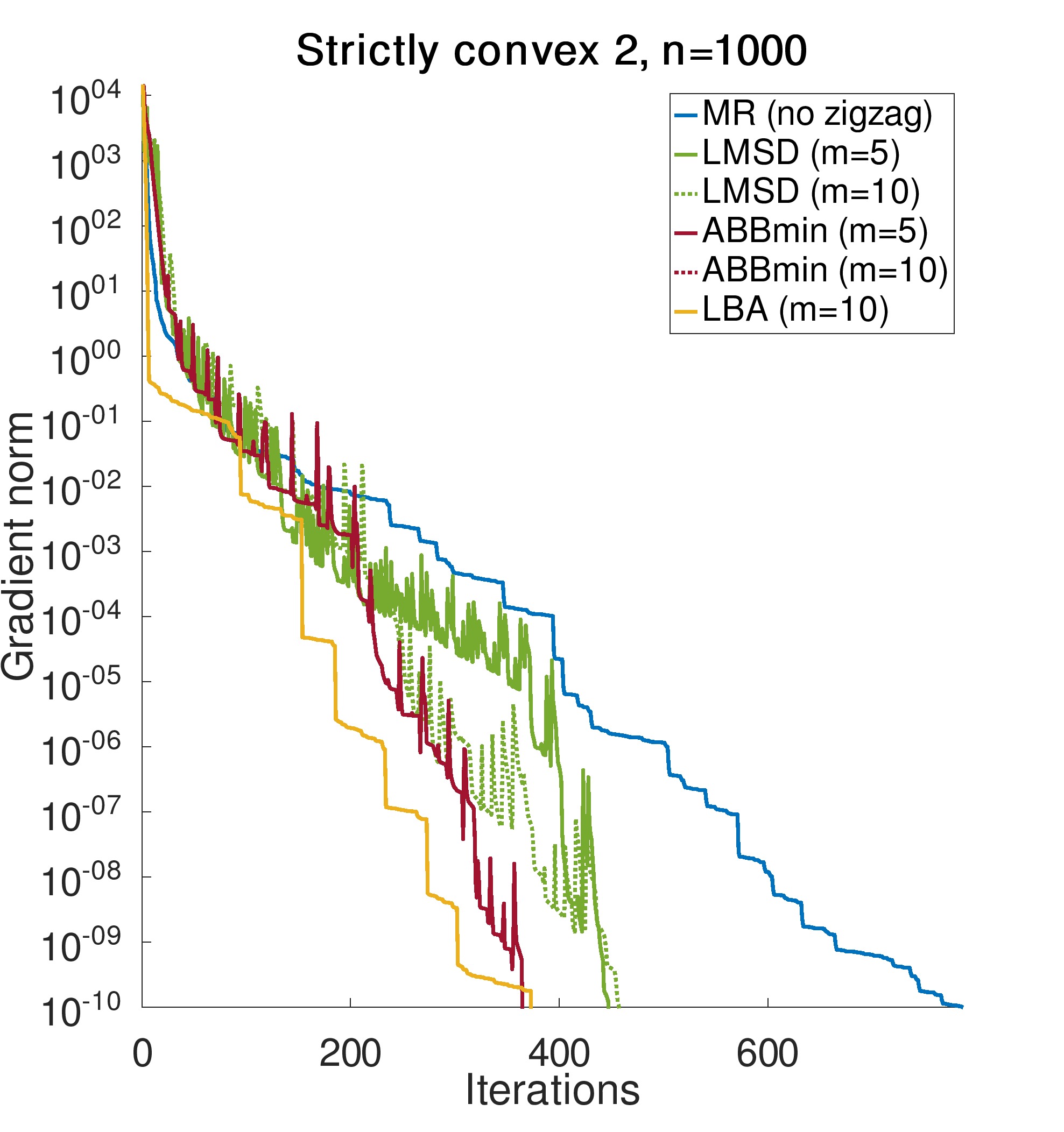}\hfill
    \includegraphics[width=0.45\linewidth]{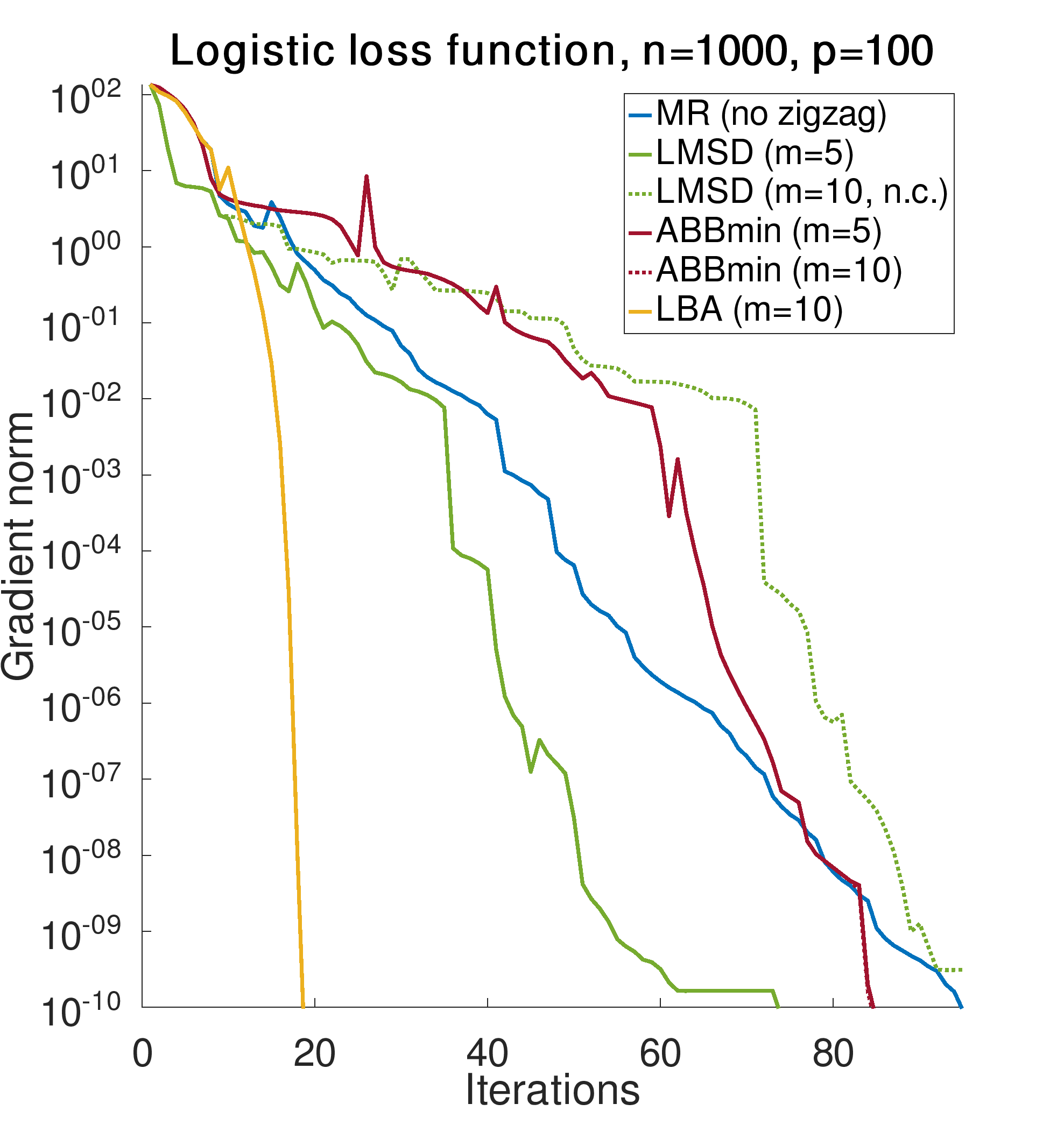}\\
    \includegraphics[width=0.45\linewidth]{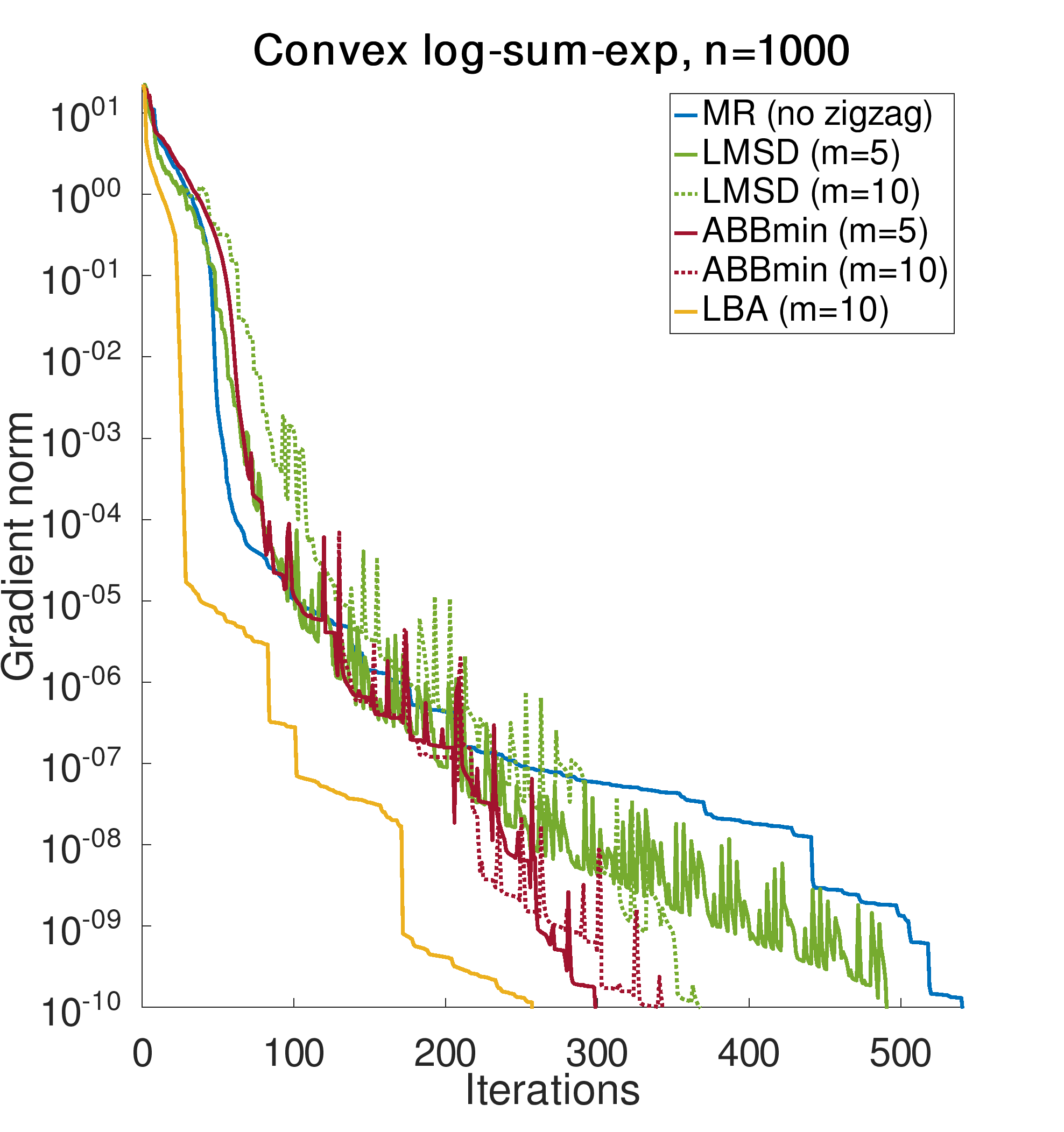}\hfill
    \includegraphics[width=0.45\linewidth]{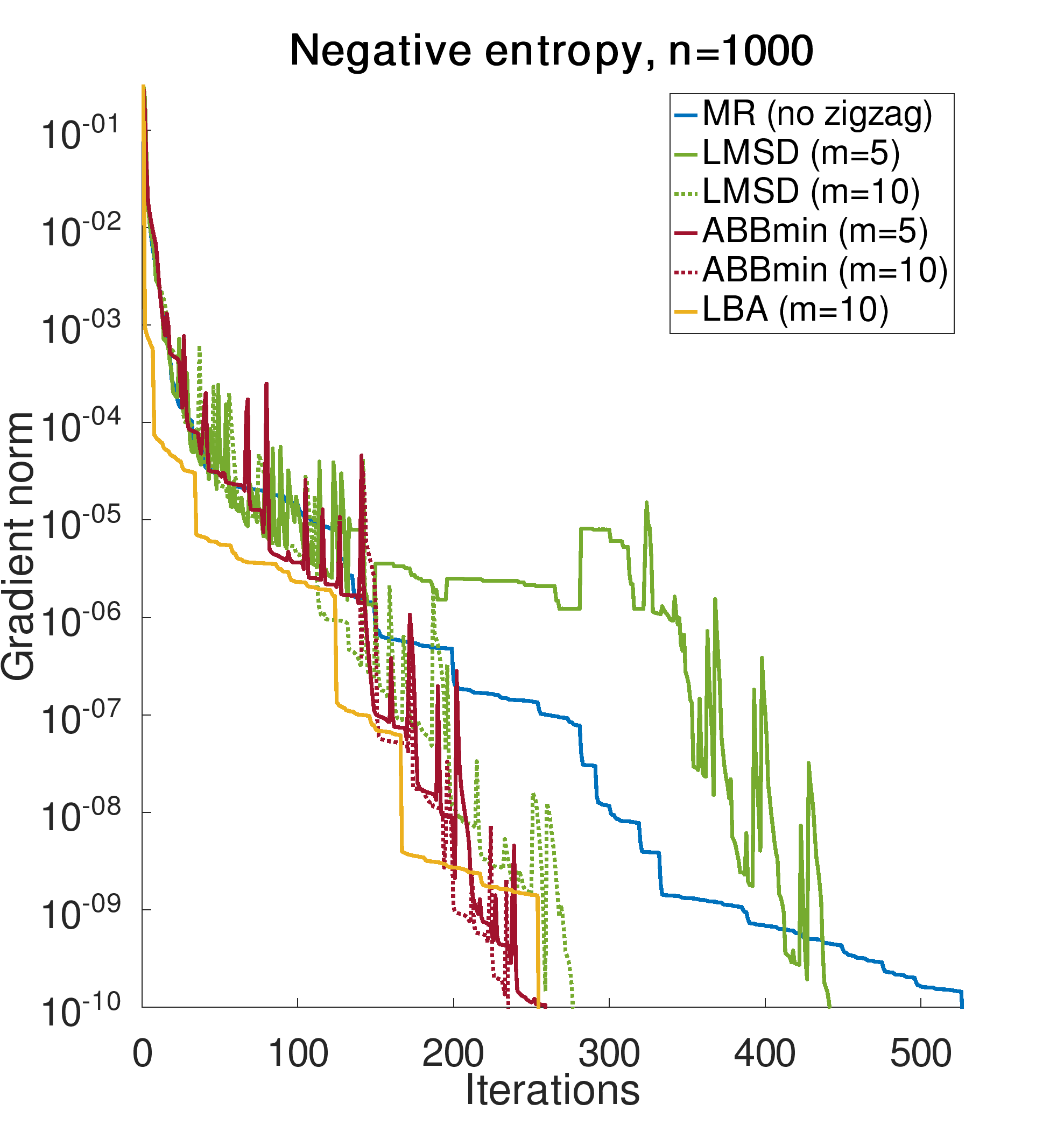}\\
    \caption{Convergence plots for various convex, non-quadratic, functionals. In terms of iterations, the LBA algorithm  ($\sigma=0.8$, $\epsilon_{\rm eig}=0.8$ and $m=10$) competes with state of the art ABBmin and LMSD algorithms (where $m$ is the history size).} 
    \label{fig:plots_nonquad}
\end{figure}

\begin{remark}[Real-world dataset]
    To further illustrate the efficiency of LBA, we used it with $m=10$ for the logistic loss function and the \texttt{mushrooms} dataset from LIBSVM \cite{CC01a}\footnote{Downloaded from \url{https://www.csie.ntu.edu.tw/~cjlin/libsvmtools/datasets/binary.html\#mushrooms} on June 15th, 2026.}. The parameters are $n=112$, $p=8124$ and the conclusions are similar to those from Table \ref{tab:comparison_nonquad} and Figure \ref{fig:plots_nonquad}, see Figure \ref{fig:champi}.

    \begin{figure}[p!]
        \begin{minipage}{0.48\textwidth}
            \includegraphics[width=\linewidth]{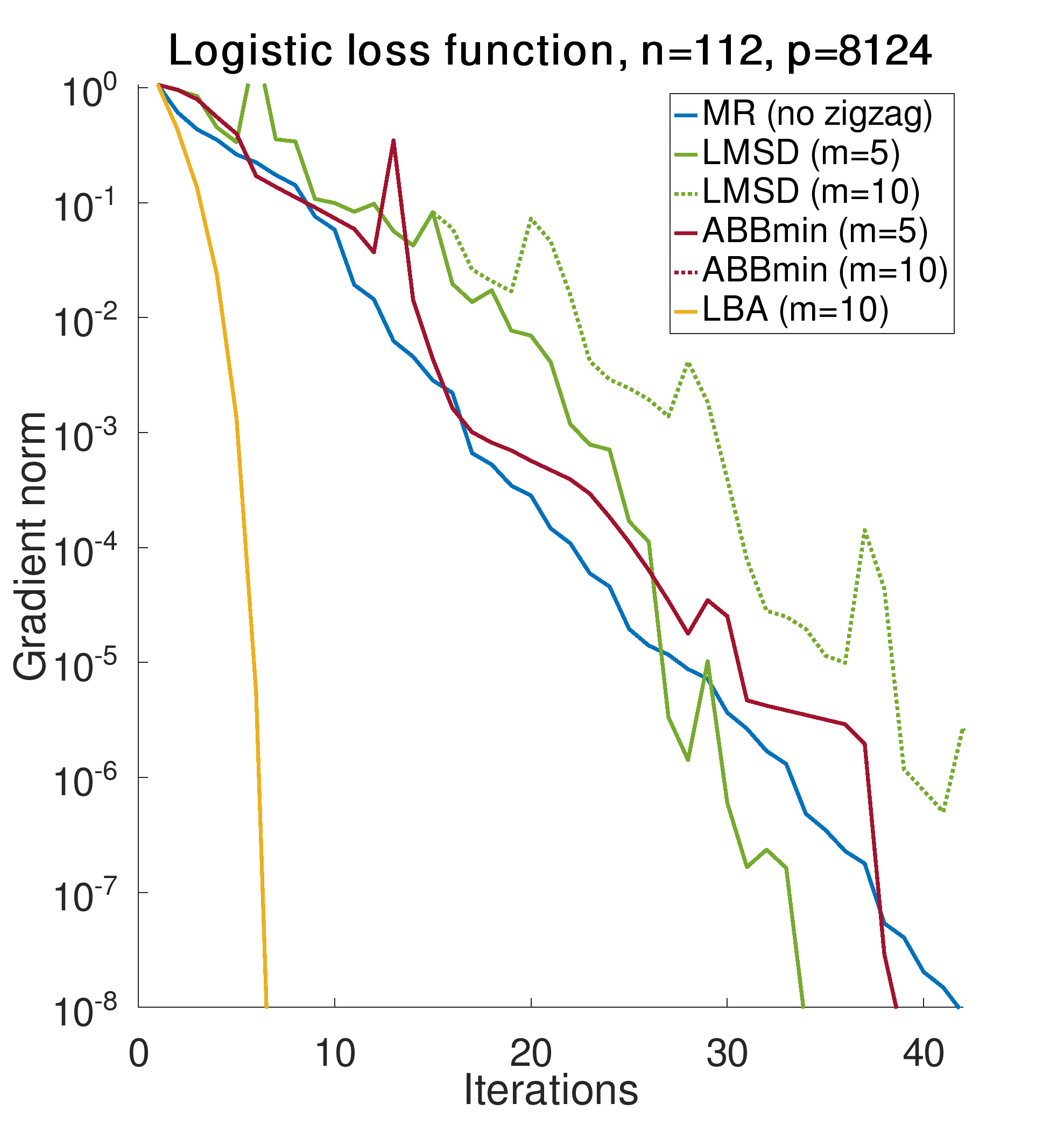}
        \end{minipage}\hfill
        \begin{minipage}{0.48\textwidth}
            \normalfont
            \centering
            \begin{tabular}{@{}ccccccc@{}}
            \toprule
            \multicolumn{2}{c}{}                        &  & \multicolumn{4}{c}{\textbf{Logistic loss}}             \\
            \multicolumn{2}{c}{}                        &  & iter.        & grad.        & Lan.        & $t$        \\ \midrule
            \multicolumn{2}{c}{\textbf{MR (no zigzag)}} &  & 41           & 83           & n.a.        & 18.6       \\ \midrule
            \textbf{LBA}                      & $m=10$  &  & 6            & 76           & 6           & 10.6       \\ \midrule
            \multirow{2}{*}{\textbf{ABBmin}}  & $m=5$   &  & 38           & 39           & n.a.        & 7.6        \\
                                              & $m=10$  &  & 38           & 39           & n.a.        & 8.0        \\ \midrule
            \multirow{2}{*}{\textbf{LMSD}}    & $m=5$   &  & 32           & 34           & n.a.        & n.a.       \\
                                              & $m=10$  &  & 47           & 49           & n.a.        & n.a.       \\ \bottomrule
            \end{tabular}
        \end{minipage}
        \caption{Performance of LBA and other algorithms for the logistic loss convex functional with parameters taken from the \texttt{mushrooms} LIBSVM dataset.}
        \label{fig:champi}
    \end{figure}
\end{remark}

\begin{remark}[Parameters choice]\label{rmk:param}
    There are two main parameters to set when using LBA with fixed $m$: $\sigma$ and $\epsilon_{\rm eig}$. The former is important to avoid the zigzag effect while the latter is crucial in the activation of the Lanczos acceleration.
    Regarding the choice of $\sigma$, as discussed in \cite[Section 8]{vandendoel12}, it is recommended to choose $\sigma\in (0,1)$ which has proved to be more effective. As for the choice of $\epsilon_{\rm eig}$, this is clearly a sensitive and problem dependent parameter: if chosen too small the acceleration might never be triggered, while if chosen too large the acceleration might be activated too often, meaning that the algorithm spends a lot of time applying the Lanczos procedure, even when the residuals are not sufficiently aligned with an eigendirection. We report in Table~\ref{tab:parameters} a test on these parameters for minimizing the quadratic functional from Section~\ref{sec:spectral} and for the (non-quadratric) functional Strictly convex 2. It appears that, for the two problems we considered, $\sigma=0.8$ is the best choice while $\epsilon_{\rm eig}$ should not be too small: choosing $\epsilon_{\rm eig}=0.6$ seems to give the best results, but $\epsilon_{\rm eig}=0.8$ is also satisfactory.

    \begin{table}[p!]
        \centering
        \begin{tabular}{@{}cccccc@{}}
        \toprule
        \multicolumn{6}{c}{\textbf{Quadratic functional}}                                                                                                                 \\ \midrule
                                              &                      & \multicolumn{4}{c}{$\sigma$}                                                                       \\
                                              &                      & 0.2                   & 0.4                   & 0.6                  & 0.8                         \\ \midrule
        \multirow{4}{*}{$\epsilon_{\rm eig}$} & 0.2                  & (577, 589, 1)    & (291, 303, 1)    & (198, 210, 1)   & (364, 395, 0)          \\
                                              & 0.4                  & (189, 223, 3)    & (192, 226, 3)    & (233, 267, 3)   & \textbf{(132, 155, 2)}          \\
                                              & 0.6                  & (200, 256, 5)    & (142, 187, 4)    & (199, 255, 5)   & \textbf{(132, 155, 2)} \\
                                              & 0.8                  & (166, 255, 8)    & (158, 225, 6)    & (164, 231, 6)   & {\it(122, 178, 5)}          \\ \midrule
        \multicolumn{1}{l}{}                  & \multicolumn{1}{l}{} & \multicolumn{1}{l}{}  & \multicolumn{1}{l}{}  & \multicolumn{1}{l}{} & \multicolumn{1}{l}{}        \\ \midrule
        \multicolumn{6}{c}{\textbf{Strictly convex 2}}                                                                                                                    \\ \midrule
                                              &                      & \multicolumn{4}{c}{$\sigma$}                                                                       \\
                                              &                      & 0.2                   & 0.4                   & 0.6                  & 0.8                         \\ \midrule
        \multirow{4}{*}{$\epsilon_{\rm eig}$} & 0.2                  & (696, 1413, 2)  & (892, 1795, 1)  & (665, 1341, 1) & (726, 1463, 1)        \\
                                              & 0.4                  & (788, 1607, 3)  & (567, 1155, 2) & (508, 1047, 3) & (468, 957, 2)         \\
                                              & 0.6                  & (590, 1301, 12) & (576, 1263, 11) & (449, 989, 9)  & \textbf{(273, 637, 9)} \\
                                              & 0.8                  & (393, 967, 18)   & (349, 859, 16)   & (362, 865, 14)  & {\it(298, 727, 13)}         \\ \bottomrule
        \end{tabular}
        \caption{Performance of LBA $(m=10)$ for several values of $\sigma$ and $\epsilon_{\rm eig}$. Results are presented as (iterations, matrix-vector products, calls to Lanczos) for the quadratic functional and as (iterations, gradient evaluations, calls to Lanczos) for the Strictly convex 2 functional. The quadratic functional is the same as in Section~\ref{sec:spectral} and LBA stops when $\|r_k\| \leq \|r_0\| \times 10^{-10}$. For the Strictly convex 2 functional, LBA stops when $\|g_k\| \leq 10 ^{-10}$. Shown in bold are the best parameter choices, and in italic those that we actually used throughout the paper, which are also satisfactory.}
        \label{tab:parameters}
    \end{table}
\end{remark}

\begin{remark}[Global optimization of quadratic functionals]
    We have also performed additional comparisons of LBA with ABBmin and LMSD for the quadratic functionals from Section~\ref{sec:spectral} that we do not report for the sake of brevity. It appears that, in the quadratic case, LBA performs similarly to ABBmin and LMSD both in terms of iterations and computational cost since LBA requires only one matrix-vector product per iteration (instead of two gradient evaluations for non-quadratic convex functionals).
\end{remark}

\begin{remark}[Reducing gradient evaluations]
    As mentioned above, the main reason why LBA requires about twice as many gradient evaluations as ABBmin and LMSD is because of the finite differences approximation of the Hessian. This issue can be bypassed by using similar tricks as those in \cite[Section 4.1]{heNlTGCRClassNonlinear2024a} where the updates are performed with only one gradient evaluations once a nearly linear regime is reached. This is left for future work.
\end{remark}

\section{Detailed analysis of relaxed gradient iterations}\label{sec:det_analysis}

The goal of this complementary section is to prove a number of results that
justify or illustrate the behavior observed numerically in Section~\ref{sec:quad}, that is the
concentration of the residuals on a few extreme eigenmodes when the relaxation
parameter $\sigma$ is chosen different than $1$. Again, we focus on the quadratic functional associated to a matrix $A$. The numerical examples illustrating this section use a $900 \times 900$ matrix $A$, obtained from the finite differences discretization of a Poisson problem on a $30 \times 30$ grid. We recall that we denote by $ \lambda_1 \le \lambda_2 \le \cdots \le \lambda_n $ the eigenvalues of $A$ and $u_1, u_2, \cdots, u_n$ an associated orthonormal set of eigenvectors.

\subsection{Preliminaries on shifted and scaled power methods}

We also recall that any gradient-type method can be seen as a scaled and shifted power method for the sequence of residuals $r_k = b - Ax_k = -g_k$ by writing
\begin{equation} \label{eq:PowerM2}
  r_{k+1}  =  r_k - \sigma \alpha_k Ar_k= (I-\sigma \alpha_kA)r_k=
  \left(\prod_{i=0}^k (I-\sigma \alpha_iA)\right)r_0.
\end{equation}
Next, if we define, for a given  nonzero vector $v$, the following Rayleigh quotients:
\begin{equation}
  \mu(A,v) \equiv \frac{ v^\top  A v }{ v^\top  v} , \qquad \nu(A,v) \equiv
  \frac{v^\top  A^2  v }{ v^\top  A v}, \label{eq:munu}
\end{equation}
then, the steplengths \eqref{eq:alpSD} and \eqref{eq:alpMR} are nothing else
than their inverses:
$ \alpha_k\up{SD}~=~\frac1{\mu(A,r_k)}$ and
$\alpha_k\up{MR}~=~\frac1{\nu(A,r_k)}$.
They can also be seen as special cases of the generalized inverse Rayleigh
quotient formula:
\begin{equation} \label{eq:Rayquo}
  \alpha_k = \frac{r_k^\top  A^{p-1} r_k}{ r_k^\top  A^p r_k},
\end{equation}
where $p \geq 1$ is an integer number. Clearly, SD corresponds to $p=1$ and MR
to $p=2$.


We now denote the components of $r_k$ in the eigenbasis of $A$ by $\beta_{1,k},
\dots, \beta_{n,k}$. If we normalize $r_k$ by its Euclidean norm for SD and the
$A$-norm for MR, the $(\beta_{i,k})_{i=1,\dots,n}$'s will satisfy
\eq{eq:normalize}
\text{(SD)}: \quad \sum_{i=1}^n \beta_{i,k}^2 =1 , \qquad \text{(MR)}: \quad
\sum_{i=1}^n \beta_{i,k}^2 \lambda_i = 1 .
\en
Under these conditions the scalars $(\beta_{i,k}^2)_{i=1,\dots,n}$ and
$(\lambda_i \beta_{i,k}^2)_{i=1,\dots,n}$ can be viewed as probability
distributions; and the SD and MR iterations can be viewed as processes that
transform these distributions: this approach is the one originally proposed by
Akaike~\cite{Akaike}. The Rayleigh quotients mentioned above can be viewed as the means of the
eigenvalues under these distributions, and thus we will use the notation:
\eq{eq:meanBet}
\text{(SD) :} \quad
\lbet = \mu(A,r_k) = \sum_{i=1}^n \beta_{i,k}^2 \lambda_i \ , \qquad
\text{(MR) :} \quad
\lbet = \nu(A,r_k) = \sum_{i=1}^n \beta_{i,k}^2 \lambda_i^2 .
\en
As convex combinations of the eigenvalues, $\sum_{i=1}^n \beta_{i,k}^2 \lambda_i
$ with $\sum_{i=1}^n \beta_{i,k}^2 = 1$ for SD and $\sum_{i=1}^n (\beta_{i,k}^2
\lambda_i) \lambda_i $ with $\sum_{i=1}^n \lambda_i^2 \beta_{i,k}^2 = 1$ for MR,
these weighted means will always belong to the interval $[\lambda_1,
\lambda_n]$. A gradient step (SD or MR) thus transforms the distribution
$(\beta_{i,k})_{i=1,\dots,n}$ into a new distribution
$(\beta_{i,k+1})_{i=1,\dots,n}$ associated to a new normalized residual
$r_{k+1}$, by applying a step of the shifted power method to the current
residual $r_k$, where the shift changes at each step -- see
equation~\eqref{eq:PowerM2}.

In both SD and MR, the new residual has the following components:
\eq{eq:betatil}
\beta_{i,k+1} = \frac{1}{s_k} \left( \frac{\lbet}{\sigma} -
  \lambda_i  \right) \beta_{i,k}, \qquad 1\le i \le n,
\en
where $s_k$ stands for a normalization factor, that depends on the iteration
$k$. The magnitude of each component $\beta_i$ is therefore `amplified' by
$| \frac{\lbet}{\sigma} - \lambda_i |/s_k$.
Since the mean $\lbet$ belongs to the interval $[\lambda_1,
\lambda_n]$, $\lbet/\sigma $ belongs to the interval
$[\frac{\lambda_1}{\sigma}, \frac{\lambda_n}{\sigma}]$.
The (un-normalized) amplification factors $|\frac{\lbet}{\sigma} -
\lambda_i|$ are illustrated in Figure~\ref{fig:amplif}.  As can be seen at every
step, the highest factor is either the one associated with $\lambda_1$ or the
one associated with $\lambda_n$. With the assumption that $\beta_{1,0}
\beta_{n,0} \ne 0$, it can therefore be seen that the shifted power method with
the shifts equal to the inverse of generalized Rayleigh quotients will tend to
produce a vector that has components only in the eigenvectors associated with
these eigenvalues. However, since $\lbet$ depends on $k$, the
dominating eigenvector might change from an iteration to the next. The following
results help to understand this phenomenon.

\begin{figure}[ht]
  \centerline{
    \includegraphics[width=0.45\textwidth]{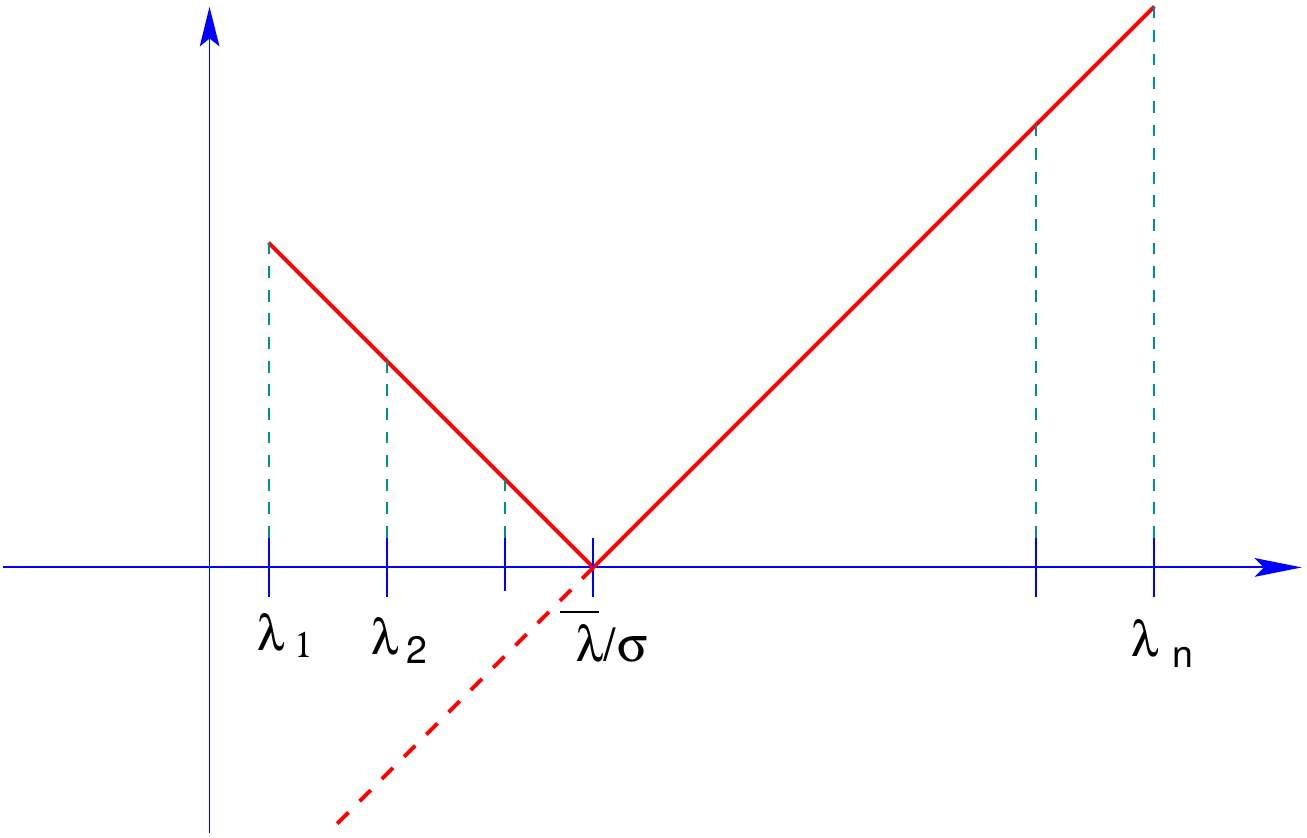} }
  \caption{Amplification  factors for relaxed SD / MR scheme}
  \label{fig:amplif}
\end{figure}

We define
\eq{eq:gamh}
  \xi_k \equiv \frac{\lbet}{\sigma}
\en
With this notation, \eqref{eq:betatil} corresponds to the shifted and scaled
power method
\eq{eq:SPower} r_{k+1} = \frac{1}{s_k} (\xi_k I - A) r_k .  \en
This is a very general form, and
the result shown next does not depend on how the shift $\xi_k$ is generated but
only assumes a shifted power iteration of the form \eqref{eq:SPower} where
$\xi_k$ is selected from the (open) interval $(\frac{\lambda_1}{\sigma},
\frac{\lambda_n}{\sigma})$ at each step.

\begin{lemma}\label{lem:SPower}
  Let $A\in\R^{n\times n}$ be a symmetric, positive definite matrix with
  eigenvalues $\lambda_1<\lambda_2\leq\dots\leq \lambda_{n-1}<\lambda_n$ and
  associated eigenvectors $(u_i)_{i=1,\dots,n}$. Consider the following shifted
  and scaled power method
  \begin{equation}\label{eq:SPower2}
    \begin{cases}
      r_0 \in \R^n,\quad \norm{r_0} = 1,\\
      r_{k+1} = \frac1{s_k}{(\xi_k I - A)r_k}.
    \end{cases}
  \end{equation}
  where $s_k$ is a normalization factor and $\xi_k$ is  selected at each step $k$
  from the interval $(\frac{\lambda_1}\sigma,\frac{\lambda_n}\sigma)$,
  with $\sigma\in(0,2)$.
  Let $(\beta_{i,k})_{k\in\N}$ be the $i$-th component of $r_k$.
  Assume that $\beta_{1,0}\beta_{n,0}\neq0$.
  Then, for $i=1,\dots,n$
  \begin{equation}\label{eq:limsup}
    \begin{cases}
      \ds\limsup_{k\to+\infty} \biggr(\prod_{j=0}^k\frac{|\xi_j - \lambda_i|}
      {|\xi_j-\lambda_1|}\biggr)^{1/k} < 1 \quad\Rightarrow\quad
      \lim_{k\to+\infty}\frac{|\beta_{i,k}|}{|\beta_{1,k}|} = 0\\
      \ds\limsup_{k\to+\infty} \biggr(\prod_{j=0}^k\frac{|\xi_j - \lambda_i|}
      {|\xi_j-\lambda_n|}\biggr)^{1/k} < 1 \quad\Rightarrow\quad
      \lim_{k\to+\infty}\frac{|\beta_{i,k}|}{|\beta_{n,k}|} = 0.\\
    \end{cases}
  \end{equation}
\end{lemma}

\begin{proof}
  Let $i=1,\dots,n$ and consider the first case. Applying successively
  $(\xi_jI-A)$ for $j=0,\dots,k$ and taking the scalar product between
  $r_{k+1}$ and $u_i$, we get for any $k$, after cancellation of the
  normalization factors,
  \[
    \frac{|\beta_{i,k+1}|}{|\beta_{1,k+1}|} =
    \frac{\ds |\beta_{i,0}| \prod_{j=0}^k|\xi_j - \lambda_i|}
    {\ds |\beta_{1,0}| \prod_{j=0}^k|\xi_j - \lambda_1|} =
    \frac{|\beta_{i,0}|}{|\beta_{1,0}|}
    \prod_{j=0}^k \frac{|\xi_j - \lambda_i|}{|\xi_j - \lambda_1|}
  \]
  Since
  \[
    \limsup_{k\to+\infty} \biggr(\prod_{j=0}^k\frac{|\xi_j - \lambda_i|}
    {|\xi_j-\lambda_1|}\biggr)^{1/k} < 1
    \quad\Rightarrow\quad
    \lim_{k\to+\infty} \prod_{j=0}^k
    \frac{|\xi_j-\lambda_i|}{|\xi_j-\lambda_1|} = 0,
  \]
  we obtain the desired result. The other case is treated similarly.
\end{proof}

\begin{remark}
  The case where $\beta_{1,0} = 0$ or $\beta_{n,0} = 0$ can be
  easily handled: $\lambda_1$ should be replaced by the smallest eigenvalue
  $\lambda_i$ for for which $\beta_{i,0}\ne 0$ and $\lambda_n $ by the largest
  eigenvalue $\lambda_i$ for for which $\beta_{i,0}\ne 0$.
  Also, the proof can be easily extended to the situation where $\xi_k$ belongs to
  the closed interval $[\lambda_1/\sigma, \lambda_n/\sigma]$ instead of being restricted to
  $(\lambda_1/\sigma, \lambda_n/\sigma)$. Indeed, if $\xi_k = \lambda_1 / \sigma$ for a given step $k$
  then $\beta_{1,k+1} = 0$ per \eqref{eq:betatil}.
  Thus, we are in a situation similar to the one where $\beta_{1,0} =
  0$ and we can apply the same remedy: replace $\lambda_1$ by smallest
  eigenvalue  for which this does not happen. We can proceed similarly  when
  $\xi_k = \lambda_n$ for a given $k$.
\end{remark}

In other words, this lemma gives a sufficient condition on the choice of the
shift parameters for the normalized residuals $i$-th component
$(\beta_{i,k})_{k\in\N}$ to be negligible  compared to the two extremal modes.
In the first case in \eqref{eq:limsup}, the condition is equivalent to
\[
  \limsup_{k\to+\infty} \frac1k \sum_{j=0}^k \log\frac{|\xi_j -
    \lambda_i|}{|\xi_j-\lambda_1|} < 0,
\]
which means that, on average, $ \frac{|\xi_j - \lambda_i|}{|\xi_j-\lambda_1|}$
is smaller than $1$ but is not prevented from being larger than $1$ from time to
time. These considerations show that the behavior of the shifted and scaled
power method depends crucially on how the $\xi_k$'s are chosen inside
$(\frac{\lambda_1}{\sigma},\frac{\lambda_n}{\sigma})$. If they are chosen as independent and
identically distributed random variables, uniformly in
$(\frac{\lambda_1}{\sigma},\frac{\lambda_n}{\sigma})$, we can prove the following Lemma. For
a more general result on random steplengths when $\sigma=1$, we refer the
interested reader to \cite{kalousekSteepestDescentMethod2017}.

\begin{lemma}\label{lem:random}
  Let $(\xi_k)_{k\in\N}$ be a sequence of independent random variables,
  uniformly distributed in $(\frac{\lambda_1}{\sigma},\frac{\lambda_n}{\sigma})$ for
  $\sigma\in(0,2)$. Then, under the same setting as Lemma~\ref{lem:SPower},
  \begin{itemize}
    \item if $\sigma < 1$, for any $i\neq1$, $\ds\limsup_{k\to+\infty}
      \biggr(\prod_{j=0}^k\frac{|\xi_j - \lambda_i|}
      {|\xi_j-\lambda_1|}\biggr)^{1/k} < 1$. As a consequence, $|\beta_{i,k}|$
      converges to $0$ for all $i$'s but $i=1$ and
      $\lim_{k\to+\infty}|\beta_{1,k}|=1$.
    \item if $\sigma > 1$, for any $i\neq n$, $\ds\limsup_{k\to+\infty}
      \biggr(\prod_{j=0}^k\frac{|\xi_j - \lambda_i|}
      {|\xi_j-\lambda_n|}\biggr)^{1/k} < 1$. As a consequence, $|\beta_{i,k}|$
      converges to $0$ for all $i$'s but $i=n$ and
      $\lim_{k\to+\infty}|\beta_{n,k}|=1$.
    \item if $\sigma = 1$, for any $i\neq 1,n$, $\ds\limsup_{k\to+\infty}
      \biggr(\prod_{j=0}^k\frac{|\xi_j - \lambda_i|}
      {|\xi_j-\lambda_1|}\biggr)^{1/k} < 1$ and $\ds\limsup_{k\to+\infty}
      \biggr(\prod_{j=0}^k\frac{|\xi_j - \lambda_i|}
      {|\xi_j-\lambda_n|}\biggr)^{1/k} < 1$.
      As a consequence, $|\beta_{i,k}|$ converges to $0$ for all $i\neq1,n$
      and $|\beta_{1,k}|$ and $|\beta_{n,k}|$ cannot both converge to $0$.
  \end{itemize}
\end{lemma}


\begin{proof}
  We detail the case $\sigma<1$ (the case $\sigma > 1$ is treated similarly).
  Since $(\xi_k)_{k\in\N}$ is a sequence of random variables, the idea of the
  proof is to rely on the equivalence
  \begin{equation}\label{eq:equivalence}
    \limsup_{k\to+\infty}
    \biggr(\prod_{j=0}^k\frac{|\xi_j - \lambda_i|}
    {|\xi_j-\lambda_1|}\biggr)^{1/k} < 1
    \quad\Leftrightarrow\quad
    \limsup_{k\to+\infty} \frac1k \sum_{j=0}^k \log\frac{|\xi_j -
      \lambda_i|}{|\xi_j-\lambda_1|} < 0.
  \end{equation}
  By the law of large numbers, it is sufficient to compute
  $\ds
  \mathbb{E}\biggr[  \log\frac{|\xi - \lambda_i|}{|\xi-\lambda_1|}\biggr],
  $
  where $\xi$ is a random variable uniformly distributed in
  $(\frac{\lambda_1}{\sigma},\frac{\lambda_2}{\sigma})$,
  and show that it is negative for $i\neq1$. To this end, let us define
  $\alpha=\frac{\lambda_1}{\sigma}$ and $\beta=\frac{\lambda_n}{\sigma}$. Then,
  if $\lambda_1<\lambda_i<\alpha$,
  \[
    \forall\ \xi\in[\alpha,\beta],\quad
    \log\frac{|\xi-\lambda_i|}{|\xi-\lambda_1|} < 0
  \]
  and the expectation is strictly negative too. Otherwise,
  $\lambda_i\geq\alpha$ and
  \[
    \mathbb{E}\biggr[\log\frac{|\xi-\lambda_i|}{|\xi-\lambda_1|}\biggr] =
    \frac1{\beta-\alpha}\biggr(\int_\alpha^\beta \log|\xi-\lambda_i|\d \xi
    - \int_\alpha^\beta \log |\xi-\lambda_1|\d \xi\biggr).
  \]
  Since $\sigma<1$, $\lambda_1<\alpha$ and the second integral is
  \[
    \int_\alpha^\beta \log|\xi-\lambda_1|\d \xi =
    (\beta-\lambda_1)\log(\beta-\lambda_1) -
    (\alpha-\lambda_1)\log(\alpha-\lambda_1) - (\beta-\alpha).
  \]
  As $\lambda_i\geq\alpha$, the first integral can be decomposed into, with
  the convention that $0\log 0 = 0$,
  \[
    \begin{split}
      \int_\alpha^\beta \log|\xi-\lambda_i|\d \xi &=
      \int_\alpha^{\lambda_i} \log(\lambda_i-\xi)\d \xi +
      \int_{\lambda_i}^\beta \log(\xi-\lambda_i)\d \xi\\
      &=(\lambda_i-\alpha)\log(\lambda_i-\alpha) + (\alpha-\lambda_i)+
      (\beta-\lambda_i)\log(\beta-\lambda_i) - (\beta-\lambda_i)\\
      &=(\beta-\lambda_i)\log(\beta-\lambda_i) +
      (\lambda_i-\alpha)\log(\lambda_i-\alpha) - (\beta-\alpha)\\
    \end{split}
  \]
  Putting the two terms together yields
  \[
    \begin{split}
      \mathbb{E}\biggr[\log\frac{|\xi-\lambda_i|}{|\xi-\lambda_1|}\biggr]
      &= \frac1{\beta-\alpha}\biggr(
      (\beta-\lambda_i)\log\frac{\beta-\lambda_i}{\beta-\lambda_1} +
      (\lambda_i-\alpha)\log\frac{\lambda_i-\alpha}{\beta-\lambda_1} +
      (\alpha-\lambda_1)\log\frac{\alpha-\lambda_1}{\beta-\lambda_1} \biggr).
    \end{split}
  \]
  Recall that $\lambda_1 < \alpha \leq \lambda_i \leq \lambda_n < \beta$.
  Hence, the three terms in the last sum are all negative: the expectation is
  therefore negative too. By \eqref{eq:equivalence} and
  Lemma~\ref{lem:SPower}, $\frac{|\beta_{i,k}|}{|\beta_{1,k}|}$ converge to
  $0$. Since $r_k$ is normalized for every $k$, we recall that $\sum_{j=1}^n
  |\beta_{j,k}|^2 = 1$: this implies that $|\beta_{i,k}|$ converges to $0$ for
  $i\neq1$ while $|\beta_{1,k}|$ converge to $1$.

  Next, if $\sigma=1$, then $\alpha=\lambda_1$, $\beta=\lambda_n$, and
  the calculation above shows that
  \[
    \begin{split}
      \mathbb{E}\biggr[\log\frac{|\xi-\lambda_i|}{|\xi-\lambda_1|}\biggr]
      &= \frac1{\lambda_n-\lambda_1}\biggr(
      (\lambda_n-\lambda_i)\log\frac{\lambda_n-\lambda_i}{\lambda_n-\lambda_1} +
      (\lambda_i-\lambda_1)\log\frac{\lambda_i-\lambda_1}{\lambda_n-\lambda_1}\biggr).
    \end{split}
  \]
  For $i=n$, the expectation is thus 0 and one cannot conclude on the
  convergence of $|\beta_{n,k}|$ towards 0. The calculations are similar in all
  the other cases.
\end{proof}

The previous Lemma shows that a sufficient condition for the residual to
concentrate on the extremal modes is reached when the $\xi_k$'s are uniformly
distributed inside $(\frac{\lambda_1}{\sigma},\frac{\lambda_n}{\sigma})$. This
result is illustrated in Figure~\ref{fig:test_sigma_random}, where we track the
components of the normalized residual in the eigenbasis for various $\sigma$.
For $\sigma < 1$, the dominating mode is $\beta_{1,k}$ after 100 iterations and
for $\sigma > 1$, the dominating mode is $\beta_{n,k}$. For $\sigma=1$, both
scenarios can happen. Showing that such a condition also holds in the case of
MR/SD seems a more difficult task and numerical experiments show that, when $\xi$
is computed as in the MR or SD algorithms, such a simple convergence behavior does not
seem to hold, see Figure~\ref{fig:test_sigma}.

\begin{figure}[p!]
  \centering
  \includegraphics[width=0.5\linewidth]{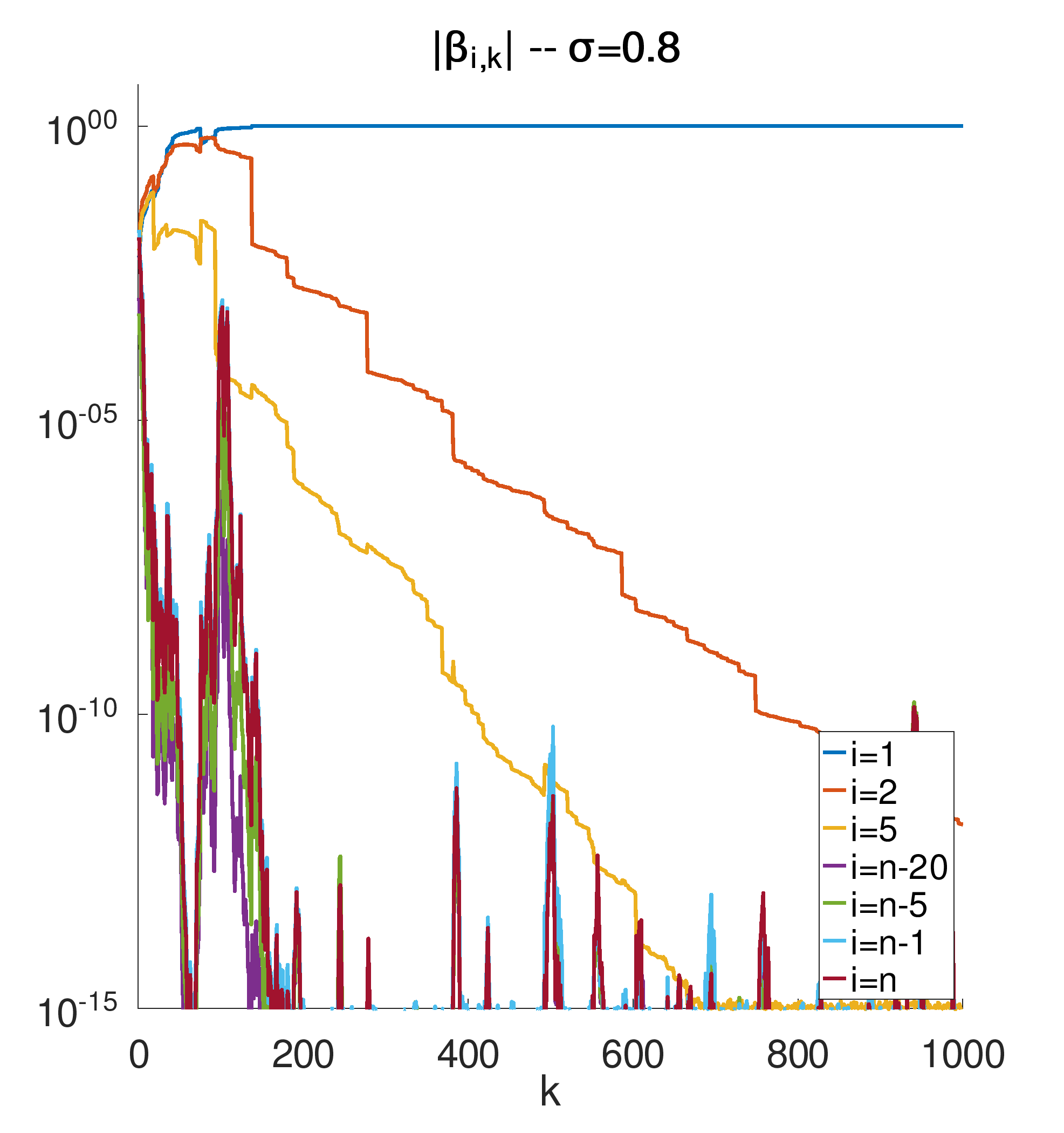}\hfill
  \includegraphics[width=0.5\linewidth]{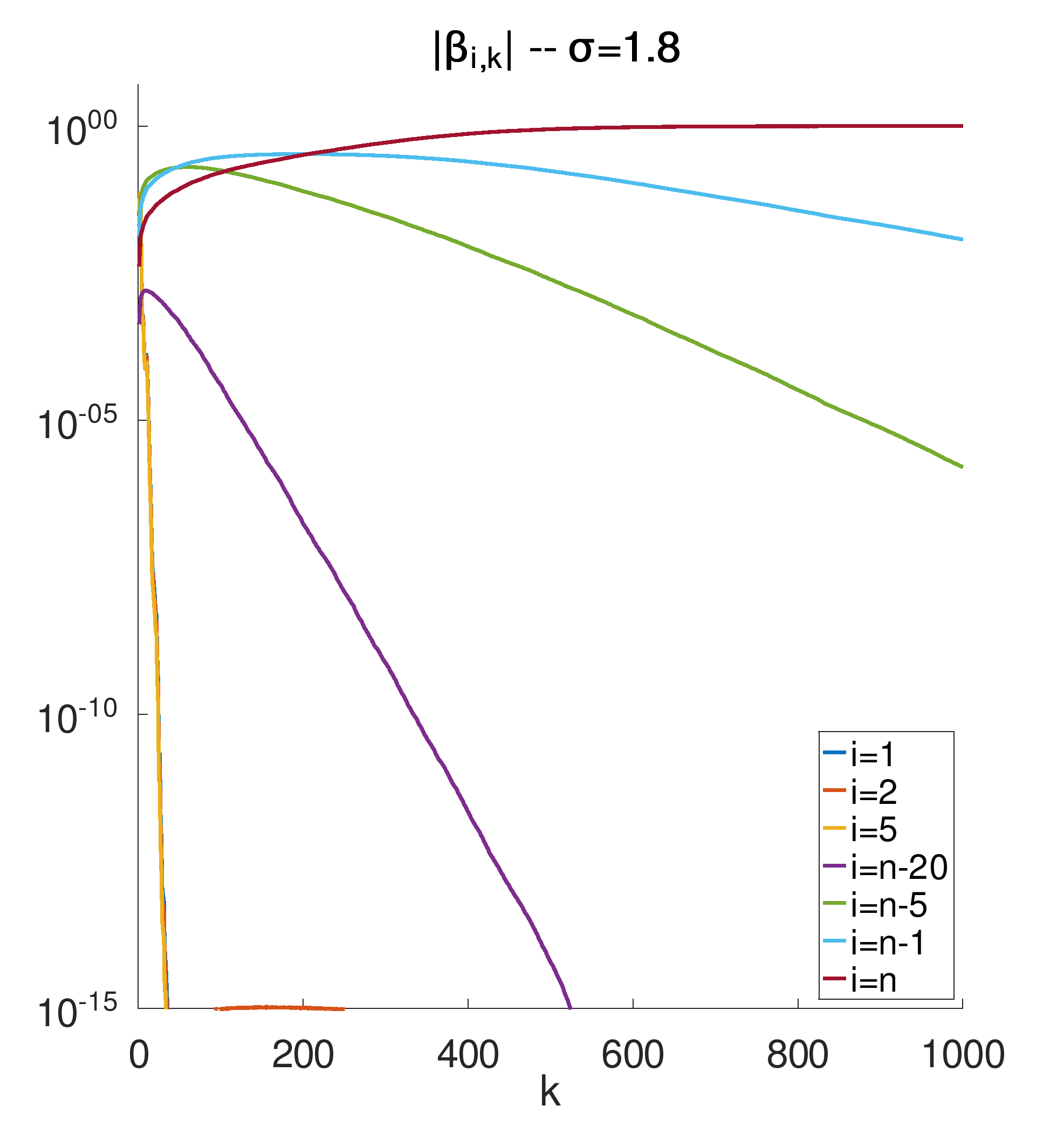}\\
  \includegraphics[width=0.5\linewidth]{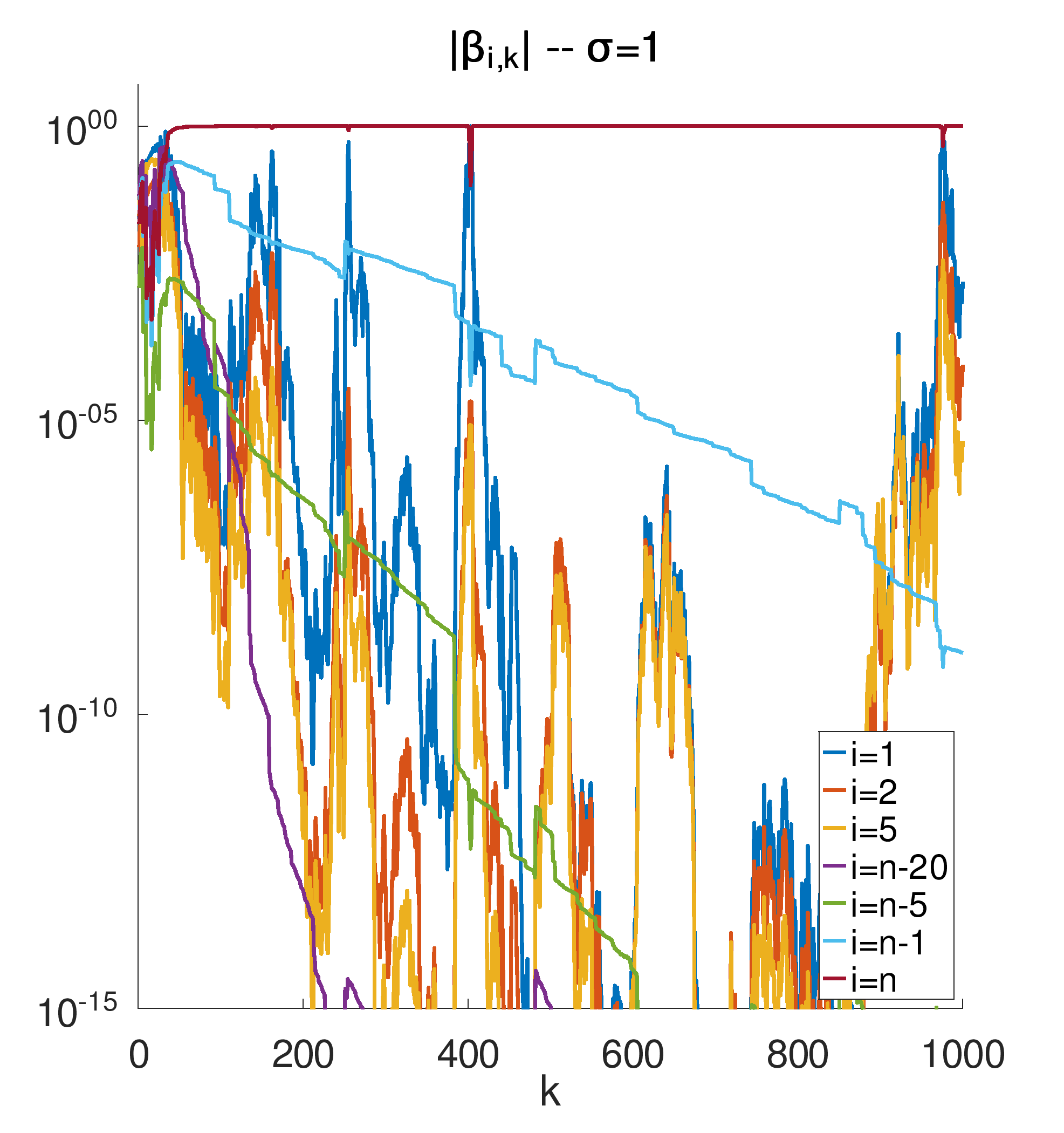}\hfill
  \includegraphics[width=0.5\linewidth]{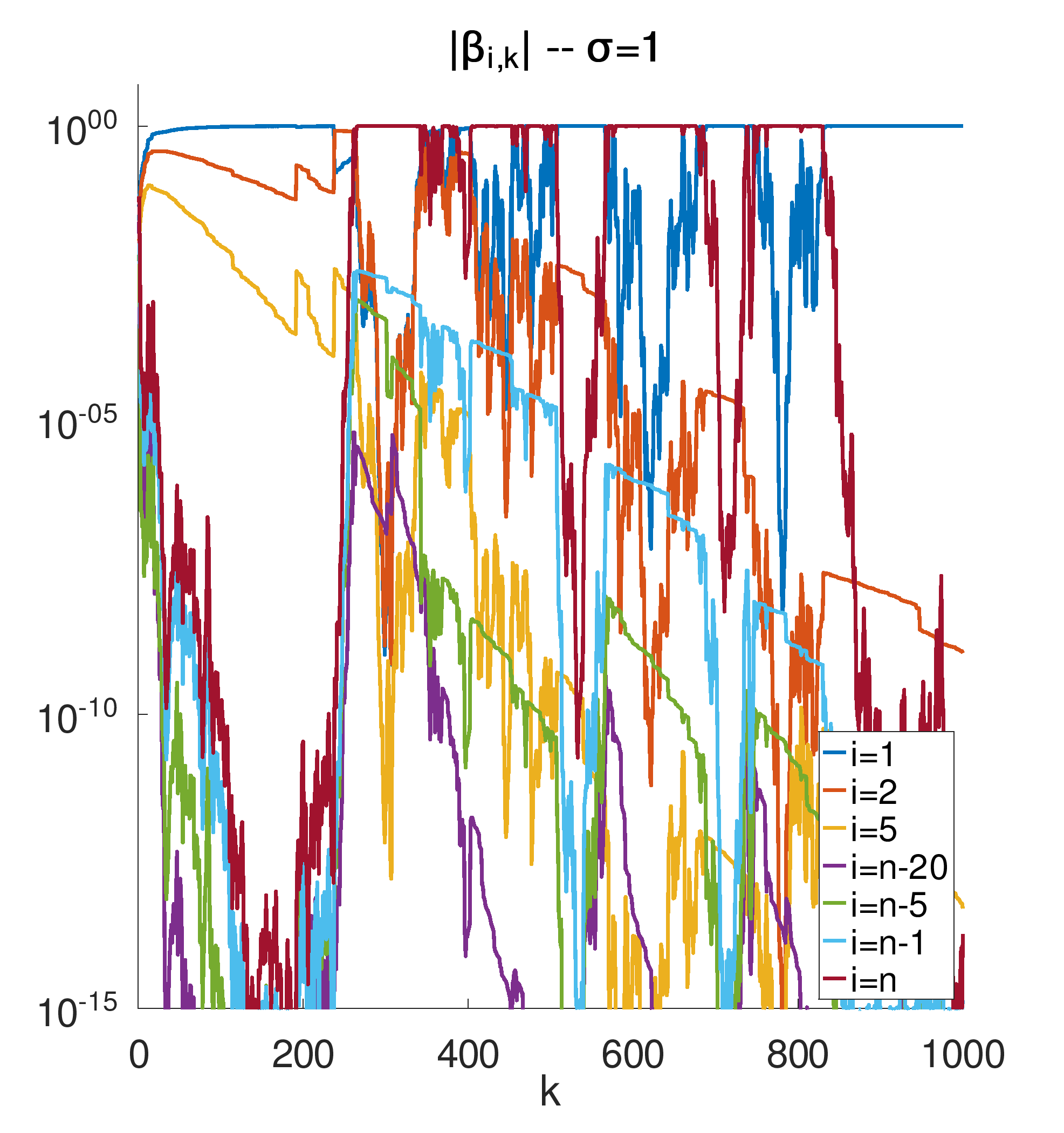}\\
  \caption{Convergence of $|\beta_{i,k}|$ obtained with the shifted and scaled
    power method \eqref{eq:SPower}, for a few selected values of $i$ and
    $\sigma\in\{0.8,1,1.8\}$. The $\xi_k$'s are chosen randomly in
    $(\frac{\lambda_1}{\sigma},\frac{\lambda_n}{\sigma})$. When $\sigma<1$, the
    normalized residuals are supported by the smallest mode while for
    $\sigma>1$, the normalized residuals are supported by the largest mode (all
    the other modes vanish asymptotically). When $\sigma=1$, one cannot conclude
    on the convergence of the residuals towards one of the two extremal modes,
    but all the other intermediate modes vanish asymptotically.}
  \label{fig:test_sigma_random}
\end{figure}

\subsection{A simplified analysis of the relaxed SD iteration}\label{sec:anlys}

In the previous section, we showed that, under appropriate assumptions on the
shift $(\xi_k)_{k\in\N}$, the scaled and shifted power method tends to produce
residuals with support on a few extreme eigenmodes. We also observed this
scenario for the MR and SD algorithms without the zigzag effect in Section~\ref{sec:quad}.
These observations lead us to now study, without loss of generality, the situation where we have
only two nonzero components $\beta_{1,0}$ and $\beta_{n,0}$.
First, according to \eqref{eq:betatil}, we recall that:
\begin{equation}\label{eq:tilbet0}
  \beta_{1,k+1} = \frac{1}{s_k} \left(\frac{\lbet}{\sigma} -
    \lambda_1\right) \beta_{1,k}, \qquad\text{and}\qquad
  \beta_{n,k+1} = \frac{1}{s_k} \left(\frac{\lbet}{\sigma} -
    \lambda_n\right) \beta_{n,k}.
\end{equation}
To simplify notations, let us introduce
\eq{eq:anot}
\gamma \equiv  \frac{\lambda_1+\lambda_n}{2}, \quad h \equiv \frac{\lambda_n -\lambda_1}{2}, \quad
a_k \equiv \frac{\lbet}{\sigma}  - \gamma \quad \Rightarrow \quad
\frac{\lbet}{\sigma}  - \lambda_1 = a_k+h,\quad
\frac{\lbet}{\sigma}  - \lambda_n = a_k-h.
\en
Then, \eqref{eq:tilbet0} becomes:
\begin{equation}\label{eq:tilbet1}
  \beta_{1,k+1} = \frac{1}{s_k} \left(a_k+h \right) \beta_{1,k} , \qquad
  \beta_{n,k+1} = \frac{1}{s_k} \left(a_k-h\right) \beta_{n,k} .
\end{equation}

\begin{lemma} \label{lem:evolvLam}
  Assume that we have a 2-dimensional SD iteration with components only on the
  eigenvectors $u_1$ and $u_n$ governed by equations \eqref{eq:tilbet1}. The
  following then holds:

  \begin{itemize}
    \item[(a)]  The coefficients $\beta_{1,k}$ and $\beta_{n,k}$ are such that
      \eq{eq:beta1betan}
      \beta_{1,k}^2  = \frac{\lambda_n - \lbet}{\lambda_n -
        \lambda_1}, \qquad
      \beta_{n,k}^2  = \frac{\lbet - \lambda_1}{\lambda_n - \lambda_1}.
      \en
    \item[(b)]  The new  Rayleigh quotient $ \llbet$ satisfies:
      \begin{align}
        \llbet
        &=  \lambda_1 + \frac{1}{s_k^2}\left[ (\lambda_n - \lambda_1)
        (a_k - h)^2 \beta_{n,k}^2 \right]
        =  \lambda_n - \frac{1}{s_k^2}\left[(\lambda_n - \lambda_1) (a_k+h)^2
        \beta_{1,k}^2 \right]  .  \label{eq:LamNew1}
      \end{align}
    \item[(c)]  The scaling factor $s_k$ is bounded from above and below as follows:
      \eq{eq:sBound}
      h  \le s_k \le h + \max
      \left\{\left|\frac{\lambda_n}{\sigma} - \gamma\right|,
        \left|\gamma -\frac{\lambda_1}{\sigma}\right|\right\}.  \en
  \end{itemize}
\end{lemma}

\begin{proof}
  \begin{itemize}
    \item[(a)] In all cases (that is for all values of $\sigma\in(0,2)$) we have
      the equations $\lbet = \beta_{1,k}^2 \lambda_1 + \beta_{n,k}^2 \lambda_n$
      and  $1 = \beta_{1,k}^2 + \beta_{n,k}^2 $. From these equations we can
      extract $\beta_{n,k}^2, \beta_{n,k}^2$ -- leading to
      \eqref{eq:beta1betan}.

    \item[(b)]
      The square of the normalization factor $s_k$ is:
      \eq{eq:s2}
      s_k^2 = (a_k+h)^2 \beta_{1,k}^2 + (a_k-h)^2 \beta_{n,k}^2.
      \en
      We consider now the next Rayleigh quotient $\llbet$ to write:
      \begin{align}
        \llbet
        &= \frac{(a_k+h)^2\beta_{1,k}^2 \lambda_1 + (a_k-h)^2\beta_{n,k}^2 \lambda_n}{s_k^2} \nonumber \\
        &= \frac{1}{s_k^2}\Big[(a_k+h)^2\beta_{1,k}^2 \lambda_1 + (a_k-h)^2\beta_{n,k}^2 \lambda_1
        - (a_k-h)^2\beta_{n,k}^2 \lambda_1 + (a_k-h)^2\beta_{n,k}^2 \lambda_n\Big] \nonumber\\
        &=  \lambda_1 + \frac{1}{s_k^2}\Big[(\lambda_n - \lambda_1) (a_k-h)^2
        \beta_{n,k}^2 \Big] .  \label{eq:prfLam1}
      \end{align}
      The proof of the second equality in \eqref{eq:LamNew1} follows similarly.

    \item[(c)] For the upper bound, note first that, since
      $1=\beta_{1,k}^2+\beta_{n,k}^2$, there exists $\theta_1$ such that
      $\beta_{1,k} = \cos\theta_1$ and $\beta_{n,k} = \sin\theta_1$. Then, from
      \eqref{eq:s2} it holds
      \begin{align*}
        s_k^2 &= (a_k^2+2a_kh+h^2)\beta_{1,k}^2 + (a_k^2-2a_kh+h^2)\beta_{n,k}^2\\
        & = a_k^2+h^2 + 2 a_kh (\beta_{1,k}^2-\beta_{n,k}^2) \\
        & = a_k^2+h^2 + 2 a_kh (\cos^2 \theta_1 -\sin^2 \theta_1^2) =
        a_k^2+h^2 + 2 a_kh \cos (2 \theta_1) \\
        & \le (|a_k|+h)^2 = \left(\left|\frac{\lbet}\sigma - \gamma
          \right| + h \right)^2  .
      \end{align*}
      Thus $ s_k \le h+ | \lbet/\sigma - \gamma | $. The maximum is reached when
      $\lbet$ is at one of the extremes of the interval $[\lambda_1,
      \lambda_n]$. Evaluating these extreme values will yield the resulting
      upper bound in \eqref{eq:sBound}. For the lower bound it suffices to take
      a look at Figure~\ref{fig:amplif} to see that the smallest possible value
      for the amplification factor $| \lbet/\sigma - \gamma |$ is reached when
      $\lbet/\sigma$ is in the middle of the interval $[\lambda_1,\lambda_n]$,
      i.e., when $a_k = 0$. For  this case, \eqref{eq:s2} shows that $s_k
      = h$, which is thus the smallest possible value for $s$.
  \end{itemize}
\end{proof}

We are interested in how Rayleigh quotients evolve from one step to the next.
The following result exploits the previous lemma to this end. We are also
interested in the difference between $\lbet /\sigma$ and $\gamma$ as this will
tell how the Rayleigh quotient move relatively to the line  $\sigma \gamma$, in
order to justifiy the observations from Figure~\ref{fig:RQfig}.

\begin{theorem}\label{thm:55}
  Let the assumptions of Lemma~\ref{lem:evolvLam} be satisfied.
  Then the following results hold.

  \begin{itemize}
    \item[(a)]
      From one iteration to the next, the Rayleigh quotient changes as follows:
      \eq{eq:LamIt}
      \llbet = \lbet - 2 \left(\frac{\lbet}{\sigma} - \gamma\right)
      \frac{(\lambda_n - \lbet)(\lbet - \lambda_1)}{s_k^2} .
      \en
      As a result, $\llbet$ will decrease from  $\lbet$ if $\lbet > \sigma \gamma$ and
      will increase when $\lbet < \sigma \gamma$.

    \item[(b)] In addition, \eqref{eq:LamIt} can be rewritten as:
      \eq{eq:LamIt1}
      \left( \frac{\llbet}{\sigma} - \gamma \right) =
      \left(\frac{\lbet}{\sigma} - \gamma \right)
      \left[ 1 - 2
        \frac{(\lambda_n - \lbet)(\lbet - \lambda_1)}{\sigma s_k^2} \right]
      \en
      which means that $\llbet$ will change sides relative to the biased median
      $\sigma \gamma$ whenever
      \eq{eq:chSides0}
      \sigma s_k^2 < 2 (\lambda_n - \lbet)(\lbet - \lambda_1)
      \en
      or equivalently,  when
      \eq{eq:chSides}
      \frac{ (\frac{\lbet}{\sigma}-\lambda_1)^2}{(\lbet-\lambda_1)}
      +  \frac{ (\lambda_n - \frac{\lbet}{\sigma})^2 } {(\lambda_n-\lbet)} < \frac{2(\lambda_n - \lambda_1)}{\sigma} .
      \en
      In particular, when $\sigma =1$ this inequality will \emph{always} be
      satisfied which implies that the Rayleigh quotient will oscillate around
      $\gamma$ in this case.

    \item[(c)]  When $\sigma \ge 2 \lambda_n /(\lambda_1+\lambda_n)$ then the
      sequence of Rayleigh quotients $\lbet$  converges. In any situation when
      $\lbet$  converges its  limit is either $\lambda_1$, or $ \lambda_n$ or
      $\sigma \gamma$.
  \end{itemize}
\end{theorem}

\begin{proof}
  \begin{itemize}
    \item[(a)] We start from the two relations \eqref{eq:LamNew1}.
      Multiplying them respectively by $\beta_{1,k}^2, \beta_{n,k}^2$ and adding
      them yields
      \eq{eq:pfLm1}
      \llbet =
      \lbet + \frac{(\lambda_n - \lambda_1)\beta_{1,k}^2 \beta_{n,k}^2}{s_k^2}
      \Big[(a_k-h)^2-(a_k+h)^2\Big] =
      \lbet - 4 a_k h \frac{(\lambda_n - \lambda_1)\beta_{1,k}^2 \beta_{n,k}^2}{s_k^2} .
      \en
      Replacing $h$ and $a_k$ with their definitions in \eqref{eq:anot} and
      using \eqref{eq:beta1betan}, w e obtain
      \begin{align}
        \llbet &=
        \lbet - 2 \left(\frac{\lbet}{\sigma} - \gamma \right)
        \frac{(\lambda_n - \lambda_1)^2\beta_{1,k}^2 \beta_{n,k}^2}{s_k^2}
         = \lbet - 2 \left(\frac{\lbet}{\sigma} - \gamma\right)
        \frac{(\lambda_n - \lbet)(\lbet - \lambda_1)}{s_k^2} \label{eq:pfLm2}
      \end{align}

    \item[(b)] Let us go back to \eqref{eq:pfLm1} which gives
      \eq{eq:pfLm3}
      \llbet =
      \lbet - 2 a_k \frac{(\lambda_n - \lambda_1)^2 \beta_{1,k}^2 \beta_{n,k}^2}{s_k^2}
      \en
      Dividing by $\sigma$ and subtracting $\gamma$ we finally obtain
      \eq{eq:pfLm4}
      \frac{\llbet}{\sigma} -\gamma =
      \frac{ \lbet }{\sigma} - \gamma
      - 2 a_k \frac{(\lambda_n - \lambda_1)^2 \beta_{1,k}^2 \beta_{n,k}^2}{\sigma
        s_k^2}
      = \left(\frac{ \lbet }{\sigma} - \gamma \right) \left[1
        - 2 \frac{(\lambda_n - \lambda_1)^2 \beta_{1,k}^2 \beta_{n,k}^2}{\sigma
          s_k^2}\right]
      \en
      which can be written as
      \eq{eq:prfChSign}
      a_{k+1} = a_k \left(1 -2 \frac{1}{\sigma t_k} \right)  \quad \mbox{with} \quad
      t \equiv  \frac{s_k^2}{d^2 \beta_{1,k}^2 \beta_{n,k}^2}
      \quad\mbox{and}\quad d \equiv \lambda_n-\lambda_1 .
      \en
      According to \eqref{eq:beta1betan} the term $d^2 \beta_{1,k}^2
      \beta_{n,k}^2$ equals $(\lambda_n -\lbet) (\lbet - \lambda_1)$.  In order
      for $\lbet $ to change sides with respect to $\sigma \gamma$, we need to
      have $1-2/(\sigma t) <0$, i.e., $\sigma t < 2$. This happens when
      \eqref{eq:chSides0} is satisfied.   Also from \eqref{eq:s2} :
      \begin{equation}
        t  = \frac{(a_k+h)^2}{ d^2 \beta_{n,k}^2}  + \frac{(a_k-h)^2}{ d^2 \beta_{1,k}^2}
        =  \frac{1}{d} \left[\frac{ (\frac{\lbet}{\sigma}-\lambda_1)^2}{(\lbet-\lambda_1)}
          +  \frac{ (\lambda_n - \frac{\lbet}{\sigma})^2 } {(\lambda_n-\lbet)} \right] , \label{eq:pfLm5}
      \end{equation}
      from which \eqref{eq:chSides} follows.
      Finally, when $\sigma =1 $ the left hand side of \eqref{eq:chSides}
      evaluates to $ (\lbet - \lambda_1)  + (\lambda_n - \lbet)  =
      \lambda_n-\lambda_1$: \eqref{eq:chSides} is always satisfied when
      $\sigma=1$.

    \item[(c)] Since $\lbet \le \lambda_n$, when $\sigma \ge 2 \lambda_n
      /(\lambda_1+\lambda_n)$   we have  $\lbet / \sigma \le \gamma$. Thus,
      according to \eqref{eq:LamIt}, the Rayleigh quotients $(\lbet)_{k\in\N}$
      form an increasing sequence which is bounded from above by $\lambda_n$ and
      it will therefore converge. We now need to determine their limit, assuming
      just that the sequence converges.

      Assume the limit is neither $\lambda_1$ nor $\lambda_n$. Then,
      equation~\eqref{eq:beta1betan} from Lemma~\ref{lem:evolvLam} shows that (both)
      $\cos^2 \theta_1 = \beta_{n,k}^2$ and $\sin^2 \theta_1 = \beta_{n,k}^2$
      will converge to a positive value.
      Since $s_k $ is bounded from above independently of the iteration $k$, per
      \eqref{eq:sBound}, equation \eqref{eq:LamIt} implies the desired result
      that $ | \lbet/\sigma - \gamma| $ converges to zero.
  \end{itemize}
\end{proof}

This theorem explains partly the behavior of the relaxed SD iteration. If at a
given step the Rayleigh quotient $\lbet$ is above the biased median, i.e. if
$\lbet / \sigma - \gamma >0$, then \eqref{eq:LamIt} indicates that in the next
step the new Rayleigh quotient decreases from its current value. In the opposite
situation, it will increase.

We now explore inequality \eqref{eq:chSides} a little further. It can be shown
that the left-hand side is of the form $f(\lbet)$ where \eq{eq:foft} f(t) =
\frac{\lambda_n-\lambda_1}{\sigma}\left(2 - \frac{1}{\sigma}\right) +
\left(\frac{1}{\sigma} -1\right)^2 \left[\frac{\lambda_1^2}{t - \lambda_1} +
  \frac{\lambda_n^2}{\lambda_n -t} \right] .  \en An illustration is shown in
Figure~\ref{fig:pltthm}. The curve decreases from the left pole to a minimum
achieved at the harmonic mean $ 2 / (1/\lambda_1 + 1/\lambda_n)$ and increases
to infinity at its right pole.  The left pole is extremely sharp: for
$t = \lambda_1+10^{-6}$ the value is still only $13.15923..$ In most of the
interval the inequality will be satisfied which means that the Rayleigh quotient
will \emph{generally} oscillate around the biased median.  If $\lbet$ happens to
fall above the dashed line (say close to the right pole), then according to part
(a) of the theorem, it will decrease in the next step.  It will then either
converge or cross the line and start oscillating again. From our experience,
$\lbet$ converges only when $\sigma$ is large enough (close to 2).

\begin{figure}[ht]
  \centerline{\includegraphics[height=0.25\textheight]{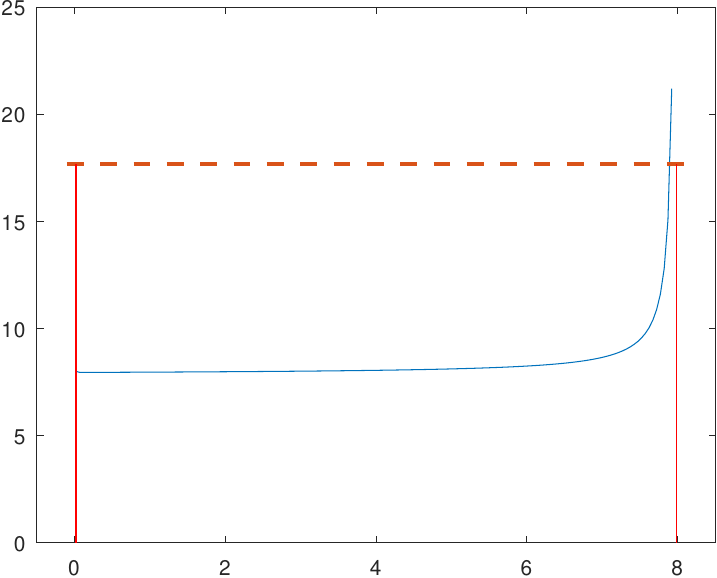}
\includegraphics[height=0.25\textheight]{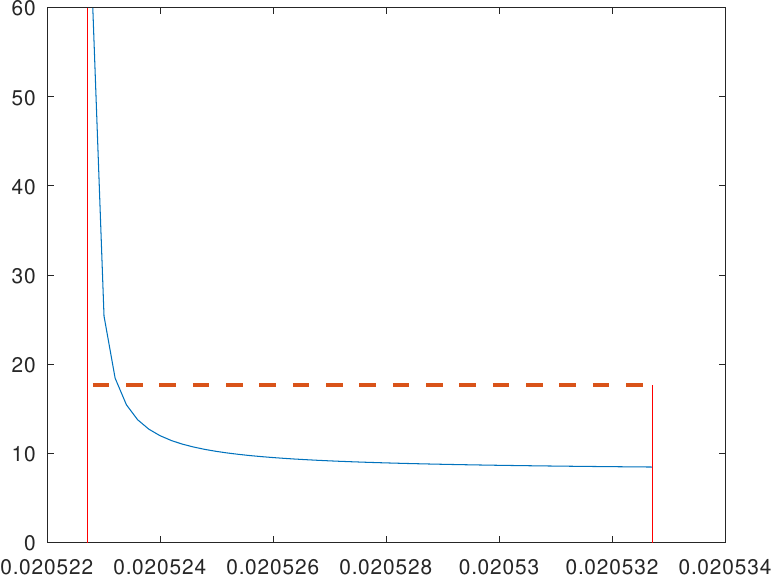}}
  \caption{Illustration of inequality \eqref{eq:chSides} with $\sigma=0.9$.
    The solid line is the left-hand side of \eqref{eq:chSides} as a function of
    $\lbet$.     The dashed one is the constant of the right-hand side. The pole
    on the left side is not visible because the curve is extremely sharp around
    it (right).  -- The right figure shows a zoom in interval
      $[\lambda_1 + 10^{-07},  \ \lambda_1 + 10^{-05}]$}
  \label{fig:pltthm}
\end{figure}

The case $\sigma = 1$ has received extensive coverage in the literature. In this
case, the scalar $t$ defined in \eqref{eq:pfLm5} in the proof of the theorem is
equal to one. Therefore \eqref{eq:prfChSign} becomes $a_{k+1} = - a_k$ or
$\left(\frac{\llbet}{\sigma} - \gamma \right) = - \left( \frac{\lbet}{\sigma} -
  \gamma \right)$.  Thus, $\lbet $ will \emph{equi-oscillate} exactly around
$\sigma \gamma$.
One special case when $\lbet$ converges is when
$\lbet/\sigma $ is always below the median $\gamma$ which takes
place in the particular situation stated in part (c) of the theorem.  It may be
possible to establish that $\lbet $ converges when $\sigma $ is
$<\lambda_n /(\lambda_n+\lambda_1)$ but larger that a certain
$\sigma_*$. Although we cannot prove this, our experience does show that a large
$\sigma $ often leads to convergence of the Rayleigh Quotients. See
Figure~\ref{fig:RQfig} (right) for an example showing a situation when
$\lbet$ converges.

In summary, the relaxed SD iteration leads to a situation where the residual
lies in a one or two-dimensional space. When $\sigma <1$, the Rayleigh quotients
will tend to oscillate around $\sigma \gamma$ and for this reason the gradient
will tend to converge (in direction) to the smallest eigenvector. Although this
analysis was restricted to the SD scheme, the same behavior  is observed in
practice for MR and it is therefore likely that similar results might be proved.
Finally, note that there are other scenarios depending on the value of $\sigma$
but what matters practically for the acceleration schemes developed in this
paper is that the gradient will be close to an eigenvector or a linear
combination of a small number of eigenvectors.

\section*{Acknowledgments and fundings}

All the authors are grateful to Université de Picardie Jules Verne for financial support of Yousef Saad's stay in Amiens and to LAMFA (UPJV and CNRS, UMR 7352) for financial support of Marcos Raydan's stay in Amiens, both in March 2025.
The authors are grateful to the two anonymous reviewers for their insightful comments and suggestions, which greatly enhanced the quality of this manuscript.
The third author is funded by national funds through the FCT – Fundação para a Ciência e a Tecnologia, I.P., under the scope of the projects UID/00297/2025 (https://doi.org/10.54499/UID/00297/2025) and UID/PRR/00297/2025 (https://doi.org/10.54499/UID/PRR/00297/2025) (Center for Mathematics and Applications).
The work of the fourth author  was supported by NSF grant: DMS-2513117.

\section*{Data availability}

All the simulations (except those with the LMSD algorithm which used the Python implementation from \cite{ferrandiLimitedMemoryGradient2025}) have been performed with a GNU Octave implementation of the algorithms. Scripts to reproduce all figures can be downloaded at

\begin{center}
\url{https://plmlab.math.cnrs.fr/gkemlin/lba-paper}
\end{center}

{\bibliographystyle{siam}
\bibliography{MarcosBib}}

\end{document}